\documentclass[10pt]{amsart}

\usepackage{array}

\usepackage{amssymb}
\usepackage{amsthm}

\usepackage{bm}

\usepackage[all,cmtip]{xy}

\usepackage[utf8]{inputenc}
\usepackage{stmaryrd}
\usepackage{url}
\usepackage{mathtools}

\usepackage{ifthen}

\usepackage{enumitem}
\setenumerate[1]{label=(\arabic*), ref=(\arabic*)}
\setenumerate[2]{label=(\alph*), ref=(\alph*)}

\usepackage[draft]{hyperref}

\usepackage{breakurl}

\numberwithin{equation}{section}

\DeclareMathOperator{\id}{id}
\DeclareMathOperator{\GL}{GL}

\DeclareMathOperator{\TotalQuot}{Q}
\newcommand{\Kernel}{\operatorname{Ker}}
\DeclareMathOperator{\Aut}{Aut}

\newtheorem{thm}{Theorem}[section]
\newtheorem{defn}[thm]{Definition}
\newtheorem{prop}[thm]{Proposition}
\newtheorem{lem}[thm]{Lemma}
\newtheorem{cor}[thm]{Corollary}
\newtheorem{ex}[thm]{Example}
\newtheorem{rmk}[thm]{Remark}

\newtheorem{ass}[thm]{Assumption}

\newenvironment{notation}
	{\par\medskip\noindent\textbf{Notation:}\itshape}
	{\normalfont\par\medskip}

\newenvironment{acknowledgements}
	{\par\medskip\noindent\textbf{Acknowledgements.}}
	{\par\medskip}

\usepackage[english]{babel}

\begin{document}

\title{Galois theory of Artinian simple module algebras}
\author{Florian Heiderich}
\address{
Institut de Mathématiques de Jussieu,
Case 247,
Université Pierre et Marie Curie,
4~place Jussieu,
75252 Paris Cedex 05
}
\email{heiderich@math.jussieu.fr}



\thanks{The author thanks La Fondation Sciences Mathématiques de Paris for financial support. He was also partially supported by the Spanish grant MTM2009-07024.}

\begin{abstract}
This main purpose of this article is the unification of the Galois theory of algebraic differential equations by Umemura and the Galois theory of algebraic difference equations by Morikawa-Umemura in a common framework using Artinian simple $D$-module algebras, where $D$ is a bialgebra.
We construct the Galois hull of an extension of Artinian simple $D$-module algebras and define its Galois group, which consists of infinitesimal coordinate transformations fulfilling certain partial differential equations and which we call Umemura functor.
We eliminate the restriction to characteristic $0$ from the above mentioned theories and remove the limitation to field extensions in the theory of Morikawa-Umemura, allowing also direct products of fields, which is essential in the theory of difference equations.
In order to compare our theory with the Picard-Vessiot theory of Artinian simple $D$-module algebras due to Amano and Masuoka, we first slightly generalize the definition and some results about them in order to encompass as well non-inversive difference rings.
Finally, we give equivalent characterizations for smooth Picard-Vessiot extensions, describe their Galois hull and show that their Umemura functor{} becomes isomorphic to the formal scheme associated to the classical Galois group scheme after a finite étale base extension.
\end{abstract}

\maketitle

\section*{Introduction}
The idea behind differential Galois theory, namely to study differential equations using group theoretical methods, dates back to Lie.
Picard and Vessiot realized a Galois theory for linear differential equations, having affine group schemes as Galois groups.
After an attempt by Drach and work by Vessiot, Umemura developed a Galois theory for non-linear algebraic differential equations (cf. \cite{Umemura:1996b}).
To an extension of differential fields $L | K$ of characteristic $0$ that is finitely generated as an extension of fields he associates a new extension $\mathcal{L} | \mathcal{K}$, the Galois hull of $L | K$, and attaches a group functor to it, the so called the infinitesimal Galois group of $L | K$.
The latter is a Lie-Ritt functor, i.e. a group functor of infinitesimal transformations fulfilling certain partial differential equations, which turns out to also be a formal group scheme.
A theory with a similar aim was developed by Malgrange in the framework of differential geometry and applied by Casale (cf. \cite{Malgrange:2001}, \cite{Casale:2007}, \cite{Casale:2008}).

With a delay in time similar theories were realized for difference equations.
Recently, Umemura sketched a difference analogue of his differential Galois theory in \cite{Umemura:2006} and developed it together with Morikawa (cf. \cite{Morikawa:2009}, \cite{MorikawaUmemura:2009}).
To an extension of difference fields $L | K$, i.e. fields equipped with an endomorphism, that is finitely generated as extension of fields they also assign an infinitesimal Galois group, which is a Lie-Ritt functor as in the differential case.
Analogues of the theory of Malgrange for non-linear ($q$-)difference equations have been developed by Casale and Granier (cf. \cite{Casale:2006}, \cite{Granier:2009}).

The purpose of this article is twofold.
On the one hand, we unify the differential Galois theory of Umemura and the difference Galois theory of Morikawa-Umemura by using Artinian simple commutative $D$-module algebras, where $D$ is a bialgebra.
Differential fields and difference fields are special instances of Artinian simple commutative $D$-module algebras for certain choices of $D$.
At the other hand, we generalize the theories of Morikawa and Umemura.
We note first that most of the above mentioned theories have been restricted to characteristic $0$.
Hasse and Schmidt introduced higher and iterative derivations as a replacement for derivations in positive characteristic (cf. \cite{HasseSchmidt:1937}) and, using them, differential Galois theories have been developed by Matzat, Okugawa and van der Put (cf. \cite{Okugawa:1987}, \cite{MatzatVanDerPut:2003}).
We first remove the restriction to characteristic $0$ from the theories of Umemura and Morikawa-Umemura using iterative derivations instead of derivations.
So our theory could be the starting point to tackle the problem raised by Morikawa and Umemura in \cite{MorikawaUmemura:2009} whether the results they obtain there have analogues in positive characteristic.
Second, we eliminate the restriction to fields from the difference Galois theory of Morikawa-Umemura by also allowing direct products of fields equipped with an injective endomorphism (they are Artinian simple commutative $D$-module algebras for a particular choice of the bialgebra $D$).
This approach is more natural, since the total Picard-Vessiot rings of difference equations are in general not difference fields, but only direct products of fields equipped with an endomorphism.
Amano and Masuoka unified the Picard-Vessiot theories for differential and difference extensions by using Artinian simple commutative $D$-module algebras as well (cf. \cite{AmanoMasuoka:2005}).
Though they restrict themselves to Hopf algebras $D$ and therefore non-inversive difference extensions are not within their scope.
Here we do not limit ourselves to Hopf algebras and use a broader class of bialgebras such that our theory encompasses as well non-inversive difference rings.
In order to compare our general Galois group with the classical Galois group scheme of Picard-Vessiot extensions, we first define Picard-Vessiot extensions of Artinian simple commutative $D$-module algebras in a slightly more general context than Amano and Masuoka, i.e. without the assumption that $D$ is a Hopf algebra, and show some of their basic properties.
Finally, we prove that in the case of smooth Picard-Vessiot extensions of Artinian simple commutative $D$-module algebras our general Galois group is closely related to the usual Galois group scheme.

More precisely this article is organized as follows:
In the first section we recall the definition of $D$-module algebras and some of their properties, where $D$ is a bialgebra over a field $C$.
We obtain some results about the extension of $D$-module algebra structures to tensor products and limits, which we need later.
We close the section with the definition and equivalent characterizations of Artinian simple $D$-module algebras, generalizing results of Amano and Masuoka.
The second section is the heart of the article.
We explain in detail the above-mentioned
unification and generalization of the differential Galois theory of Umemura and the difference Galois theory of Morikawa and Umemura.
Let $G$ be a monoid and $D^1$ an irreducible pointed cocommutative Hopf algebra of Birkhoff-Witt type that is a $C G$-module algebra;
we define $D$ to be the smash product $D^1 \# C G$.
If $L | K$ is an extension of Artinian simple commutative $D$-module algebras fulfilling a separability and finiteness condition, then we construct its Galois hull $\mathcal{L} | \mathcal{K}$ and associate to it a group functor $\operatorname{Ume}(L | K)$ on the category of commutative $L$-algebras that we call the Umemura functor{} in honor of its inventor in the case of extensions of differential fields.
We show that the latter is a Lie-Ritt functor and so in particular a formal group scheme (cf. theorem~\ref{thm:InfGal_is_a_Lie-Ritt_functor}).
These constructions unify and generalize those of Umemura and Morikawa and, using the language of $D$-module algebras, arguments become more transparent.
In section~\ref{sec:PV_extensions_of_Artinian_simple_module_algebras} we define Picard-Vessiot extensions of Artinian simple commutative $D$-module algebras and prove some of their properties, generalizing results of Amano and Masuoka, cf. \cite{AmanoMasuoka:2005}, where they are proven in the case where $D$ is a Hopf algebra satisfying some additional conditions.
We close this section with a list of several equivalent characterizations of the property that the principal $D$-module algebra $R$ of a Picard-Vessiot extension of Artinian simple commutative $D$-module algebras $L | K$ is smooth over $K$ (cf. proposition~\ref{prop:smoothness_of_Picard-Vessiot_extensions}).
In the last section we investigate this type of Picard-Vessiot extensions by describing their Galois hull and by comparing the Umemura functor{} of such an extension with its Galois group scheme.
Their Galois hull is of a particularly simple form (cf. lemma~\ref{lem:Lcal_is_generated_by_Kcal_and_gth_of_L_if_extension_is_Picard-Vessiot}) and we show that the Umemura functor{} of this kind of extension becomes isomorphic to the formal group scheme associated to its Galois group scheme after a base extension to a finite étale extension of $L$ (cf. theorem~\ref{thm_comparison_between_InfGal_and_Gal}).
This demonstrates that the Umemura functor in a way generalizes the Galois group scheme of Amano and Masuoka, though at the cost of generality some information is lost (we do not recover the Galois group scheme itself, but only a base extension of the formal group scheme associated to it).
The construction of the Galois hull $\mathcal{L} | \mathcal{K}$, the Umemura functor $\operatorname{Ume}(L | K)$ and the comparison with the classical Galois group scheme $\operatorname{Gal}(L|K)$ are illustrated in two examples of simple Picard-Vessiot extensions of iterative differential fields.

This article develops further results of the authors thesis, where only extensions of $D$-module fields were considered (cf. \cite{Heiderich:2010}).

\begin{notation}
We assume all rings and algebras to be unital and associative, but not necessarily to be commutative.
Homomorphisms of algebras are assumed to respect the units.
We further assume that all coalgebras are counital and coassociative, but not necessarily to be cocommutative.
Homomorphisms of coalgebras are assumed to respect the counits.
An algebra $(A, m, \eta)$ will be abbreviated by $A$ and multiplication and unit will then be denoted by $m_A$ and $\eta_A$, respectively.
Similarly a coalgebra $(D, \Delta, \varepsilon)$ will be abbreviated by $D$ and comultiplication and counit will then be denoted by $\Delta_D$ and $\varepsilon_D$, respectively.
If $D$ is a coalgebra and $d \in D$, then we use the $\Sigma$-notation $\Delta_D(d) = \sum_{(d)} d_{(1)} \otimes d_{(2)}$ (cf. \cite[Section 1.2]{Sweedler:1969} or \cite[1.4.2]{Montgomery:1993}).

If $R$ is a commutative ring,
then we denote by $\TotalQuot(R)$ the total ring of fractions of $R$,
by $\Omega(R)$ the set of minimal prime ideals of $R$,
by $N(R)$ the nilradical of $R$
and by $\pi_{R} \colon R \to R/N(R)$ the canonical projection.
We denote the category of commutative algebras over $R$ by ${\mathsf{CAlg}_{R}}$ and the category of left $R$-modules by ${}_{R}\mkern-1mu\mathcal{M}$; furthermore $\mathsf{Grp}$ denotes the category of groups.

If $\mathsf{C}$ is a category and $A$ and $B$ are objects in $\mathsf{C}$, then we denote the class of morphisms from $A$ to $B$ in $\mathsf{C}$ by $\mathsf{C}(A,B)$.
We denote the opposite category of $\mathsf{C}$ by ${\mathsf{C}}^{op}$.

The category of sets is denoted by $\mathsf{Set}$.
If $A$ and $B$ are sets and $a \in A$, then we denote by $\operatorname{ev}_{a} \colon \mathsf{Set}(A,B) \to B$ the evaluation map, i.e. $\operatorname{ev}_{a}(f) = f(a)$ for all $f \in \mathsf{Set}(A,B)$.
If $A' \subseteq A$ is a subset, then $\operatorname{res}^{A}_{A'} \colon \mathsf{Set}(A,B) \to \mathsf{Set}(A',B)$ denotes the restriction map.
We denote by $M_n(A)$ the set of $n \times n$-matrices with coefficients in $A$
and for elements $a,b \in A$ we denote by $\delta_{a,b}$ the Kronecker delta, i.e. $\delta_{a,a} = 1$ and $\delta_{a,b}=0$ if $a \neq b$.

If $f \colon A \to B$ is a homomorphism of rings and ${\bm{w}} = (w_{1}, \dots, w_{n})$ are algebraically independent elements over $B$, then $f\llbracket {\bm{w}} \rrbracket \colon A\llbracket {\bm{w}} \rrbracket \to B\llbracket {\bm{w}} \rrbracket$ denotes the homomorphism defined by $f\llbracket {\bm{w}} \rrbracket(\sum_{{\bm{k}} \in {\mathbb N}^{n}} a_{{\bm{k}}} {\bm{w}}^{{\bm{k}}}) = \sum_{{\bm{k}} \in {\mathbb N}^{n}} f(a_{{\bm{k}}}) {\bm{w}}^{{\bm{k}}}$.
\end{notation}

\section{Module algebras}
\begin{notation}
	Let $C$ be a commutative ring.
\end{notation}

\subsection{Definitions and basic properties}

We recall that for $C$-modules $A, B$ and $D$ there is an isomorphism of $C$-modules
\begin{equation}
	{}_{C}\mkern-1mu\mathcal{M}(D \otimes_C A, B) \to {}_{C}\mkern-1mu\mathcal{M}(A, {}_{C}\mkern-1mu\mathcal{M}(D,B)), \quad \Psi \mapsto (a \mapsto (d \mapsto \Psi(d \otimes a))).
	\label{eqn:fundamental_isomorphism}
\end{equation}

\begin{lem}
	If $(D, \Delta_D, \varepsilon_D)$ is a $C$-coalgebra and $(B, m_B, \eta_B)$ is a $C$-algebra, then the $C$-module ${}_{C}\mkern-1mu\mathcal{M}(D,B)$ becomes a $C$-algebra with respect to the \emph{convolution product}\index{convolution product}, defined by
	\[f \cdot g \coloneqq m_B \circ (f \otimes g) \circ \Delta_D\]
	for $f,g \in {}_{C}\mkern-1mu\mathcal{M}(D,B)$, and unit element given by the composition
	\[ D \xrightarrow{\varepsilon_D} C \xrightarrow{\eta_B} B. \]
	Furthermore, $D$ is cocommutative if and only if ${}_{C}\mkern-1mu\mathcal{M}(D,B)$ is commutative for every commutative $C$-algebra $B$.
	\label{lem_algebra_structure_on_Mod_C_D_B}
\end{lem}
\begin{proof}
	See for example \cite[1.3]{BrzezinskiWisbauer:2003}
\end{proof}

\begin{prop}
	Let $D$ be a $C$-coalgebra and let $A$ and $B$ be $C$-algebras.
	If $\Psi$ is an element of ${}_{C}\mkern-1mu\mathcal{M}(D \otimes_C A, B)$ and $ \rho_{}  \in {}_{C}\mkern-1mu\mathcal{M}(A, {}_{C}\mkern-1mu\mathcal{M}(D,B))$ is the image of $\Psi$ under the isomorphism (\ref{eqn:fundamental_isomorphism}),
	then the following are equivalent:
	\begin{enumerate}
		\item $ \rho_{} $ is a homomorphism of $C$-algebras, \label{eqn_prop_characterization_of_measuring_gth_is_algebra_homomorphism}
		\item for all $d \in D$ and all $a,b \in A$
		\begin{enumerate}
			\item $\Psi(d \otimes ab) = \sum_{(d)} \Psi(d_{(1)} \otimes a) \Psi(d_{(2)} \otimes b)$ if $\Delta_D(d) = \sum_{(d)} d_{(1)} \otimes d_{(2)}$ and
			\item $\Psi(d \otimes 1_A) = \varepsilon_D(d) 1_B$,
		\end{enumerate}
		hold and
		\label{eqn_prop_characterization_of_measuring_concrete}
		\item the diagrams
		\begin{align*}
		\xymatrixcolsep{3.8em}
		\xymatrixrowsep{3em}
		\xymatrix{
			D \otimes_C A \otimes_C A \ar@{->}[r]^{\id_D \otimes m_A} \ar@{->}[d]^{\Delta_D \otimes \id_A \otimes \id_A} & D \otimes_C A \ar@{->}[r]^{\Psi} & B \\
		D \otimes_C D \otimes_C A \otimes_C A \ar@{->}[r]^{\id_D \otimes \tau \otimes \id_A} & D \otimes_C A \otimes_C D \otimes_C A \ar@{->}[r]^-{\Psi \otimes \Psi} & B \ar[u]^{m_B}\\
		}
		\end{align*}
		and
		\begin{align*}
		\xymatrixcolsep{3.8em}
		\xymatrixrowsep{3em}
		\xymatrix{
			D \otimes_C C \ar@{->}[r]^{\varepsilon_D \otimes \eta_B} \ar@{->}[d]^{\id_D \otimes \eta_A} & C \otimes_C B \ar@{->}[d]^{\sim} \\
			D \otimes_C A \ar@{->}[r]^{\Psi} & B \\
		}
		\end{align*}
		commute. \label{eqn_prop_characterization_of_measuring_diagrams}
	\end{enumerate}
	\label{prop_characterization_of_measuring}
\end{prop}
\begin{proof}
	The equivalence between \ref{eqn_prop_characterization_of_measuring_gth_is_algebra_homomorphism} and \ref{eqn_prop_characterization_of_measuring_concrete} can be proven as in \cite[Proposition 7.0.1]{Sweedler:1969} and the one between \ref{eqn_prop_characterization_of_measuring_concrete} and \ref{eqn_prop_characterization_of_measuring_diagrams} is clear.
\end{proof}

\begin{defn}
\label{defn:measuring}
	Let $D$ be a $C$-coalgebra and $A$ and $B$ be $C$-algebras.
	If $\Psi \in {}_{C}\mkern-1mu\mathcal{M}(D \otimes_C A, B)$, then we say that $\Psi$ \emph{measures}\index{measuring} $A$ to $B$ if the equivalent conditions in proposition \ref{prop_characterization_of_measuring} are satisfied.

	If $A_{1}, A_{2}, B_{1}$ and $B_{2}$ are $C$-algebras, $\Psi_{1} \colon D \otimes_{C} A_{1} \to B_{1}$ measures $A_{1}$ to $B_{1}$ and $\Psi_{2} \colon D \otimes_{C} A_{2} \to B_{2}$ measures $A_{2}$ to $B_{2}$, then we say that homomorphisms $\varphi_{A} \colon A_{1} \to A_{2}$ and $\varphi_{B} \colon B_{1} \to B_{2}$ of $C$-algebras are \emph{compatible with the measurings} if the diagram
	\begin{align*}
		\xymatrix{
			D \otimes_{C} A_{1} \ar[d]^{ \id_{D} \otimes \varphi_{A}} \ar[r]^-{\Psi_{1}} & B_{1} \ar[d]^{\varphi_{B}} \\
			D \otimes_{C} A_{2}  \ar[r]^-{\Psi_{2}} & B_{2}
		}
	\end{align*}
commutes.
\end{defn}

The following lemmata are clear from the definition.

\begin{lem}
	\label{lem:equivalent_defining_conditions_for_module_algebras}
	Let $D$ be a $C$-bialgebra and $A$ be a $C$-algebra.
		If $\Psi \in {}_{C}\mkern-1mu\mathcal{M}(D \otimes_{C} A, A)$ and $ \rho_{}  \colon A \to {}_{C}\mkern-1mu\mathcal{M}(D,A)$ is the homomorphism associated to $\Psi$ via \eqref{eqn:fundamental_isomorphism}, then $\Psi$ makes $A$ into a $D$-module if and only if the diagrams
		\begin{align*}
			\xymatrixcolsep{5em}
			\xymatrix{A \ar[r]^-{ \rho_{} } \ar[d]^{ \rho_{} } & {}_{C}\mkern-1mu\mathcal{M}(D,A) \ar[d]^{{}_{C}\mkern-1mu\mathcal{M}(D, \rho_{} )} \\
			{}_{C}\mkern-1mu\mathcal{M}(D,A) \ar[r]^-{{}_{C}\mkern-1mu\mathcal{M}(m_D,A)} & {}_{C}\mkern-1mu\mathcal{M}(D \otimes_{C} D, A) \cong {}_{C}\mkern-1mu\mathcal{M}(D, {}_{C}\mkern-1mu\mathcal{M}(D,A))
			}
		\end{align*}
		and
		\begin{align*}
			\xymatrixcolsep{5em}
			\xymatrix{A \ar[r]^-{ \rho_{} } \ar[rd]^{\id} & {}_{C}\mkern-1mu\mathcal{M}(D,A) \ar[d]^{\operatorname{ev}_{1_D}} \\
			 & A
			}
		\end{align*}
		commute.
\end{lem}

\begin{lem}
	\label{lem:equivalent_defining_conditions_morphisms_compatible_with_measurings}
	Let $D$ be a $C$-coalgebra and $A_1, A_2, B_1$ and $B_2$ be $C$-algebras.
	If $\Psi_{1} \in {}_{C}\mkern-1mu\mathcal{M}(D \otimes_C A_1, B_1)$ measures $A_1$ to $B_1$ and $\Psi_{2} \in {}_{C}\mkern-1mu\mathcal{M}(D \otimes_C A_2, B_2)$ measures $A_2$ to $B_2$ and $ \rho_{1} $ and $ \rho_{2} $ are the associated homomorphisms of $C$-algebras, then homomorphisms of $C$-algebras $\varphi_A \colon A_1 \to A_2$ and $\varphi_B \colon B_1 \to B_2$ are compatible with the measurings if and only if the diagram
	\begin{align*}
		\xymatrix{
			A_1 \ar[r]^-{ \rho_{1} } \ar[d]^{\varphi_A} & {}_{C}\mkern-1mu\mathcal{M}(\mkern-1mu D, \mkern-1mu B_1) \ar[d]^{{}_{C}\mkern-1mu\mathcal{M}(\mkern-1mu D,\varphi_B)} \\
			A_2 \ar[r]^-{ \rho_{2} } & {}_{C}\mkern-1mu\mathcal{M}(\mkern-1mu D, \mkern-1mu B_2)
		}
	\end{align*}
	commutes.
\end{lem}

\begin{defn}
\label{defn:module_algebra}
	Let $A$ be a $C$-algebra, $D$ be a $C$-bialgebra and let $\Psi \in {}_{C}\mkern-1mu\mathcal{M}(D \otimes_C A,A)$ measure $A$ to $A$.
	We say that $\Psi$ is a \emph{$D$-module algebra structure} on $A$ if the equivalent conditions in lemma~\ref{lem:equivalent_defining_conditions_for_module_algebras} are fulfilled.
	The pair $(A,\Psi)$ is then called \emph{$D$-module algebra}.

	A \emph{commutative $D$-module algebra} is a $D$-module algebra $(A,\Psi)$ such that the $C$-algebra $A$ is commutative.

	A \emph{homomorphism of $D$-module algebras} from $(A_1,\Psi_1)$ to $(A_2,\Psi_2)$ is a homomorphism of $C$-algebras $\varphi \colon A_1 \to A_2$ that fulfills the equivalent conditions of lemma~\ref{lem:equivalent_defining_conditions_morphisms_compatible_with_measurings} (with $B_1 = A_1$ and $B_2=A_2$).
\end{defn}

\begin{notation}
	If $\Psi \colon D \otimes_C A \to B$ is a homomorphism of $C$-modules, then we denote by $ \rho_{}  \colon A \to {}_{C}\mkern-1mu\mathcal{M}(D,B)$ the homomorphism corresponding to $\Psi$ under the isomorphism \eqref{eqn:fundamental_isomorphism} and vice versa.
	If $d \in D$ and $a \in A$, then we denote $\Psi(d \otimes a)$ also by $d.a$ if there is no danger of confusion.
\end{notation}

\begin{defn}
\label{def:constants}
	For a $C$-coalgebra $D$, a $C$-module $V$ and $\Psi \in {}_{C}\mkern-1mu\mathcal{M}(D \otimes_C V, V)$ we define the \emph{constants}\index{constants} of $V$ with respect to $\Psi$ as
	\[ V^{\Psi} \coloneqq \{v \in V \mid \Psi(d \otimes v) = \varepsilon_D(d)v \quad \text{ for all } d \in D \}. \]
	If $ \rho_{}  \in {}_{C}\mkern-1mu\mathcal{M}(V, {}_{C}\mkern-1mu\mathcal{M}(D,V))$ is the element corresponding to $\Psi$ under the isomorphism \eqref{eqn:fundamental_isomorphism}, then we denote $V^{\Psi}$ also by $V^{ \rho_{} }$.
	Sometimes we will also denote it by $V^D$.
\end{defn}

\begin{ex}
\label{ex:module_algebra_structures}
Let $A$ be a commutative $C$-algebra.
\begin{enumerate}
	\label{ex_D-module_algebra_structures}
	\item
	If $D_{der} \coloneqq C[{\mathbb G}_a]$ is the Hopf algebra on the coordinate ring of the additive group scheme ${\mathbb G}_a$ over $C$, then $D_{der}$-module algebra structures on $A$ are in 1-1 correspondence with $C$-derivations on $A$.

	\item
	If $D_{aut} \coloneqq C[{\mathbb G}_m]$ is the Hopf algebra on the coordinate ring of the multiplicative group scheme ${\mathbb G}_m$ over $C$, then $D_{aut}$-module algebra structures on $A$ are in 1-1 correspondence with automorphisms of the $C$-algebra $A$.

	\item
	\label{itm:ex:module_algebra_structures:endomorphisms}
	If $D_{end} \coloneqq C[t]$ is the polynomial algebra over $C$ with coalgebra structure defined by $\Delta(t^n) = t^n \otimes t^n$ and $\varepsilon(t^n) = 1$ for all $n \in {\mathbb N}$, then $D_{end}$-module algebra structures on $A$ are in 1-1 correspondence with endomorphisms of the $C$-algebra $A$.

	\item
	\label{itm:ex:module_algebra_structures:group_like_bialgebras_over_a_monoid}
	For a monoid $G$ we define a $C$-bialgebra $C G$ by taking the monoid algebra $CG$ as the underlying algebra with the coalgebra structure defined by $\Delta(g) = g \otimes g$ and $\varepsilon(g) = 1$ for all $g \in G$.
	Operations of the monoid $G$ as endomorphisms on the $C$-algebra $A$ are in 1-1 correspondence with $C G$-module algebra structure on $A$.

	If $G = ({\mathbb N}, +)$ is the monoid of natural numbers, then the bialgebra $C{\mathbb N}$ is isomorphic to $D_{end}$.
	In the case where $G = ({\mathbb Z}, +)$ is the group of integers, the bialgebra $C {\mathbb Z}$ is isomorphic to $D_{aut}$.

	\item
	Let $D_{I \mkern-1.9mu D^{}}$ be the free $C$-module with basis consisting of $\theta_{}^{(k)}$ for all $k \in {\mathbb N}$ with $C$-algebra structure defined by
	\[ \theta_{}^{(i)} \theta_{}^{(j)} \coloneqq \binom{i+j}{i} \theta_{}^{(i+j)} \quad \text{and} \quad 1 \coloneqq \theta_{}^{(0)} \]
	for all $i,j \in {\mathbb N}$ and $C$-coalgebra structure defined by
	\[ \Delta(\theta_{}^{(i)}) \coloneqq \sum_{i_1 + i_2 = i} \theta_{}^{(i_1)} \otimes \theta_{}^{(i_2)} \quad \text{and} \quad \varepsilon(\theta_{}^{(i)}) = \delta_{i,0} \]
	for all $i \in {\mathbb N}$.
	Then $D_{I \mkern-1.9mu D^{}}$ becomes a Hopf algebra with antipode defined by $S(\theta_{}^{(i)}) = (-1)^{i} \theta_{}^{(i)}$ and $D_{I \mkern-1.9mu D^{}}$-module algebra structures on $A$ are in 1-1 correspondence with iterative derivations on $A$ over $C$ (cf. \cite{HasseSchmidt:1937} or \cite[\S 27]{Matsumura:1989}).
	We note that the bialgebras $D_{der}$ and $D_{I \mkern-1.9mu D^{}}$ are isomorphic if ${\mathbb Q} \subseteq C$.
	Therefore derivations and iterative derivations are equivalent on ${\mathbb Q}$-algebras.

	We note that there are differences in the definition of higher derivations.
	In \cite{Matsumura:1989}, \cite{Heiderich:2007} and \cite{Heiderich:2010} a condition on $\theta_{}^{(0)}$ is posed, which is not present in \cite{Sweedler:1969}.
	Using the definition of Sweedler, higher derivation from $A$ to another commutative $C$-algebra $B$ are in 1-1 correspondence with $D_{I \mkern-1.9mu D^{}}$-measurings from $A$ to $B$.
	When $B$ is a commutative $A$-algebra via $A \overset{g}{\to} B$, then higher derivations (of length $\infty$) in the sense of Matsumura (cf. \cite[\S 27]{Matsumura:1989}) are in 1-1 correspondence with $D_{I \mkern-1.9mu D^{}}$-measurings from $A$ to $B$ such that the associated homomorphism of $C$-algebras $ \rho_{}  \colon A \to {}_{C}\mkern-1mu\mathcal{M}(D,B)$ fulfills $\operatorname{ev}_{1_D} \circ  \rho_{}  = g$.

	\item	If $m \in {\mathbb N}$ and $D_{H \mkern-1.9mu D_{(m)}^{}}$ is the free $C$-module with basis consisting of $\theta_{}^{(0)}, \dots, \theta_{}^{(m)}$
	and $C$-coalgebra structure defined by
	\[ \Delta(\theta_{}^{(i)}) \coloneqq \sum_{i_1 + i_2 = i} \theta_{}^{(i_1)} \otimes \theta_{}^{(i_2)} \quad \text{and} \quad \varepsilon(\theta_{}^{(i)}) = \delta_{i,0} \]
	for all $i \in \{0, \dots, m\}$,
	then $D_{H \mkern-1.9mu D_{(m)}^{}}$-measurings from $A$ to another commutative $C$-algebra $B$ are in 1-1 correspondence with higher derivations of length $m$ from $A$ to $B$ over $C$ in the sense of Sweedler (cf. \cite{Sweedler:1969}).
	When $B$ is a commutative $A$-algebra via $A \overset{g}{\to} B$, then higher derivations of length $m$ in the sense of Matsumura (cf. \cite[\S 27]{Matsumura:1989}) are in 1-1 correspondence with $D_{H \mkern-1.9mu D_{(m)}^{}}$-measurings from $A$ to $B$ such that the associated homomorphism of $C$-algebras $ \rho_{}  \colon A \to {}_{C}\mkern-1mu\mathcal{M}(D,B)$ fulfills $\operatorname{ev}_{1_D} \circ  \rho_{}  = g$.
	We note that $D_{I \mkern-1.9mu D^{}}$ is isomorphic to $\varinjlim_{m \in {\mathbb N}} D_{H \mkern-1.9mu D_{(m)}^{}}$ as $C$-coalgebra.

	\item For every $C$-bialgebra $D$, there is a $D$-module algebra structure
	\[ {\Psi_{0}} \colon D \otimes_C A \to A \]
	on $A$ defined as the composition
	\begin{align*}
	\xymatrixcolsep{3em}
	\xymatrix{ D \otimes_{C} A \ar[r]^{\varepsilon_{D} \otimes \id_{A}} & C \otimes_{C} A \ar[r]^-{\sim} & A.}
	\end{align*}
	We call ${\Psi_{0}}$ the \emph{trivial $D$-module algebra structure} on $A$.
\end{enumerate}
\end{ex}

\begin{ex}
\label{ex_duals_and_GTHs}
Let $A$ be a commutative $C$-algebra.
For the $C$-bialgebras $D$ in example~\ref{ex_D-module_algebra_structures} the $C$-algebras ${}_{C}\mkern-1mu\mathcal{M}(D,A)$ and the homomorphisms $ \rho_{}  \colon A \to {}_{C}\mkern-1mu\mathcal{M}(D,A)$ associated to $D$-module algebra structures $\Psi \in {}_{C}\mkern-1mu\mathcal{M}(D \otimes_C A, A)$ on $A$ are well known:
\begin{enumerate}
	\item If ${\mathbb Q} \subseteq A$, then ${}_{C}\mkern-1mu\mathcal{M}(D_{der},A)$ is isomorphic to the formal power series ring $A\llbracket w \rrbracket$.
	If $\partial$ is a $C$-derivation on $A$, $\Psi$ the corresponding $D_{der}$-module algebra structure on A  and $ \rho_{} $ is the homomorphism of $C$-algebras corresponding to $\Psi$ via \eqref{eqn:fundamental_isomorphism}, then the composition $A \overset{ \rho_{} }{\to} {}_{C}\mkern-1mu\mathcal{M}(D_{der},A) \overset{\sim}{\rightarrow} A\llbracket w \rrbracket$ is given by $a \mapsto \sum_{k \in {\mathbb N}} \frac{\partial^k(a)}{k!} w^k$.
	This homomorphism appears in \cite{Umemura:1996b} and is called \emph{universal Taylor homomorphism} there.

	In contrast, ${}_{C}\mkern-1mu\mathcal{M}(D_{der},A)$ is not reduced in positive characteristic.
It is isomorphic to the ring of so called \emph{Hurwirtz series} as defined by Keigher (cf. \cite{Keigher:1997}).

	\item The $C$-algebra ${}_{C}\mkern-1mu\mathcal{M}(D_{aut},A)$ is isomorphic to $A^{\mathbb Z}$, the ring of maps from the integers to $A$ with pointwise addition and multiplication.
	If $\sigma$ is an automorphism of the $C$-algebra $A$, $\Psi$ the corresponding $D_{aut}$-module algebra structure and $ \rho_{} $ the associated homomorphism via \eqref{eqn:fundamental_isomorphism}, then the composition $A \overset{ \rho_{} }{\to} {}_{C}\mkern-1mu\mathcal{M}(D_{aut},A) \overset{\sim}{\rightarrow} A^{\mathbb Z}$ is given by $a \mapsto (k \mapsto \sigma^{k}(a))$.
	\item The $C$-algebra ${}_{C}\mkern-1mu\mathcal{M}(D_{end},A)$ is isomorphic to $A^{\mathbb N}$, the ring of maps from the natural numbers to $A$ with pointwise addition and multiplication.
	If $\sigma$ is an endomorphism of the $C$-algebra $A$, $\Psi$ the corresponding $D_{end}$-module algebra structure and $ \rho_{} $ the associated homomorphism via \eqref{eqn:fundamental_isomorphism}, then the composition $A \overset{ \rho_{} }{\to} {}_{C}\mkern-1mu\mathcal{M}(D_{end},A) \overset{\sim}{\rightarrow} A^{\mathbb N}$ is given by $a \mapsto (k \mapsto \sigma^k(a))$.
	This homomorphism appears in \cite{Morikawa:2009} and is called \emph{universal Euler homomorphism} there.

	\item The $C$-algebra ${}_{C}\mkern-1mu\mathcal{M}(C G, A)$ is isomorphic to $A^G$, the ring of maps from $G$ to $A$ with pointwise addition and multiplication.
	If $\Psi \colon C G \otimes_C A \to A$ is the $C G$-module algebra structure on $A$ corresponding to an operation of $G$ on $A$ and $ \rho_{}  \colon A \to {}_{C}\mkern-1mu\mathcal{M}(C G, A)$ the corresponding homomorphism, then the composition $A \to {}_{C}\mkern-1mu\mathcal{M}(C G, A) \to A^G$ is given by $a \mapsto (g \mapsto g.a)$.

	\item The $C$-algebra ${}_{C}\mkern-1mu\mathcal{M}(D_{I \mkern-1.9mu D^{}},A)$ is isomorphic to $A\llbracket w \rrbracket$ and if $(\theta_{}^{(k)})_{k \in {\mathbb N}}$ is an iterative derivation on $A$ over $C$, $\Psi$ the induced $D_{I \mkern-1.9mu D^{}}$-module algebra structure and $ \rho_{} $ the associated homomorphism of $C$-algebras, then the composition $A \overset{ \rho_{} }{\to} {}_{C}\mkern-1mu\mathcal{M}(D_{I \mkern-1.9mu D^{}},A) \overset{\sim}{\rightarrow} A\llbracket w \rrbracket$ is given by $a \mapsto \sum_{k \in {\mathbb N}} \theta_{}^{(k)}(a) w^k$.
	We will identify an iterative derivation with this homomorphism as in \cite{Heiderich:2007}.
	\item If $B$ is a commutative $A$-algebra via $A \overset{g}{\to} B$, the $C$-algebra ${}_{C}\mkern-1mu\mathcal{M}(D_{H \mkern-1.9mu D_{(m)}^{}}, B)$ is isomorphic to $B\llbracket w \rrbracket/(w^{m+1})$ and if $(\theta_{}^{(k)})_{k=0, \dots, m}$ is a higher derivation of length $m$ from $A$ to $B$ over $C$, $\Psi$ is the induced $D_{H \mkern-1.9mu D_{(m)}^{}}$-measuring and $ \rho_{} $ the associated homomorphism, then the composition $ \rho_{}  \colon A \to {}_{C}\mkern-1mu\mathcal{M}(D_{H \mkern-1.9mu D_{(m)}^{}},B) \overset{\sim}{\rightarrow} B\llbracket w \rrbracket/(w^{m+1})$ is given by $a \mapsto \sum_{k=0}^m \theta_{}^{(k)}(a) w^k + (w^{m+1})$.
	\item \label{ex_duals_and_GTHs_trivial_MAS} The homomorphism $ \rho_{0}  \colon A \to {}_{C}\mkern-1mu\mathcal{M}(D,A)$ associated to the trivial $D$-module algebra structure ${\Psi_{0}}$ is given by $ \rho_{0} (a)(d) = \varepsilon_{D}(d)a$ for all $a \in A$ and $d \in D$.
\end{enumerate}
\end{ex}

If $D_1$ and $D_2$ are $C$-bialgebras, then $D_1 \otimes_C D_2$ becomes a $C$-bialgebra in a natural way.
A pair of commuting $D_1$- and $D_2$-module algebra structures on $A$ gives rise to a $D_1 \otimes_C D_2$-module algebra structure on $A$ and vice versa.

For every $n \in {\mathbb N}$ the tensor product $D_{I \mkern-1.9mu D^{n}} \coloneqq {D_{I \mkern-1.9mu D^{}}}^{\otimes_{C}^{n}}$
is a cocommutative Hopf algebra and $D_{I \mkern-1.9mu D^{n}}$-module algebra structures correspond to systems of $n$ commuting iterative derivations, which we call \emph{$n$-variate iterative derivations} (cf. \cite{Heiderich:2007} or \cite{Maurischat:2010b}).
Furthermore, $D_{I \mkern-1.9mu D^{n}}$-measures from $R$ to itself such that $1_{D_{I \mkern-1.9mu D^{n}}}.a = a$ for all $a \in R$ correspond 1-1 to $n$-variate higher derivations on $R$ in the sense of \cite{Heiderich:2007}.
Similarly, we define for $n \in {\mathbb N}$ and ${\bm{m}} = (m_1, \dots, m_n) \in {\mathbb N}^n$ the $C$-coalgebra $D_{H \mkern-1.9mu D_{({\bm{m}})}^{n}} \coloneqq D_{H \mkern-1.9mu D_{(m_1)}^{}} \otimes_{C} \cdots \otimes_{C} D_{H \mkern-1.9mu D_{(m_n)}^{}}$ and note that $D_{I \mkern-1.9mu D^{n}}$ is isomorphic to $\varinjlim_{{\bm{m}} \in {\mathbb N}^{n}} D_{H \mkern-1.9mu D_{({\bm{m}})}^{n}}$ as $C$-coalgebra.

\begin{notation}
	For $C$-bialgebras $D_{1}$ and $D_{2}$ and $C$-algebras $A$ we often make implicitly use of the isomorphism of $C$-algebras
	\begin{equation}
	\label{eqn_implicit_isom}
		{}_{C}\mkern-1mu\mathcal{M}(D_{2}, {}_{C}\mkern-1mu\mathcal{M}(D_{1},A))
			\cong {}_{C}\mkern-1mu\mathcal{M}(D_{1} \otimes_C D_{2},A)
			\cong {}_{C}\mkern-1mu\mathcal{M}(D_{1}, {}_{C}\mkern-1mu\mathcal{M}(D_{2},A)).
	\end{equation}
\end{notation}

\begin{lem}
\label{lem_internal_D-module_algebra_structure_on_ModCDA}
Let $D$ be a $C$-bialgebra and $A$ a be $C$-algebra.
\begin{enumerate}
	\item	\label{itm_lem_internal_D-module_algebra_structure_on_ModCDA_def}
	The $C$-algebra ${}_{C}\mkern-1mu\mathcal{M}(D,A)$ becomes a $D$-module algebra by the homomorphism of $C$-modules
	\[{\Psi_{\mkern-0.7mu int}} \colon D \otimes_C {}_{C}\mkern-1mu\mathcal{M}(D,A) \to {}_{C}\mkern-1mu\mathcal{M}(D,A)\]
	\index{${\Psi_{\mkern-0.7mu int}}$}that sends $d \otimes f \in D \otimes_C {}_{C}\mkern-1mu\mathcal{M}(D,A)$ to the homomorphism of $C$-modules
	\[ {\Psi_{\mkern-0.7mu int}}(d \otimes f) \colon D \to A, \quad \tilde{d} \mapsto f(\tilde{d} d) \quad \text{ for all } \tilde{d} \in D. \]
	For any homomorphism of $C$-algebras $\varphi \colon A \to B$ the induced homomorphism of $C$-algebras
	\[ {}_{C}\mkern-1mu\mathcal{M}(D,\varphi) \colon {}_{C}\mkern-1mu\mathcal{M}(D,A) \to {}_{C}\mkern-1mu\mathcal{M}(D,B) \]
	is a homomorphism of $D$-module algebras with respect to the $D$-module algebra structures on ${}_{C}\mkern-1mu\mathcal{M}(D,A)$ and ${}_{C}\mkern-1mu\mathcal{M}(D,B)$ given by ${\Psi_{\mkern-0.7mu int}}$.
	Thus, ${}_{C}\mkern-1mu\mathcal{M}(D, -)$ is a functor from the category of $C$-algebras to the category of $D$-module algebras.

	\item \label{itm_lem_internal_D-module_algebra_structure_on_ModCDA_const} The constants ${}_{C}\mkern-1mu\mathcal{M}(D,A)^{{\Psi_{\mkern-0.7mu int}}}$ are equal to $ \rho_{0} (A)$, where $ \rho_{0}  \colon A \to {}_{C}\mkern-1mu\mathcal{M}(D,A)$ is the homomorphism associated to the trivial $D$-module algebra structure ${\Psi_{0}}$ on $A$ (cf. example~\ref{ex_duals_and_GTHs} \ref{ex_duals_and_GTHs_trivial_MAS}).

	\item \label{itm_lem_internal_D-module_algebra_structure_on_ModCDA_another_MAS} If $D'$ is another $C$-bialgebra and $\Psi'$ is a $D'$-module algebra structure on $A$ with associated homomorphism $ \rho_{} '$, then $\Psi'$ induces a $D'$-module algebra structure on ${}_{C}\mkern-1mu\mathcal{M}(D,A)$ with associated homomorphism given by
	\begin{equation}
		\label{eqn_module_algebra_structure_on_images}
		{}_{C}\mkern-1mu\mathcal{M}(D,  \rho_{} ') \colon {}_{C}\mkern-1mu\mathcal{M}(D,A) \to {}_{C}\mkern-1mu\mathcal{M}(D, {}_{C}\mkern-1mu\mathcal{M}(D',A)) \cong {}_{C}\mkern-1mu\mathcal{M}(D', {}_{C}\mkern-1mu\mathcal{M}(D,A)).
	\end{equation}
	This $D'$-module algebra structure commutes with the $D$-module algebra structure ${\Psi_{\mkern-0.7mu int}}$ on ${}_{C}\mkern-1mu\mathcal{M}(D,A)$ and thus ${}_{C}\mkern-1mu\mathcal{M}(D,A)$ becomes a $D \otimes_C D'$-module algebra.
\end{enumerate}
\end{lem}
\begin{proof}
	We note that the homomorphism of $C$-modules\index{$ \rho_{int} $}
	\[  \rho_{int}  \colon {}_{C}\mkern-1mu\mathcal{M}(D,A) \to {}_{C}\mkern-1mu\mathcal{M}(D, {}_{C}\mkern-1mu\mathcal{M}(D,A)), \]
	associated to ${\Psi_{\mkern-0.7mu int}}$ via the isomorphism~\eqref{eqn:fundamental_isomorphism}, corresponds to
	\[ {}_{C}\mkern-1mu\mathcal{M}(m_D, A) \colon {}_{C}\mkern-1mu\mathcal{M}(D,A) \to {}_{C}\mkern-1mu\mathcal{M}(D \otimes_C D, A) \]
	under the isomorphism of $C$-algebras~\eqref{eqn_implicit_isom}.
	Since $m_D$ is a homomorphism of $C$-coalgebras, ${}_{C}\mkern-1mu\mathcal{M}(m_D, A)$ and therefore also $ \rho_{int} $ are homomorphisms of $C$-algebras.
	The diagram
	\begin{align*}
		\xymatrix{
			{}_{C}\mkern-1mu\mathcal{M}(D,A) \ar[r]^-{ \rho_{int} } \ar[rd]_{\id} & {}_{C}\mkern-1mu\mathcal{M}(D, {}_{C}\mkern-1mu\mathcal{M}(D, A)) \ar[d]^{\operatorname{ev}_{1_D}} \\
			& {}_{C}\mkern-1mu\mathcal{M}(D,A)
		}
	\end{align*}
	obviously commutes.
	Using again the isomorphism~\eqref{eqn_implicit_isom},
	the commutativity of the diagram
	\begin{align*}
		\xymatrixcolsep{8em}
		\xymatrix{
			{}_{C}\mkern-1mu\mathcal{M}(D,A) \ar[r]^-{ \rho_{int} } \ar[d]^{ \rho_{int} } & {}_{C}\mkern-1mu\mathcal{M}(D,{}_{C}\mkern-1mu\mathcal{M}(D,A)) \ar[d]^{{}_{C}\mkern-1mu\mathcal{M}(D,  \rho_{int} )} \\
			{}_{C}\mkern-1mu\mathcal{M}(D,{}_{C}\mkern-1mu\mathcal{M}(D,A)) \ar[r]^-{{}_{C}\mkern-1mu\mathcal{M}(m_D, {}_{C}\mkern-1mu\mathcal{M}(D,A))} & {}_{C}\mkern-1mu\mathcal{M}(D \otimes_C D, {}_{C}\mkern-1mu\mathcal{M}(D,A))
		}
	\end{align*}
	follows from the associativity of $D$.
	Therefore, ${\Psi_{\mkern-0.7mu int}}$ is in fact a $D$-module algebra structure on ${}_{C}\mkern-1mu\mathcal{M}(D,A)$.
	For a homomorphism of $C$-algebras $\varphi \colon A \to B$, the big rectangle and the square on the right in the diagram
	\begin{align*}
		\xymatrixcolsep{3.2em}
		\xymatrix{
			{}_{C}\mkern-1mu\mathcal{M}(\mkern-1mu D, \mkern-1mu A) \ar[r]^-{ \rho_{int} } \ar[d]^{{}_{C}\mkern-1mu\mathcal{M}(\mkern-1mu D,\varphi)} \ar@/^1.7em/[rr]^{{}_{C}\mkern-1mu\mathcal{M}(m_D, A)} & {}_{C}\mkern-1mu\mathcal{M}(\mkern-1mu D, {}_{C}\mkern-1mu\mathcal{M}(\mkern-1mu D, \mkern-1mu A)) \ar[d]^{{}_{C}\mkern-1mu\mathcal{M}(\mkern-1mu D, {}_{C}\mkern-1mu\mathcal{M}(\mkern-1mu D,\varphi))} \ar[r]^-{\sim} & {}_{C}\mkern-1mu\mathcal{M}(\mkern-1mu D \mkern-1mu \otimes_C \mkern-2mu D, \mkern-1mu A) \ar[d]^{{}_{C}\mkern-1mu\mathcal{M}(\mkern-1mu D \otimes_C D,\varphi)} \\
			{}_{C}\mkern-1mu\mathcal{M}(\mkern-1mu D, \mkern-1mu B) \ar[r]^-{ \rho_{int} } \ar@/_1.7em/[rr]^{{}_{C}\mkern-1mu\mathcal{M}(m_D, B)} & {}_{C}\mkern-1mu\mathcal{M}(\mkern-1mu D, {}_{C}\mkern-1mu\mathcal{M}(\mkern-1mu D, \mkern-1mu B)) \ar[r]^-{\sim} & {}_{C}\mkern-1mu\mathcal{M}(\mkern-1mu D \mkern-1mu \otimes_C \mkern-2mu D, \mkern-1mu B) \\
		}
	\end{align*}
	commute and thus the rectangle on the left commutes too, i.e. ${}_{C}\mkern-1mu\mathcal{M}(D, \varphi)$ is a homomorphism of $D$-module algebras with respect to the $D$-module algebra structures given by ${\Psi_{\mkern-0.7mu int}}$ on ${}_{C}\mkern-1mu\mathcal{M}(D,A)$ and ${}_{C}\mkern-1mu\mathcal{M}(D,B)$.

	To prove part~\ref{itm_lem_internal_D-module_algebra_structure_on_ModCDA_const}, let $f \in {}_{C}\mkern-1mu\mathcal{M}(D,A)$.
Then $f$ is constant with respect to ${\Psi_{\mkern-0.7mu int}}$ if and only if $f(d) = ({\Psi_{\mkern-0.7mu int}}(d \otimes f))(1) = \varepsilon_D(d) f(1)$ for all $d \in D$, i.e. if and only if $f =  \rho_{0} (f(1))$.

	It is clear that ${}_{C}\mkern-1mu\mathcal{M}(D,  \rho_{} {'})$ induces a $D'$-module algebra structure on ${}_{C}\mkern-1mu\mathcal{M}(D,A)$, which we denote again by $\Psi'$.
Then we have for all $f \in {}_{C}\mkern-1mu\mathcal{M}(D,A), d, \tilde{d} \in D$ and $d' \in D'$
	\[
	\Psi'(d' \otimes {\Psi_{\mkern-0.7mu int}}(d \otimes f) )(\tilde{d})
	= \Psi'(d' \otimes f(\tilde{d} d) )
	= {\Psi_{\mkern-0.7mu int}}(d \otimes \Psi'(d' \otimes f)(\tilde{d}) ).
	\]
\end{proof}

\begin{lem}
	Given a $C$-bialgebra $D$, a $D$-module algebra structure $\Psi \in {}_{C}\mkern-1mu\mathcal{M}(D \otimes_C A, A)$ on $A$, the associated homomorphism $ \rho_{}  \colon A \to {}_{C}\mkern-1mu\mathcal{M}(D, A)$ is a homomorphism of $D$-module algebras from $(A,\Psi)$ to $({}_{C}\mkern-1mu\mathcal{M}(D,A), {\Psi_{\mkern-0.7mu int}})$.
	The homomorphism $ \rho_{} $ is universal among all homomorphisms of $D$-module algebras $\Lambda \colon (A,\Psi) \to ({}_{C}\mkern-1mu\mathcal{M}(D,B), {\Psi_{\mkern-0.7mu int}})$, where $B$ is a $C$-algebra, in the sense that for every such $\Lambda$ there exists a unique homomorphism of $C$-algebras $\lambda \colon A \to B$ such that $\Lambda = {}_{C}\mkern-1mu\mathcal{M}(D,\lambda) \circ  \rho_{} $.
	\label{lem_generalized_universal_Taylor_morphism_is_a_D-module_homomorphism}
\end{lem}
\begin{proof}
	Since ${}_{C}\mkern-1mu\mathcal{M}(D,  \rho_{} ) \circ  \rho_{}  = {}_{C}\mkern-1mu\mathcal{M}(m_D,A) \circ  \rho_{} $ and since the homomorphism of $C$-algebras $ \rho_{int} $ associated to ${\Psi_{\mkern-0.7mu int}}$ corresponds to ${}_{C}\mkern-1mu\mathcal{M}(m_D,A)$ under the isomorphism of $C$-algebras~\eqref{eqn_implicit_isom}, we see that $ \rho_{} $ is in fact a $D$-module algebra homomorphism from $(\mkern-0.5mu A, \mkern-0.5mu \Psi)$ to $({}_{C}\mkern-1mu\mathcal{M}\mkern-0.5mu (D,A), \mkern-0.5mu {\Psi_{\mkern-0.7mu int}})$.
	To show the universality of $ \rho_{} $, let $\Lambda \colon (A,\Psi) \to ({}_{C}\mkern-1mu\mathcal{M}(D,B), {\Psi_{\mkern-0.7mu int}})$ be a homomorphism of $D$-module algebras.
	We define $\lambda \coloneqq \operatorname{ev}_{1_D} \mkern-2mu \circ \Lambda$ and obtain
	${}_{C}\mkern-1mu\mathcal{M}(D,\lambda) \circ  \rho_{}
		= {}_{C}\mkern-1mu\mathcal{M}(D,\operatorname{ev}_{1_D}) \circ {}_{C}\mkern-1mu\mathcal{M}(D,\Lambda) \circ  \rho_{}
		= {}_{C}\mkern-1mu\mathcal{M}(D,\operatorname{ev}_{1_D}) \circ  \rho_{int}  \circ \Lambda
		= \Lambda$.
\end{proof}
For $D=D_{der}$ and $D=D_{end}$ one recoveres as corollaries
\cite[Proposition~1.4]{Umemura:1996b} and \cite[Propositions~2.5 and 2.7]{Morikawa:2009}.

\subsection{Extensions of module algebra structures}

\begin{prop}
	Let $D$ be a cocommutative $C$-bialgebra,
	\[(S, \Psi_S) \leftarrow (R, \Psi_R) \to (T, \Psi_T)\]
	be a diagram in the category of commutative $D$-module algebras and $ \rho_{R} $, $ \rho_{S} $ and $ \rho_{T} $ the homomorphisms associated to $\Psi_{R}$, $\Psi_{S}$ and $\Psi_{T}$, respectively.
	Then $S \otimes_R T$ carries a unique $D$-module algebra structure $\Psi$ such that $(S \otimes_R T, \Psi)$ becomes the coproduct of $(S, \Psi_S)$ and $(T, \Psi_T)$ over $(R, \Psi_R)$ in the category of commutative $D$-module algebras.
		This $D$-module algebra structure on $S \otimes_R T$ is given as
		\begin{equation}
			\Psi \colon D \otimes_C S \otimes_R T \to S \otimes_R T, \quad \quad d \otimes s \otimes t \mapsto \sum_{(d)} \mkern-2mu \Psi_S( d_{(1)} \otimes s) \otimes \Psi_T(d_{(2)} \otimes t)
			\label{eqn_D-mod_alg_structure_on_tensor_product_formula}
		\end{equation}
		and $\Psi$ corresponds to $ \rho_{S}  \otimes  \rho_{T} $ under the isomorphism \eqref{eqn:fundamental_isomorphism} when we identify ${}_{C}\mkern-1mu\mathcal{M}(D,S) \otimes_{{}_{C}\mkern-1mu\mathcal{M}(D,R)} {}_{C}\mkern-1mu\mathcal{M}(D,T)$ with ${}_{C}\mkern-1mu\mathcal{M}(D, S \otimes_R T)$.
	\label{prop_D-mod_alg_structure_on_tensor_product}
\end{prop}
\begin{proof}
	Since $D$ is cocommutative and $R$, $S$ and $T$ are commutative $C$-algebras, the $C$-algebras ${}_{C}\mkern-1mu\mathcal{M}(D,R)$, ${}_{C}\mkern-1mu\mathcal{M}(D,S)$ and ${}_{C}\mkern-1mu\mathcal{M}(D,T)$ are commutative.
	By the universal property of the coproduct $S \otimes_R T$ in the category of commutative $C$-algebras, there exists a unique homomorphism $ \rho_{}  \colon S \otimes_R T \to {}_{C}\mkern-1mu\mathcal{M}(D,S \otimes_R T)$ of $C$-algebras that makes the diagram
	\begin{align*}
		\xymatrixrowsep{3em}
		\xymatrixcolsep{4em}
		\xymatrix@!0{
			& S \ar[rr]^{ \rho_{S} } \ar[dr] & & {}_{C}\mkern-1mu\mathcal{M}(D,S) \ar[dr] \\
			R \ar[ur] \ar[dr] & & S \otimes_R T \ar[rr]^-{ \rho_{} } & & {}_{C}\mkern-1mu\mathcal{M}(D,S \otimes_R T)\\
			& T \ar[rr]^{ \rho_{T} } \ar[ur] & & {}_{C}\mkern-1mu\mathcal{M}(D,T) \ar[ur]
		}
	\end{align*}
	commutative.
	This homomorphism gives rise to a $D$-module algebra structure on $S \otimes_R T$, since the diagrams
	\begin{align*}
		\xymatrixrowsep{3em}
		\xymatrix{
			S \otimes_R T \ar[r]^-{ \rho_{} }  \ar[dr]_{\id} & {}_{C}\mkern-1mu\mathcal{M}(D, S \otimes_R T) \ar[d]^{\operatorname{ev}_{1_D}} \\
			& S \otimes_R T
		}
	\end{align*}
	and
	\begin{align*}
		\xymatrixrowsep{3em}
		\xymatrixcolsep{6em}
		\xymatrix{
			S \otimes_R T \ar[r]^{ \rho_{} } \ar[d]^{ \rho_{} } & {}_{C}\mkern-1mu\mathcal{M}(D, S \otimes_R T) \ar[d]^{{}_{C}\mkern-1mu\mathcal{M}(D,  \rho_{} )} \\
			{}_{C}\mkern-1mu\mathcal{M}(D, S \otimes_R T) \ar[r]^-{{}_{C}\mkern-1mu\mathcal{M}(m_D, S \otimes_R T)} & {}_{C}\mkern-1mu\mathcal{M}(D, {}_{C}\mkern-1mu\mathcal{M}(D, S \otimes_R T))
		}
	\end{align*}
	commute (which also follows from the universal property of $S \otimes_R T$).
	Using the universal property of $S \otimes_R T$ again, we see that $(S \otimes_R T, \Psi)$ is in fact the coproduct of $(S, \Psi_{S})$ and $(T, \Psi_{T})$ over $(R, \Psi_{R})$ in the category of commutative $D$-module algebras.

	Since $S \to S \otimes_R T, s \mapsto s \otimes 1$ and $T \to S \otimes_R T, t \mapsto 1 \otimes t$ are homomorphisms of $D$-module algebras, we have $\Psi(d \otimes s \otimes 1) = \Psi_S(d \otimes s) \otimes 1$ and $\Psi(d \otimes 1 \otimes t) = 1 \otimes \Psi_T(d \otimes t)$ for all $s \in S, t \in T$ and $d \in D$.
	Since $\Psi$ measures $S \otimes_R T$ to itself, it follows $\Psi(d \otimes s \otimes t) = \sum_{{(d)}} \Psi_S(d_{(1)} \otimes s) \otimes \Psi_T(d_{(2)} \otimes t)$.
\end{proof}

\begin{prop}
	\label{prop:extension_of_measuring_and_D-module_algebras_to_inverse_limits}
	\begin{enumerate}
	\item \label{itm:prop:extension_of_measuring_and_D-module_algebras_to_inverse_limits_measuring_part}
	Let $I$ and $J$ be two small categories,
	\[ R \colon {J}^{op} \to {\mathsf{CAlg}_{C}} \quad \text{ and } \quad S \colon {I}^{op} \to {\mathsf{CAlg}_{C}} \]
	be two functors.
	We write $R_j$ and $S_i$ instead of $R(j)$ and $S(i)$ for $j \in J$ and $i \in I$, respectively.
	Let $\Gamma$ be a small category and $\tilde{D}$ be a functor from $\Gamma$ to the category of cocommutative $C$-coalgebras and let $D$ be the colimit of $\tilde{D}$, i.e. $D = \varinjlim_{\gamma \in \Gamma} \tilde{D}(\gamma$).
	We denote $\tilde{D}(\gamma)$ by $D_\gamma$.
	Let
	\[\alpha \colon {(I \times \Gamma)}^{op} \to {J}^{op}\]
	be a functor and
	suppose that for every $i \in I$ and every $\gamma \in \Gamma$ the $C$-coalgebra $D_\gamma$ measures $R_{\alpha(i,\gamma)}$ to $S_{i}$ and denote the associated homomorphism of $C$-algebras by
	\[  \rho_{i,\gamma}  \colon R_{\alpha(i,\gamma)} \to {}_{C}\mkern-1mu\mathcal{M}(D_\gamma,S_{i}). \]
	We suppose that these measurings are compatible in the sense that for every morphism $i_1 \to i_2$ in $I$ and $\gamma_1 \to \gamma_2$ in $\Gamma$ the diagram
	\begin{align*}
		\xymatrix{	R_{\alpha(i_{2}, \gamma_2)} \ar[r]^-{ \rho_{i_{2}, \gamma_2} } \ar[d] & {}_{C}\mkern-1mu\mathcal{M}(D_{\gamma_2},S_{i_{2}}) \ar[d] \\
				R_{\alpha(i_{1}, \gamma_1)} \ar[r]^-{ \rho_{i_{1}, \gamma_1} } & {}_{C}\mkern-1mu\mathcal{M}(D_{\gamma_1},S_{i_{1}})
				}
	\end{align*}
	commutes.
	\footnote{This amounts to giving a natural transformation from the functor $R \circ \alpha \colon {(I \times \Gamma)}^{op} \to {\mathsf{CAlg}_{C}}, (i,\gamma) \mapsto R_{\alpha(i, \gamma)}$ to the functor ${(I \times \Gamma)}^{op} \to {\mathsf{CAlg}_{C}}, (i,\gamma) \mapsto {}_{C}\mkern-1mu\mathcal{M}(D_\gamma, S_i)$.}
	Then there exists a unique $D$-measure from $\hat{R} \coloneqq \varprojlim_{j \in {J}^{op}} R_j$ to $\hat{S} \coloneqq \varprojlim_{i \in {I}^{op}} S_i$ with associated homomorphism $\hat{ \rho_{} } \colon \hat{R} \to {}_{C}\mkern-1mu\mathcal{M}(D,\hat{S})$ such that
	\begin{align*}
		\xymatrix{	\hat{R} \ar[r]^-{\hat{ \rho_{} }} \ar[d]^{\pi_{\alpha(i,\gamma)}} & {}_{C}\mkern-1mu\mathcal{M}(D,\hat{S}) \ar[d]^{\operatorname{res}^{D}_{D_\gamma} \circ {}_{C}\mkern-1mu\mathcal{M}(D, \pi_i)} \\
				R_{\alpha(i, \gamma)} \ar[r]^-{ \rho_{i, \gamma} } & {}_{C}\mkern-1mu\mathcal{M}(D_\gamma,S_{i})
				}
	\end{align*}
	commutes for all $i \in I$ and $\gamma \in \Gamma$, where $\pi_j \colon \hat{R} \to R_j$ and $\pi_i \colon \hat{S} \to S_i$ denote the projections.

	\item \label{itm:prop:extension_of_measuring_and_D-module_algebras_to_inverse_limits_unit_part}
	We assume in addition that there are compatible homomorphisms $\eta_{D_\gamma} \colon C \to D_\gamma$ such that $1_{D_\gamma} \coloneqq \eta_{D_\gamma}(1_C)$ fulfills $\Delta_{D_\gamma}(1_{D_\gamma}) = 1_{D_\gamma} \otimes 1_{D_\gamma}$ and $\varepsilon_{D_\gamma}(1_{D_\gamma}) = 1_C$ for all $\gamma \in \Gamma$.
	Furthermore, we assume that there is a natural transformation from the functor $R \circ \alpha$ to $S$ (both are considered as functors from ${(I \times \Gamma)}^{op}$ to ${\mathsf{CAlg}_{C}}$, where $S$ does not depend on the second factor), giving rise to homomorphisms of $C$-algebras
	\begin{align}
		\label{eqn:prop:extension_of_measuring_and_D-module_algebras_to_inverse_limits:structure_homomorphism_Raig_Si}
		R_{\alpha(i,\gamma)} \to S_i
	\end{align}
	for all $i \in I$ and $\gamma \in \Gamma$ and thus to a homomorphism
	\begin{align}
		\label{eqn:prop:extension_of_measuring_and_D-module_algebras_to_inverse_limits:structure_homomorphism_R_S}
		\hat{R} \to \hat{S}.
	\end{align}
	If the diagram
	\begin{align*}
		\xymatrix{R_{\alpha(i,\gamma)} \ar[r]^-{ \rho_{i,\gamma} } \ar[rd] & {}_{C}\mkern-1mu\mathcal{M}(D_\gamma,S_i) \ar[d]^{\operatorname{ev}_{1_{D_\gamma}}} \\
		 & S_i
		}
	\end{align*}
	commutes for all $i \in I$ and $\gamma \in \Gamma$, then $\hat{ \rho_{} }$ makes the diagram
	\begin{align*}
		\xymatrix{\hat{R} \ar[r]^-{\hat{ \rho_{} }} \ar[rd] & {}_{C}\mkern-1mu\mathcal{M}(D,\hat{S}) \ar[d]^{\operatorname{ev}_{1_D}} \\
		 & \hat{S}
		}
	\end{align*}
	commutative.

	\item \label{itm:prop:extension_of_measuring_and_D-module_algebras_to_inverse_limits_iterative_part}
	We assume furthermore that $D$ is a (cocommutative) $C$-bialgebra, that $I=J$, that $R = S$ and that there is a functor $\mu \colon \Gamma \times \Gamma \to \Gamma, (\gamma, \tilde{\gamma}) \mapsto \mu_{\gamma \mkern-2mu, \mkern-2mu \tilde{\gamma}}$ and a natural transformation from $(\tilde{D} \otimes_C \tilde{D}) \colon \Gamma \times \Gamma \to {\mathsf{CAlg}_{C}}, (\gamma, \tilde{\gamma}) \mapsto \tilde{D}(\gamma) \otimes_{C} \tilde{D}(\tilde{\gamma})$ to $\tilde{D }\circ \mu$ inducing homomorphisms $m_{\gamma, \tilde{\gamma}} \colon D_{\gamma} \otimes_{C} D_{\tilde{\gamma}} \to D_{\mu_{\gamma \mkern-2mu, \mkern-2mu \tilde{\gamma}}}$ that are compatible with the multiplication $m_{D}$ in the sense that
	the diagram
	\begin{align}
	\label{eqn:compatibility_of_multiplication}
	\xymatrix{
		D \otimes_C D \ar[r]^-{m_{D}} & D \\
		D_{\gamma} \otimes_C D_{\tilde{\gamma}} \ar[u] \ar[r]^-{m_{\gamma, \tilde{\gamma}}} & D_{\mu_{\gamma \mkern-2mu, \mkern-2mu \tilde{\gamma}}} \ar[u]
	}
	\end{align}
	commutes for all $\gamma, \tilde{\gamma} \in \Gamma$.
	If the diagram
	\begin{align}
		\label{eqn:prop:extension_of_measuring_and_D-module_algebras_to_inverse_limits_iterativity_diagram}
		\xymatrixcolsep{2.4em}
		\xymatrix{	R_{\alpha(\alpha(i, \mu_{\gamma \mkern-2mu, \mkern-2mu \tilde{\gamma}}),\mu_{\gamma \mkern-2mu, \mkern-2mu \tilde{\gamma}})} \ar[d] \ar[r]
					& R_{\alpha(\alpha(i, \gamma)),\tilde{\gamma})} \ar[r]^-{ \rho_{\alpha(i,\gamma),\tilde{\gamma}} }
					&{}_{C}\mkern-1mu\mathcal{M}(D_{\tilde{\gamma}}, R_{\alpha(i,\gamma)}) \ar[dd]^{{}_{C}\mkern-1mu\mathcal{M}(D_{\tilde{\gamma}},  \rho_{i,\gamma} )} \\
				R_{\alpha(i,\mu_{\gamma \mkern-2mu, \mkern-2mu \tilde{\gamma}})} \ar[d]^-{ \rho_{i, \mu_{\gamma \mkern-2mu, \mkern-2mu \tilde{\gamma}}} } \\
				{}_{C}\mkern-1mu\mathcal{M}(D_{\mu_{\gamma \mkern-2mu, \mkern-2mu \tilde{\gamma}}}, R_{i}) \ar[rr]^-{{}_{C}\mkern-1mu\mathcal{M}(m_{\gamma, \tilde{\gamma}}, R_i)}
					&
					& {}_{C}\mkern-1mu\mathcal{M}(D_{\gamma} \otimes_C D_{\tilde{\gamma}}, R_i),
				}
	\end{align}
	where the vertical homomorphism with source the top left corner is an instance of \eqref{eqn:prop:extension_of_measuring_and_D-module_algebras_to_inverse_limits:structure_homomorphism_Raig_Si},
	commutes for all $\gamma, \tilde{\gamma} \in \Gamma$ and $i \in I$, then the $D$-measure from $\hat{R}$ to itself above is a $D$-module algebra structure.

	Furthermore, $(\hat{R}, \hat{ \rho_{} })$ has the following universal mapping property:
	Given a $D$-module algebra $(T,  \rho_{T} )$ and compatible homomorphisms $(\varphi_i \colon T \to R_i)_{i \in I}$ such that
	\begin{align}
	\label{eqn:compatibility_of_homomorphisms_varphi_i}
		\operatorname{res}^{D}_{D_\gamma} \circ {}_{C}\mkern-1mu\mathcal{M}(D, \varphi_i) \circ  \rho_{T}  =  \rho_{i,\gamma}  \circ \varphi_{\alpha(i,\gamma)}
	\end{align}
	for all $i \in I$ and $\gamma \in \Gamma$, there exists a unique homomorphism of $D$-module algebras $\Phi \colon T \to \hat{R}$ such that $\pi_i \circ \Phi = \varphi_i$ for all $i \in I$.
	\end{enumerate}
\end{prop}
\begin{proof}
	The universal mapping property of the limit
	$\varprojlim_{(i, \gamma) \in {(I \times \Gamma)}^{op}} \mkern-2mu {}_{C}\mkern-1mu\mathcal{M}(D_\gamma ,S_i)$
	provides a unique homomorphism of $C$-algebras
	\begin{align}
	\label{eqn:universal_Taylor_on_extension_first_part}
		\hat{R} \to \varprojlim_{(i, \gamma) \in {(I \times \Gamma)}^{op}} {}_{C}\mkern-1mu\mathcal{M}(D_\gamma,S_i)
	\end{align}
	such that the diagram
	\begin{align*}
		\xymatrixrowsep{1.5em}
		\xymatrix{
			& R_{\alpha(i_2,\gamma_2)} \ar[rr]^{ \rho_{i_2, \gamma_2} } \ar[dd]
				& & {}_{C}\mkern-1mu\mathcal{M}(D_{\gamma_2},S_{i_2}) \ar[dd] \\
			\hat{R} \ar[ur]^{\pi_{\alpha(i_2,\gamma_2)}} \ar[dr]_{\pi_{\alpha(i_1, \gamma_1)}} \ar@{.>}[rr]
				& & \varprojlim_{(i,\gamma) \in {(I \times \Gamma)}^{op}} {}_{C}\mkern-1mu\mathcal{M}(D_\gamma, S_i) \ar[ur] \ar[dr] & \\
			& R_{\alpha(i_1,\gamma_1)} \ar[rr]^{ \rho_{i_1, \gamma_1} } & & {}_{C}\mkern-1mu\mathcal{M}(D_{\gamma_1},S_{i_1})
		}
	\end{align*}
	commutes for all $i_1 \to i_2$ in $I$ and all $\gamma_1 \to \gamma_2$ in $\Gamma$.
	Since the functor ${}_{C}\mkern-1mu\mathcal{M}(D, -)$ preserves limits of $C$-modules
	and since the functors ${}_{C}\mkern-1mu\mathcal{M}(-, S_i)$ turn colimits into limits (cf. \cite[V.4]{MacLane:1971}),
	we have an isomorphism
	\begin{align}
		\label{eqn:universal_Taylor_on_extension_second_part}
		\varprojlim_{(i, \gamma) \in {(I \times \Gamma)}^{op}} {}_{C}\mkern-1mu\mathcal{M}(D_\gamma, S_i)
			\cong {}_{C}\mkern-1mu\mathcal{M}(\varinjlim_{\gamma \in \Gamma} D_\gamma, \varprojlim_{i \in {I}^{op}} S_i)
			= {}_{C}\mkern-1mu\mathcal{M}(D,\hat{S}),
	\end{align}
	which is a homomorphism of $C$-algebras.
	By composing \eqref{eqn:universal_Taylor_on_extension_first_part} and \eqref{eqn:universal_Taylor_on_extension_second_part}, we obtain a homomorphism of $C$-algebras
	\[
		\hat{ \rho_{} } \colon \hat{R} \to {}_{C}\mkern-1mu\mathcal{M}(D,\hat{S}).
	\]

	For the proof of part~\ref{itm:prop:extension_of_measuring_and_D-module_algebras_to_inverse_limits_unit_part}
	we note that the outer triangle in the diagram
	\begin{align*}
	\xymatrixrowsep{1.5em}
	\xymatrix@!C{
		R_{\alpha(i,\gamma)} \ar[rrr]^{ \rho_{i, \gamma} } \ar@/_4em/[rrrddd] & & & {}_{C}\mkern-1mu\mathcal{M}(D_\gamma,S_i) \ar[ddd]^{\operatorname{ev}_{1_{D}}} \\
		& \hat{R} \ar[lu]_{\pi_{\alpha(i,\gamma)}} \ar[r]^{\hat{ \rho_{} }} \ar[dr] & {}_{C}\mkern-1mu\mathcal{M}(D,\hat{S}) \ar[d]^{\operatorname{ev}_{1_D}} \\
		& & \hat{S} \ar[rd]^{\pi_i} \\
		& & & S_i,
	}
	\end{align*}
	where the unlabeled arrows are \eqref{eqn:prop:extension_of_measuring_and_D-module_algebras_to_inverse_limits:structure_homomorphism_Raig_Si} and \eqref{eqn:prop:extension_of_measuring_and_D-module_algebras_to_inverse_limits:structure_homomorphism_R_S},
	commutes by assumption for all $i \in I$ and $\gamma \in \Gamma$.
	Therefore, the inner triangle also commutes by the universal mapping property of the limit $\hat{S} = \varprojlim_{i \in {I}^{op}} S_i$.

	 By the assumption in part~\ref{itm:prop:extension_of_measuring_and_D-module_algebras_to_inverse_limits_iterative_part} the outer rectangle in the diagram
	\begin{align*}
		\xymatrixcolsep{0.3em}
		\xymatrixrowsep{2.0em}
		\xymatrix{
		R_{\alpha(\alpha( i, \mu_{\gamma \mkern-2mu, \mkern-2mu \tilde{\gamma}},\mu_{\gamma \mkern-2mu, \mkern-2mu \tilde{\gamma}})} \ar[dd] \ar[rr]
			&
			& R_{\alpha( \alpha(i, \gamma) ),\tilde{\gamma})} \ar[rr]^-{ \rho_{\alpha(i,\gamma),\tilde{\gamma}} }
			&
			& {}_{}\mkern-1mu\mathcal{M}( D_{\tilde{\gamma}}, R_{\alpha(i,\gamma )} ) \ar[dddd]^{{}_{}\mkern-1mu\mathcal{M}( D_{\tilde{\gamma}},  \rho_{i, \gamma} )} \\
		& \hat{R} \ar[rr]^{\hat{ \rho_{} }} \ar[dd]^{\hat{ \rho_{} }} \ar[ul]
			&
			& {}_{}\mkern-1mu\mathcal{M}(D, \hat{R}) \ar[dd]^{{}_{}\mkern-1mu\mathcal{M}(D,\hat{ \rho_{} })} \ar[ur]^(0.3){\operatorname{res}^{D}_{D_{\tilde{\gamma}}} \mkern-9mu \circ {}_{C}\mkern-1mu\mathcal{M}(D, \pi_{\alpha(i,\gamma)}) \phantom{-}} & \\
		R_{\alpha(i,\mu_{\gamma \mkern-2mu, \mkern-2mu \tilde{\gamma}})} \ar[dd]^-{ \rho_{i, \mu_{\gamma \mkern-2mu, \mkern-2mu \tilde{\gamma}}} } \\
		& {}_{}\mkern-1mu\mathcal{M}(D, \hat{R}) \ar[dl]^(0.3){\operatorname{res}^{D}_{D_{\mu_{\gamma \mkern-2mu, \mkern-2mu \tilde{\gamma}}}} \mkern-11mu \circ {}_{}\mkern-1mu\mathcal{M}(D,\pi_i)} \ar[rr]^-{{}_{C}\mkern-1mu\mathcal{M}(m_{D}, \hat{R})}
			&
				& {}_{}\mkern-1mu\mathcal{M}(D \otimes D, \hat{R}) \ar[rd]_(0.3){\operatorname{res}^{D \otimes D}_{D_{\gamma} \otimes D_{\tilde{\gamma}}} \mkern-11mu \circ {}_{}\mkern-1mu\mathcal{M}(D \otimes D, \pi_{i}) \phantom{--}} & \\
		{}_{}\mkern-1mu\mathcal{M}(D_{\mu_{\gamma \mkern-2mu, \mkern-2mu \tilde{\gamma}}}, R_{i}) \ar[rrrr]^-{{}_{}\mkern-1mu\mathcal{M}(m_{\gamma, \tilde{\gamma}}, R_i)}
			&
			&
			&
			& {}_{}\mkern-1mu\mathcal{M}(D_{\gamma} \otimes D_{\tilde{\gamma}}, R_i)
		}
	\end{align*}
	commutes for all $i \in I$ and $\gamma, \tilde{\gamma} \in \Gamma$, where all tensor products are over $C$ and where we abbreviate ${}_{C}\mkern-1mu\mathcal{M}$ by ${}_{}\mkern-1mu\mathcal{M}$.
	The trapezoids commute, since the projections $\pi_i \colon \hat{R} \to R_i$ are compatible with the $D$-measurings and since \eqref{eqn:compatibility_of_multiplication} commutes.
	Thus, by the universal mapping property of
	\[
		{}_{C}\mkern-1mu\mathcal{M}(D \otimes_C D, \hat{R}) \cong \varprojlim_{(i, \gamma, \tilde{\gamma}) \in {(I \times \Gamma \times \Gamma)}^{op}} {}_{C}\mkern-1mu\mathcal{M}(D_{\gamma} \otimes_C D_{\tilde{\gamma}}, R_i),
	\]
	the inner rectangle also commutes and we see that $\hat{R}$ is a $D$-module algebra.

	If $\varphi_i \colon T \to R_i$ are compatible homomorphisms such that \eqref{eqn:compatibility_of_homomorphisms_varphi_i} holds, then, by the universal property of $\hat{R} = \varprojlim_{i \in {I}^{op}} R_i$ in the category of commutative $C$-algebras, there exists a homomorphism of $C$-algebras $\Phi \colon T \to \hat{R}$ such that $\pi_{i} \circ \Phi = \varphi_i$ and so the triangles at the left and right in the diagram
	\begin{align*}
	\xymatrix{
		T \ar[rrr]^{ \rho_{T} } \ar[dd]^{\Phi} \ar[rd]^{\varphi_{\alpha(i,\gamma)}} & & & {}_{C}\mkern-1mu\mathcal{M}(D,T) \ar[dd]^{{}_{C}\mkern-1mu\mathcal{M}(D,\Phi)} \ar[ld]_{\operatorname{res}^{D}_{D_\gamma} \circ {}_{C}\mkern-1mu\mathcal{M}(D,\varphi_i) \phantom{-}} \\
		& R_{\alpha(i,\gamma)} \ar[r]^-{ \rho_{i, \gamma} } & {}_{C}\mkern-1mu\mathcal{M}(D_\gamma,R_i) \\
		\hat{R} \ar[rrr]^-{\hat{ \rho_{} }} \ar[ru]_{\pi_{\alpha(i,\gamma)}} & & & {}_{C}\mkern-1mu\mathcal{M}(D,\hat{R}) \ar[lu]^{\operatorname{res}^{D}_{D_\gamma} \circ {}_{C}\mkern-1mu\mathcal{M}(D,\pi_i) \phantom{--}}
	}
	\end{align*}
	commute for all $i \in I$.
	The two trapezoids at the top and bottom commute by assumption and by the previously shown, respectively.
	By the universal property of ${}_{C}\mkern-1mu\mathcal{M}(D,\hat{R}) \cong \varprojlim_{(i,\gamma) \in {(I \times \Gamma)}^{op}} {}_{C}\mkern-1mu\mathcal{M}(D_\gamma,R_i)$ we obtain ${}_{C}\mkern-1mu\mathcal{M}(D, \Phi) \circ  \rho_{T}  = \hat{ \rho_{} } \circ \Phi$, i.e. $\Phi$ is a homomorphism of $D$-module algebras.
\end{proof}

\begin{cor}
	\label{cor:inverse_systems_of_D-module_algebras_have_limits}
	Let $D$ be a cocommutative $C$-bialgebra.
	Then small inverse systems of $D$-module algebras have inverse limits.
\end{cor}
\begin{proof}
	This follows from proposition~\ref{prop:extension_of_measuring_and_D-module_algebras_to_inverse_limits} by taking $I=J$, $R=S$, the category $\Gamma$ to consists of exactly one object $\gamma$ and one morphism $\id_\gamma$ and $D_\gamma = D$ and by defining $\alpha \colon {I}^{op} \times {\Gamma}^{op} \to {I}^{op}$ by $\alpha(i,\gamma)=i$ for all $i \in I$.
\end{proof}

\begin{cor}
	\label{cor_extension_of_continous_D-measurings_to_completions}
	Let $R$ be a linear topological ring with fundamental system of neighborhoods $\mathcal{B}$ of $0$.
	Let $\Gamma$ be a small category and $\tilde{D}$ be a functor from $\Gamma$ to the category of cocommutative $C$-coalgebras and let $D$ be the colimit of $\tilde{D}$, i.e. $D = \varinjlim_{\gamma \in \Gamma} \tilde{D}(\gamma$).
	We denote $\tilde{D}(\gamma)$ by $D_\gamma$.
	Let $\Psi \in {}_{C}\mkern-1mu\mathcal{M}(D \otimes_C R, R)$ measure $R$ to $R$ and assume that the associated homomorphism of $C$-algebras $ \rho_{}  \colon R \to {}_{C}\mkern-1mu\mathcal{M}(D,R)$ is continuous with respect to the given topology on $R$ and the linear topology on ${}_{C}\mkern-1mu\mathcal{M}(D,R)$ with fundamental system of neighborhoods of $0$ given by the ideals
	\[ ({}_{C}\mkern-1mu\mathcal{M}(D,I),\Kernel({}_{C}\mkern-1mu\mathcal{M}(D,R) \to {}_{C}\mkern-1mu\mathcal{M}(D_\gamma,R))) \]
	for all $I \in \mathcal{B}$ and $\gamma \in \Gamma$.
	Then $\Psi$ uniquely extends to a $D$-measure $\hat{\Psi} \in {}_{C}\mkern-1mu\mathcal{M}(D \otimes_C \hat{R}, \hat{R})$ from $\hat{R} \coloneqq \varprojlim_{I \in \mathcal{B}} R/I$ to itself and the associated homomorphism of $C$-algebras $\hat{ \rho_{} } \colon \hat{R} \to {}_{C}\mkern-1mu\mathcal{M}(D, \hat{R})$ is continuous with respect to the induced topology on $\hat{R}$ and the linear topology on ${}_{C}\mkern-1mu\mathcal{M}(D,\hat{R})$ with fundamental system of neighborhoods of $0$ given by the ideals
	\[ ({}_{C}\mkern-1mu\mathcal{M}(D,\hat{I}),\Kernel({}_{C}\mkern-1mu\mathcal{M}(D,\hat{R}) \to {}_{C}\mkern-1mu\mathcal{M}(D_\gamma,\hat{R}))) \]
	for all $\hat{I} \in \mathcal{\hat{B}}$ and $\gamma \in \Gamma$, where $\mathcal{\hat{B}}$ is a fundamental system of neighborhoods of $0 \in \hat{R}$.

	If for every $\gamma \in \Gamma$ there is a homomorphism $\eta_{D_\gamma}$ as in proposition~\ref{prop:extension_of_measuring_and_D-module_algebras_to_inverse_limits} \ref{itm:prop:extension_of_measuring_and_D-module_algebras_to_inverse_limits_unit_part} such that $1_{D_\gamma} \coloneqq \eta_{D_\gamma}(1_C)$ fulfills $\Psi(1_{D_\gamma} \otimes a) = a$ for all $a \in R$, then $\hat{\Psi}(1_D \otimes a) = a$ holds for all $a \in \hat{R}$.

	We furthermore assume the existence of a functor $\mu \colon \Gamma \times \Gamma \to \Gamma$ and homomorphisms $m_{\gamma_1, \gamma_2} \colon D_{\gamma_1} \otimes_{C} D_{\gamma_2} \to D_{\mu_{\gamma_1 \mkern-2mu, \mkern-2mu \gamma_2}}$ for all $\gamma_1, \gamma_2 \in \Gamma$ as in proposition~\ref{prop:extension_of_measuring_and_D-module_algebras_to_inverse_limits} \ref{itm:prop:extension_of_measuring_and_D-module_algebras_to_inverse_limits_iterative_part} and that $D$ is a (cocommutative) $C$-bialgebra.
	If $\Psi \in {}_{C}\mkern-1mu\mathcal{M}(D \otimes_C R, R)$ is a $D$-module algebra structure on $R$, then its extension $\hat{\Psi}$ is a $D$-module algebra structure on $\hat{R}$.
\end{cor}
\begin{proof}
	For all $I \in \mathcal{B}$ and $\gamma \in \Gamma$ there exists an ideal $\alpha(I,\gamma)$ in $\mathcal{B}$ such that $ \rho_{} (\alpha(I,\gamma)) \subseteq ({}_{C}\mkern-1mu\mathcal{M}(D,I),\Kernel({}_{C}\mkern-1mu\mathcal{M}(D,R) \to {}_{C}\mkern-1mu\mathcal{M}(D_\gamma,R)))$
	and we obtain homomorphisms of $C$-algebras $ \rho_{I,\gamma}  \colon R/{\alpha(I,\gamma)} \to {}_{C}\mkern-1mu\mathcal{M}(D_\gamma,R/I)$, which are compatible in the sense that for all $I_2 \subseteq I_1$ in $\mathcal{B}$ and all $\gamma_1 \to \gamma_2$ in $\Gamma$ the diagram
	\begin{align*}
		\xymatrix{	R/{\alpha(I_{2},\gamma_2)} \ar[r]^-{ \rho_{I_{2}, \gamma_2} } \ar[d] & {}_{C}\mkern-1mu\mathcal{M}(D_{\gamma_2},R/{I_{2}}) \ar[d] \\
				R/{\alpha(I_{1},\gamma_1)} \ar[r]^-{ \rho_{I_{1}, \gamma_1} } & {}_{C}\mkern-1mu\mathcal{M}(D_{\gamma_1},R/{I_{1}})
				}
	\end{align*}
	commutes.
	Then the three claims follow from the corresponding parts of proposition~\ref{prop:extension_of_measuring_and_D-module_algebras_to_inverse_limits}.
\end{proof}

In the special case of higher and iterative derivations on adic linear topological rings we obtain:
\begin{cor}
	\label{cor:continous_iterative_derivations_extend_to_completions}
	Let $R$ be a linear topological ring with respect to the $I$-adic topology on $R$, where $I$ is an ideal in $R$, and $\theta_{} \colon R \to R\llbracket {\bm{t}} \rrbracket$ be an $n$-variate higher derivation on $R$ that is continuous with respect to the $I$-adic topology on $R$ and the $(I,{\bm{t}})$-adic topology on $R\llbracket {\bm{t}} \rrbracket$.
	Then the $n$-variate higher derivation $\theta_{}$ extends uniquely to an $n$-variate higher derivation $\hat{\theta}_{} \colon \hat{R} \to \hat{R}\llbracket {\bm{t}} \rrbracket$ on the completion $\hat{R}$, which is again continuous.
	If $\theta_{}$ is iterative, then $\hat{\theta}_{}$ is iterative too.
\end{cor}
\begin{proof}
	We define $\mathcal{B} \coloneqq \{ I^k \mid k \in {\mathbb N} \}$,  $D \coloneqq D_{I \mkern-1.9mu D^{n}}$, $\Gamma \coloneqq {\mathbb N}^n$, $D_{\bm{l}} \coloneqq D_{H \mkern-1.9mu D_{({\bm{l}})}^{n}}$ for all ${\bm{l}} \in {\mathbb N}^n$, $\mu({\bm{l_1}}, {\bm{l_2}}) \coloneqq {\bm{l_1}} + {\bm{l_2}}$ and let
	\[ m_{{\bm{l_1}}, {\bm{l_2}}} \colon D_{H \mkern-1.9mu D_{({\bm{l_1}})}^{n}} \otimes_C D_{H \mkern-1.9mu D_{({\bm{l_2}})}^{n}} \to D_{H \mkern-1.9mu D_{({\bm{l_1 + l_2}})}^{n}}\]
	be the restriction of the multiplication on $D_{I \mkern-1.9mu D^{n}}$.
	Identifying ${}_{C}\mkern-1mu\mathcal{M}(D_{I \mkern-1.9mu D^{n}},R)$ with $R\llbracket {\bm{t}} \rrbracket$, the ideal $({}_{C}\mkern-1mu\mathcal{M}(D,I^k),\Kernel({}_{C}\mkern-1mu\mathcal{M}(D,R) \to {}_{C}\mkern-1mu\mathcal{M}(D_{\bm{l}},R)))$ corresponds to the ideal $(I^k, t_{1}^{l_{1} + 1}, \dots, t_{n}^{l_{n} + 1})$ for all $k \in {\mathbb N}$ and ${\bm{l}} \in {\mathbb N}^n$, which form a base of neighborhoods of $0$ of the linear topological ring $R\llbracket {\bm{t}} \rrbracket$.
	Since $\theta_{}$ is continuous, there exists for every $k \in {\mathbb N}$ and ${\bm{l}} \in {\mathbb N}^n$ an $\ell \in {\mathbb N}$ such that $\theta_{}(I^\ell) \subseteq (I^k, t_{1}^{l_{1} + 1}, \dots, t_{n}^{l_{n} + 1})$.
	So the claim follows from corollary~\ref{cor_extension_of_continous_D-measurings_to_completions}.
\end{proof}

\begin{ex}
\label{ex:iterative_derivation_with_respect_to_variables}
Let $R$ be a commutative $C$-algebra and $n \in {\mathbb N}$.
On $R[{\bm{x}}]$, where ${\bm{x}} = (x_1, \dots, x_n)$, there is an $n$-variate iterative derivation $\theta_{{\bm{x}}} \colon R[{\bm{x}}] \to R[{\bm{x}}]\llbracket {\bm{w}} \rrbracket$ (with ${\bm{w}} \coloneqq (w_{1}, \dots, w_{n})$) defined by $\theta_{{\bm{x}}}(x_i) = x_i + w_i$ for $i=1, \dots, n$.
This $n$-variate iterative derivation extends uniquely to $R({\bm{x}})$ (cf. \cite[Theorem 27.2]{Matsumura:1989} or \cite[Proposition 1.2.2]{Heiderich:2010}) and by corollary~\ref{cor:continous_iterative_derivations_extend_to_completions} it uniquely extends to $R\llbracket {\bm{x}} \rrbracket$.
We denote these extensions (and their extensions to formally étale extensions) again by $\theta_{{\bm{x}}}$.
\end{ex}

\subsection{Simple and Artinian simple module algebras}
\begin{defn}
	Let $D$ be a $C$-bialgebra.
	\begin{enumerate}
		\item
		A \emph{simple commutative $D$-module algebra} is a commutative $D$-module algebra $(A, \Psi_{A})$ that has no non-trivial $D$-stable ideals, i.e. ideals $I \unlhd A$ such that $\Psi_{A}(D \otimes I) \subseteq I$.

		\item
		An \emph{Artinian simple commutative $D$-module algebra} is a simple commutative $D$-module algebra $A$ that is Artinian as a ring.
	\end{enumerate}
\end{defn}

\begin{defn}
	Let $D$ be a $C$-bialgebra.
	\begin{enumerate}
		\item If $L | K$ is an extension of Artinian simple commutative $D$-module algebras and $B$ is a subset of $L$, then we denote by $K \langle B \rangle$ the smallest Artinian simple $D$-module subalgebra of $L$ that contains $K$ and $B$.
		The extension $L | K$ is \emph{finitely generated} as extension of Artinian simple commutative $D$-module algebras if there exists a finite subset $B$ of $L$ such that $L = K \langle B \rangle$.

		\item If $S | R$ is an extension of commutative $D$-module algebras and $B$ is a subset of $S$, then we denote by $R \{ B \}_{D}$ (or also by $R \{ B \}_{\Psi}$, $R \{ B \}_{ \rho_{} }$ or $R \{ B \}$ if there is no risk of confusion) the smallest $D$-module subalgebra of $S$ that contains $R$ and $B$.
	\end{enumerate}
\end{defn}

\begin{lem}
\label{lem:constants_of_simple_module_algebras_are_fields_and_linear_disjointness}
	Let $D$ be a $C$-bialgebra and $(A, \Psi_{A})$ be a simple commutative $D$-module algebra.
	Then ${A}^{\Psi_{A}}$ is a field and for every $C$-module $V$ and every $\Psi_{V} \in {}_{C}\mkern-1mu\mathcal{M}(A \# D \otimes_{C} V, V)$ the natural homomorphism
	\[ A \otimes_{A^{\Psi_{A}}} V^{\tilde{\Psi}_{V}} \to V \]
	is injective, where $\tilde{\Psi}_{V} \colon D \otimes_{C} V \to V$ is the restriction of $\Psi_{V}$.
\end{lem}
\begin{proof}
	Amano and Masuoka prove this in \cite[Corollary 3.2]{AmanoMasuoka:2005} under the hyphothesis that $D$ is cocommutative and their proof holds also if $D$ is not cocommutative.
\end{proof}

\begin{lem}
	\label{lem_simple_linear_disjointness_for_constants}
	Let $A$ be a commutative $C$-algebra, $B$ be a commutative $A$-algebra and $D$ be a cocommutative $C$-bialgebra.
	Then ${}_{C}\mkern-1mu\mathcal{M}(D,A)$ and $ \rho_{0} (B)$ are linearly disjoint over $ \rho_{0} (A)$ as subalgebras of ${}_{C}\mkern-1mu\mathcal{M}(D,B)$.
\end{lem}
\begin{proof}
	Let $b_1, \dots, b_n \in B$ be linearly independent over $A$.
	If $f_1, \dots, f_n \in {}_{C}\mkern-1mu\mathcal{M}(D,A)$ are such that $\sum_{i=1}^n  f_i  \rho_{0} (b_i) = 0$, it follows $0 = \sum_{i=1}^n  (f_i  \rho_{0} (b_i))(d) = \sum_{i=1}^n f_i(d) b_i$ for all $d \in D$.
	Since $b_1, \dots, b_n$ are linearly independent over $A$, we obtain $f_i(d) = 0$ for all $d \in D$ and all $i=1, \dots, n$.
\end{proof}

The following results are generalizations of results from
\cite[Section 2]{AmanoMasuoka:2005}, the main difference being that in our case $D$ is not necessarily a Hopf algebra.
Instead, until the end of this section we make the following
\begin{ass}
\label{ass:bialgebra_D}
	Let $C$ be a field, $G$ a monoid and $D^1$ be a pointed irreducible cocommutative Hopf algebra of Birkhoff-Witt type over $C$, i.e. $D^{1}$ is of the form $B(U)$, where $U$ is a $C$-vector space and $B(U)$ is the cofree pointed irreducible cocommutative coalgebra on $U$ as defined in \cite[pp. 261-271]{Sweedler:1969}, such that $D^1$ is a $C G$-module algebra, where $C G$ is the $C$-bialgebra defined in example~\ref{ex:module_algebra_structures}~\ref{itm:ex:module_algebra_structures:group_like_bialgebras_over_a_monoid}.
	We define $D \coloneqq D^1 \# C G$ and for every submonoid $G'$ of $G$ we define $D(G') \coloneqq D^1 \# C G'$ (cf. \cite[Section 7.2]{Sweedler:1969}).
\end{ass}

\begin{prop}
\label{prop:statements_on_Noetherian_simple_D-module_algebras}
	Let $R$ be a commutative $D$-module algebra that is Noetherian as a ring and simple as $D$-module algebra.
	We further assume that each $g \in G$ acts as an injective endomorphism on $R$, i.e. the endomorphism $R \to R, a \mapsto \Psi(g \otimes a)$ is injective (this is the case for example if $G$ is a group or if $G$ is commutative, cf. lemma~\ref{lem:elements_of_commutative_monoids_act_as_injective_endomorphisms_on_Noetherian_simple_D-module_algebras}).
	We consider the induced action of $G$ on $\Omega(R)$ from the right defined by $P.g \coloneqq g^{-1}(P)$ for $g \in G$ and $P \in \Omega(R)$, where $g^{-1}(P)$ denotes the inverse image of $P$ under $g$, and denote by $G_{P} \coloneqq \{ g \in G \mid g^{-1}(P) = P \}$ the stabilizer of $P$ in $G$.
	Then the following hold:
	\begin{enumerate}
		\item \label{minimal_prime_ideals_D1_stable}
		Every $P \in \Omega(R)$ is $D^{1}$-stable, so that $R/P$ becomes a $D(G_{P})$-module domain.
		As $D(G_{\Omega(R)})$-module algebra $R/P$ is simple, where $G_{\Omega(R)} \coloneqq \cap_{Q \in \Omega(R)} G_Q$.
		\item \label{itm:prop:statements_on_Noetherian_simple_D-module_algebras:action_transitive}
		  The right action of $G$ on $\Omega(R)$ is transitive.
		  Therefore the stabilizers $G_{P}$ are conjugate to each other.
		\item \label{product_decomposition}
		There is a natural isomorphism $R \cong \prod_{P \in \Omega(R)} R/P$.
	\end{enumerate}
\end{prop}
\begin{proof}
	The proof is similar to the one of \cite[Proposition 2.4]{AmanoMasuoka:2005}, but some modifications are necessary due to the fact that in general the $g \in G$ do not act as automorphisms on $R$.
	We first note that the set $\Omega(R)$ of minimal prime ideals of $R$ is finite.
	Since $R$ is reduced, $\cap_{Q \in \Omega(R)} Q = 0$ and since all $g \in G$ act as injective endomorphisms on $R$, it follows
	\[\cap_{Q \in \Omega(R)} g^{-1}(Q) = \cap_{Q \in \Omega(R)} Q = 0. \]
	Therefore $g^{-1}(Q) \in \Omega(R)$ for all $Q \in \Omega(R)$ and all $g \in G$.
	Since $[G:G_Q] \leq | \Omega(R) |$, the index $[G:G_Q]$ is finite for all $Q \in \Omega(R)$ and hence is $[G : G_{\Omega(R)}]$.

	To prove \ref{minimal_prime_ideals_D1_stable}, we denote by $ \rho_{1}  \colon R \to {}_{C}\mkern-1mu\mathcal{M}(D^{1},R)$ the homomorphism of $C$-algebras associated to the restriction of the $D$-module algebra structure to a $D^1$-module algebra structure.
	Then the inverse image $ \rho_{1} ^{-1}({}_{C}\mkern-1mu\mathcal{M}(D^1,P))$ is a $D^1$-stable prime ideal of $R$ that is included in $P$.
	Since $P$ is a minimal prime ideal, it follows $P =  \rho_{1} ^{-1}( {}_{C}\mkern-1mu\mathcal{M}(D^1,P) )$ and so $P$ is $D^1$-stable.

	Let $J $ be a proper $D(G_{\Omega(R)})$-stable ideal of $R$ containing $P$.
	Let $1 = g_1, \dots, g_\nu$ be a system of representators of the cosets $G/G_{\Omega(R)}$.
	Then $\cap_{i=1}^\nu g_i^{-1}(J)$ is $D$-stable
	and thus equal to $(0)$, since $R$ is a simple $D$-module algebra.
	Since $P$ is a prime ideal, there exists an $i \in \{1, \dots, \nu\}$ such that $g_i^{-1}(J) \subseteq P$.
	Together with our assumption $P \subseteq J$ we obtain $g_i^{-1}(P) \subseteq g_i^{-1}(J) \subseteq P$.
	Since $P$ is a minimal prime ideal, it follows $g_i^{-1}(P) = P $, i.e. $i = 1$ and thus $J = P$.
	Therefore $R/P$ is a simple $D(G_{\Omega(R)})$-module algebra.

	To prove \ref{itm:prop:statements_on_Noetherian_simple_D-module_algebras:action_transitive}, let $P \in \Omega(R)$.
	By part \ref{minimal_prime_ideals_D1_stable}, $g^{-1}(P)$ is $D^{1}$-stable for all $g \in G$.
	Therefore the intersection $\cap_{g \in G/G_P} g^{-1}(P)$ is $D$-stable and thus
	we have
	\[ \cap_{Q \in \Omega(R)} Q = (0) = \cap_{g \in G/G_P} g^{-1}(P), \]
	since
	$R$ is a simple $D$-module algebra.
	It follows $\{g^{-1}(P) \mid g \in G/G_P \} = \Omega(R)$.

	Now we prove \ref{product_decomposition}:
	By the previous we have a bijection $G/G_P \cong \Omega(R)$ for every $P \in \Omega(R)$.
	If $Q$ and $Q'$ are distinct elements of $\Omega(R)$, then $(Q \subsetneq) Q + Q' = R$, since $R/Q$ is a simple $D(G_{\Omega(R)})$-module algebra.
	Hence $R \to \prod_{Q \in \Omega(R)} R/Q \cong \prod_{g \in G/G_P} R/{g^{-1}(P)}$ is an isomorphism by the Chinese remainder theorem.
\end{proof}

\begin{lem}
\label{lem:elements_of_commutative_monoids_act_as_injective_endomorphisms_on_Noetherian_simple_D-module_algebras}
	If the monoid $G$ is commutative and $R$ is a commutative $D$-module algebra, then the kernel $\Kernel{g}$ of each $g \in G$ is a $D$-stable ideal of $R$.
	In particular, if $R$ is a simple commutative $D$-module algebra, then all $g \in G$ act as injective endomorphisms of $R$.
\end{lem}
\begin{proof}
	We note that for all $g \in G$ and $\sum_{i=1}^n d_i \# h_i \in D^1 \# C G$ we have
	\[
	(1 \# g)(\sum_{i=1}^n d_i \# h_i)
	= \sum_{i=1}^n g(d_i) \# g h_i
	= \sum_{i=1}^n g(d_i) \# h_i g
	= (\sum_{i=1}^n g(d_i) \# h_i) (1 \# g).
	\]
	So for $a \in \Kernel{g}$ we have
	$(1 \# g)(\sum_{i=1}^n d_i \# h_i) (a) = (\sum_{i=1}^n g(d_i) \# h_i) (1 \# g)(a) = 0$, and thus $(\sum_{i=1}^n d_i \# h_i) (a) \in \Kernel(g)$.
	Therefore $\Kernel(g)$ is a $D$-stable ideal of $R$.
\end{proof}

\begin{prop}
\label{prop:equivalent_conditions_to_Artinian_for_a_Noetherian_simple_D-module_algebra_with_injective_group_like_elements}
	Let $R$ be a Noetherian simple commutative $D$-module algebra.
	If furthermore each $g \in G$ acts as an injective endomorphism on $R$,
	then the following are equivalent:
	\begin{enumerate}
		\item $R$ is total,
		\item $R/P$ is a field for every $P \in \Omega(R)$,
		\item The Krull dimension of $R$ is zero and
		\item $R$ is Artinian as a ring.
	\end{enumerate}
	In this case every $R$-module is projective.

\end{prop}
\begin{proof}
	Since $R \cong \prod_{P \in \Omega(R)} R/P$ by proposition~\ref{prop:statements_on_Noetherian_simple_D-module_algebras} \ref{product_decomposition}, $R$ is total if and only if $R/P$ is total for every $P \in \Omega(R)$ and since $R/P$ is an integral domain, this is the case if and only if $R/P$ is a field.
	The Krull dimension of $R$ is zero if and only if each $P \in \Omega(R)$ is a maximal ideal, i.e. if $R/P$ is a field.
	Since $R$ is Noetherian, it is Artinian if and only if its Krull dimension is $0$ (cf. \cite[Chapitre IV, \S 2.5, Proposition 9]{Bourbaki:1985:AlgebreCommutative:Ch1-4}).

	In this case $R$ is semi-simple and so every $R$-module is projective (cf. \cite[Chapitre I, \S 2.4]{Bourbaki:1985:AlgebreCommutative:Ch1-4}).
\end{proof}
If $G$ is a group, then we recover the first part of \cite[Corollary 2.5]{AmanoMasuoka:2005} from proposition~\ref{prop:equivalent_conditions_to_Artinian_for_a_Noetherian_simple_D-module_algebra_with_injective_group_like_elements}.

\section{Galois theory of Artinian simple module algebras}

\subsection{Notation}
\label{ssec:notation}

As in assumption~\ref{ass:bialgebra_D}, let $C$ be a field, $G$ be a monoid
and let $C G$ be the $C$-bialgebra defined in example~\ref{ex:module_algebra_structures}~\ref{itm:ex:module_algebra_structures:group_like_bialgebras_over_a_monoid}.
Let further $D^1$ be a pointed irreducible cocommutative Hopf algebra of Birkhoff-Witt type over $C$ such that $D^1$ is a $C G$-module algebra.
We define $D \coloneqq D^1 \# C G$.

\begin{rmk}
These conditions allow for example the choices $D_{end}$, $D_{aut}$ or $D_{I \mkern-1.9mu D^{n}}$ for $D$ (cf. example~\ref{ex:module_algebra_structures}).
The cocommutative pointed irreducible commutative Hopf algebra $D_{der}$ is of Birkhoff-Witt type if ${\mathbb Q} \subseteq C$.
Another example is provided by the Hopf algebra constructed by Masuoka in \cite{Masuoka:cross_bialgebra:preprint}, describing iterative $q$-difference operators as introduced by Hardouin in \cite{Hardouin:2010}.
\end{rmk}

Let $L | K$ be an extension of Artinian simple commutative $D$-module algebras such that the elements $g \in G \subseteq D$ act as injective endomorphisms on $L$.
By proposition~\ref{prop:equivalent_conditions_to_Artinian_for_a_Noetherian_simple_D-module_algebra_with_injective_group_like_elements}, $L$ and $K$ have Krull dimension $0$, are total and the monoid $G$ acts transitively on $\Omega(L)$ and $\Omega(K)$ (cf. proposition~\ref{prop:statements_on_Noetherian_simple_D-module_algebras} \ref{itm:prop:statements_on_Noetherian_simple_D-module_algebras:action_transitive}).
If $Q \in \Omega(L)$ and $P \in \Omega(K)$, then we denote by $G_Q$ and $G_{P}$ the stabilizers in $G$ of $Q$ and $P$, respectively.
It follows $\Omega(L) \cong G/G_Q$, $\Omega(K) \cong G/G_P$ and by proposition~\ref{prop:statements_on_Noetherian_simple_D-module_algebras} we have
\[ K \cong \prod_{P \in \Omega(P)} K/P
\quad \text{and} \quad
L \cong \prod_{Q \in \Omega(L)} L/Q.\]
We assume that for every $Q \in \Omega(L)$ the field extension $L/Q$ over $K/(K \cap Q)$ is separable and finitely generated and that its degree of transcendence is the same for every $Q$, say $n$.
Let ${\bm{u}}^Q = (u^Q_1, \dots, u^Q_n)$ be a separating transcendence basis of it and let $\theta_{{\bm{u}}^Q} \colon L/Q \to L/Q\llbracket {\bm{w}} \rrbracket$ be the associated $n$-variate iterative derivation of $L/Q$ over $K/(Q \cap K)$ defined by $\theta_{{\bm{u}}^Q}(u^Q_i) = u^Q_i + w_i$ for all $i=1, \dots, n$ (cf. example~\ref{ex:iterative_derivation_with_respect_to_variables}).
By corollary~\ref{cor:inverse_systems_of_D-module_algebras_have_limits}, there exists a unique $n$-variate iterative derivation
\begin{equation}
\label{eq:iterative_derivation_wrt_u}
	\theta_{{\bm{u}}} \colon L \to L \llbracket {\bm{w}} \rrbracket
\end{equation}
on the product $L \cong \prod_{Q \in \Omega(L)} L/Q$ over $K$ such that the projections to all factors $L/Q$ are iterative differential homomorphisms.
We extend $\theta_{{\bm{u}}}$ to ${}_{C}\mkern-1mu\mathcal{M}(D,L)$ as in \eqref{eqn_module_algebra_structure_on_images} and by abuse of notation we denote this $n$-variate iterative derivation again by $\theta_{{\bm{u}}}$.
	By \cite[Chapitre 0, 21.7.4]{EGA4-1_IHES:1964}, $L/Q$ is formally étale over $(K/(Q \cap K))({\bm{u}}^Q)$ for all  $Q \in \Omega(L)$.
	We define $u_{i} \coloneqq (u_{i}^Q)_{Q \in \Omega(L)} \in L$ for all $i \in \{1, \dots, n\}$ and ${\bm{u}} \coloneqq (u_{1}, \dots, u_{n})$.
	Then $L$ is formally étale over
	$\prod_{Q \in \Omega(L)} (K/(Q \cap K))({\bm{u}}^Q)$.
	We denote the latter by $K({\bm{u}})$.

\subsection{The Galois hull $\mathcal{L} | \mathcal{K}$ of an extension $L | K$ of Artinian simple $D$-module algebras}
\label{ssec:extension_Lcal_Kcal}

\begin{lem}
	If ${\bm{u^Q}} = (u^Q_1, \dots, u^Q_n)$ and ${\bm{v^Q}} = (v^Q_1, \dots, v^Q_n)$ are separating transcendence bases of $L/Q$ over $K/(Q \cap K)$ for all $Q \in \Omega(L)$ and $\theta_{{\bm{u}}}$ and $\theta_{{\bm{v}}}$ are the associated $n$-variate iterative derivations on $L$ over $K$ as defined in~\eqref{eq:iterative_derivation_wrt_u}, then there exists an automorphism $\varphi$ of the $L$-algebra $L\llbracket {\bm{w}} \rrbracket$ such that $\theta_{{\bm{v}}} = \varphi \circ \theta_{{\bm{u}}}$ and the iterative differential subalgebras of ${}_{C}\mkern-1mu\mathcal{M}(D,L)$ generated by $ \rho_{} (L)$ and $ \rho_{0} (L)$, once with respect to $\theta_{{\bm{u}}}$, once with respect to $\theta_{{\bm{v}}}$, are equal, i.e.
	\[  \rho_{0} (L) \{  \rho_{} (L) \}_{\theta_{{\bm{u}}}} =  \rho_{0} (L) \{  \rho_{} (L) \}_{\theta_{{\bm{v}}}}. \]
\end{lem}
\begin{proof}
We first show that for each $Q \in \Omega(L)$ there exists an automorphism $\varphi_{Q}$ of the $L/Q$-algebra $(L/Q)\llbracket {\bm{w}} \rrbracket$ such that $\theta_{{\bm{v^Q}}} = \varphi_{Q} \circ \theta_{{\bm{u^Q}}}$.
In fact, the set of formal power series $(\theta_{{\bm{u^Q}}}(v_{i}^Q) - v_{i}^Q)_{i=1,\dots,n}$ has an invertible Jacobian matrix $(\theta_{{\bm{u^Q}}}^{({\bm{\delta}}_j)}(v_{i}^Q))_{i,j=1}^{n}$
and by the formal inverse function theorem (cf. for example \cite[A.4]{Hazewinkel:1978}) there exists a continuous homomorphism of $L/Q$-algebras $\varphi_{Q} \colon (L/Q)\llbracket {\bm{w}} \rrbracket \to (L/Q)\llbracket {\bm{w}} \rrbracket$ such that $\varphi_{Q}( \theta_{{\bm{u^Q}}}(v_{i}^Q) - v_{i}^Q) = w_{i} = \theta_{{\bm{v^Q}}}(v_{i}^Q) - v_{i}^Q$ and thus also $\varphi_{Q}(\theta_{{\bm{u^Q}}}(v_{i}^Q)) = \theta_{{\bm{v}}^Q}(v_{i}^Q)$ for all $i=1,\dots, n$, so that $\varphi_{Q} \circ \theta_{{\bm{u}}^Q}$ and $\theta_{{\bm{v}}^Q}$ coincide on $K(v_{1}^Q, \dots, v_{n}^Q)$.
Since $L/Q$ is a formally étale extension of $(K/(Q \cap K))(v_{1}^Q, \dots, v_{n}^Q)$, they coincide on $L/Q$ as well (cf. \cite[Theorem 27.2]{Matsumura:1989} or \cite[Proposition 1.2.2]{Heiderich:2010}).
The automorphisms $(\varphi_{Q})_{Q \in \Omega(L)}$ induce an automorphism $\varphi$ of the $L$-algebra $L\llbracket {\bm{w}} \rrbracket$ such that $\theta_{{\bm{v}}} = \varphi \circ \theta_{{\bm{u}}}$.
The last claim is a direct consequence of this.
\end{proof}

We define
\begin{align*}
	\mathcal{L} \coloneqq  \rho_{0} (L) \{  \rho_{} (L) \}_{\theta_{{\bm{u}}}}
	\qquad
	\text{and}
	\qquad
	\mathcal{K} \coloneqq  \rho_{0} (L) [  \rho_{} (K) ]
\end{align*}
as the iterative differential subalgebras of $({}_{C}\mkern-1mu\mathcal{M}(D,L), \theta_{{\bm{u}}})$ generated by $ \rho_{0} (L)$ and $ \rho_{} (L)$ and by $ \rho_{0} (L)$ and $ \rho_{} (K)$, respectively.
Both are $D \otimes_C D_{I \mkern-1.9mu D^{n}}$-module subalgebras of $({}_{C}\mkern-1mu\mathcal{M}(D,L),  \rho_{int}  \otimes \theta_{{\bm{u}}})$ and by the previous lemma $\mathcal{L}$ does not depend on the choice of ${\bm{u}}$.\footnote{Originally Umemura defined $\mathcal{L}$ to be a field, but later the definition was changed. The definition we give here coincides with the definition in \cite{Morikawa:2009} when $D = D_{end}$ and $L$ and $K$ are fields.}
We call the extension $\mathcal{L}| \mathcal{K}$ the \emph{Galois hull} of $L | K$.

\begin{ex}
\label{ex:function_with_constant_derivation}
Let $K$ be a field and $L=K(y)$ be a purely transcendental extension field of $K$.
We define an iterative derivation $\theta_{y}$ on $L$ as the $K$-linear homomorphism $\theta_{y} \colon L \to L\llbracket t \rrbracket$ fulfilling $\theta_{y}(y) \coloneqq y + t$ (cf. example~\ref{ex:iterative_derivation_with_respect_to_variables}).
We chose $u=y$ as separating transcendence basis of $L$ over $K$.
Then $\mathcal{L} = L[ K(y+t) ]$ and $\mathcal{K} = L$.
\end{ex}

\begin{ex}
\label{ex:exponential_function}
Let $K$ be a field containing ${\mathbb Q}$ and $L=K(y)$ be a purely transcendental extension field of $K$.
We define an iterative derivation $\theta_{y}$ on $L$ as the $K$-linear homomorphism $\theta_{} \colon L \to L\llbracket t \rrbracket$ fulfilling $\theta_{}(y) \coloneqq y \exp(t)$.
We chose $u=y$ as separating transcendence basis of $L$ over $K$.
Then $\mathcal{L} = L[K(y \exp(t))]$ and $\mathcal{K} = L$.
\end{ex}

\subsection{The Umemura functor{}}
For every commutative $L$-algebra $A$ we consider the tensor product
\[ {}_{C}\mkern-1mu\mathcal{M}(D,L) \otimes_L A\llbracket {\bm{w}} \rrbracket, \]
where the $L$-algebra structures on ${}_{C}\mkern-1mu\mathcal{M}(D,L)$ and on $A \llbracket {\bm{w}} \rrbracket$ are given by $ \rho_{0}  \colon L \to {}_{C}\mkern-1mu\mathcal{M}(D,L)$ (cf. example~\ref{ex_duals_and_GTHs} \ref{ex_duals_and_GTHs_trivial_MAS}) and by $\theta_{{\bm{u}}} \colon L \to L\llbracket {\bm{w}} \rrbracket$ (cf. \eqref{eq:iterative_derivation_wrt_u}), respectively.
This tensor product carries a $D \otimes_{C} D_{I \mkern-1.9mu D^{n}}$-module algebra structure $ \rho_{}  \otimes \theta_{}$,
induced by
\begin{align}
	\xymatrix{
		( {}_{C}\mkern-1mu\mathcal{M}(D,L), {\Psi_{\mkern-0.7mu int}} \otimes \theta_{{\bm{u}}}) & (L, {\Psi_{0}} \otimes \theta_{{\bm{u}}}) \ar[l]_-{ \rho_{0} } \ar[r]^-{\theta_{{\bm{u}}}} & (A \llbracket {\bm{w}} \rrbracket, {\Psi_{0}} \otimes \theta_{{\bm{w}}})
	}
	\label{eqn_lem_module_algebra_structure_on_ModCDL_tensor_Aw_diag1}
\end{align}
using proposition \ref{prop_D-mod_alg_structure_on_tensor_product}.
The homomorphism
\[  \rho_{}  \otimes \theta_{} \colon {}_{C}\mkern-1mu\mathcal{M}(D,L) \otimes_L A\llbracket {\bm{w}} \rrbracket \to {}_{C}\mkern-1mu\mathcal{M}(D \otimes_C D_{I \mkern-1.9mu D^{n}}, {}_{C}\mkern-1mu\mathcal{M}(D,L) \otimes_L A\llbracket {\bm{w}} \rrbracket)\]
is continuous with respect to the $({\bm{w}})$-adic topology on ${}_{C}\mkern-1mu\mathcal{M}(D,L) \otimes_L A\llbracket {\bm{w}} \rrbracket$ and the $({\bm{w}}, {\bm{T}})$-adic topology on
\[{}_{C}\mkern-1mu\mathcal{M}(D,{}_{C}\mkern-1mu\mathcal{M}(D,L) \otimes_L A\llbracket {\bm{w}} \rrbracket)\llbracket {\bm{T}} \rrbracket \cong {}_{C}\mkern-1mu\mathcal{M}(D \otimes_C D_{I \mkern-1.9mu D^{n}}, {}_{C}\mkern-1mu\mathcal{M}(D,L) \otimes_L A\llbracket {\bm{w}} \rrbracket).\]
Therefore, this $D \otimes_{C} D_{I \mkern-1.9mu D^{n}}$-module algebra structure extends to the completion
\[ {}_{C}\mkern-1mu\mathcal{M}(D,L) \hat{\otimes}_L A\llbracket {\bm{w}} \rrbracket \]
with respect to
the $({\bm{w}})$-adic topology by corollary~\ref{cor_extension_of_continous_D-measurings_to_completions}.
The algebras $\mathcal{L} \hat{\otimes}_L A\llbracket {\bm{w}} \rrbracket$ and $\mathcal{K} \hat{\otimes}_L A\llbracket {\bm{w}} \rrbracket$ are $D \otimes_C D_{I \mkern-1.9mu D^{n}}$-module subalgebras.

\begin{defn}
	The \emph{Umemura functor{}} of $L | K$ is the functor
	\[\operatorname{Ume}(L | K) \colon {\mathsf{CAlg}_{L}} \to \mathsf{Grp}, \index{$\operatorname{Ume}(L | K)$} \]
	where for each commutative $L$-algebra $A$ we define $\operatorname{Ume}(L | K)(A)$ to be the group of automorphisms $\varphi$ of the of $D \otimes_C D_{I \mkern-1.9mu D^{n}}$-module algebra $\mathcal{L} \hat{\otimes}_L A\llbracket {\bm{w}} \rrbracket$ that leave $\mathcal{K} \hat{\otimes}_L A\llbracket {\bm{w}} \rrbracket$ fixed and make the diagram
\begin{align*}
	\xymatrixcolsep{4.0pc}
	\xymatrixrowsep{3.0pc}
	\xymatrix{
		\mathcal{L} \hat{\otimes}_L A\llbracket {\bm{w}}\rrbracket  \ar@{->}[d]^{\varphi} \ar@{->}[rd]^{\id_\mathcal{L} \hat{\otimes} \pi_A \llbracket {\bm{w}}\rrbracket } & \\
		\mathcal{L} \hat{\otimes}_L A\llbracket {\bm{w}}\rrbracket  \ar@{->}[r]^(0.4){\id_\mathcal{L} \hat{\otimes} \pi_A \llbracket {\bm{w}}\rrbracket } & \mathcal{L} \hat{\otimes}_L (A/N(A))\llbracket {\bm{w}}\rrbracket,
	}
	\end{align*}
	commutative.
	If $\lambda \colon A \to B$ is a homomorphism of commutative $L$-algebras, we define
	\[ \operatorname{Ume}(L | K)(\lambda) \colon \operatorname{Ume}(L | K)(A) \to \operatorname{Ume}(L | K)(B) \]
	by sending $\varphi \in \operatorname{Ume}(L | K)(A)$ to $\varphi \hat{\otimes}_{A\llbracket {\bm{w}} \rrbracket} \id_{B\llbracket {\bm{w}} \rrbracket}$, where we consider $B\llbracket {\bm{w}} \rrbracket$ as $A\llbracket {\bm{w}} \rrbracket$-algebra via the homomorphism $\lambda\llbracket {\bm{w}} \rrbracket \colon A\llbracket {\bm{w}} \rrbracket \to B\llbracket {\bm{w}} \rrbracket$.
	\label{def_InfGal}
\end{defn}

\subsection{Lie-Ritt functors}
Umemura defines Lie-Ritt functors in \cite{Umemura:1996b}.
Here we use a slightly changed version of them.

\begin{notation}
In this subsection let $L$ be an arbitrary commutative ring and $A$ be a commutative $L$-algebra.
\end{notation}
The \emph{infinitesimal coordinate transformations of $n$ variables over $A$}
\[
	\bm{\Gamma}_{\mkern-3mu nL}(\mkern-1mu A) \coloneqq \{ (\varphi_1, \dots, \varphi_{n}) \in (A\llbracket {\bm{w}}\rrbracket )^n \mid \varphi_i \equiv w_i \mod N(A)\llbracket {\bm{w}}\rrbracket \; \forall \; i \in \{ 1, \dots , n \} \},
\]
where $n \in {\mathbb N}$ and where we denote by ${\bm{w}}$ the tuple $(w_1, \dots, w_n)$, form a group with multiplication given by composition, i.e. if $\Phi = (\varphi_1, \dots, \varphi_n), \Psi \in \bm{\Gamma}_{\mkern-3mu nL}(A)$, then $\Phi \cdot \Psi$ is defined as $(\varphi_1(\Psi), \dots, \varphi_n(\Psi))$ (cf. \cite[Chapitre IV, \S 4.3 and \S 4.7]{Bourbaki:1981:Algebre:Ch4-7}).

We equip the ring $A\llbracket {\bm{w}}\rrbracket$ with the $n$-variate iterative derivation $\theta_{}$ over $A$ with respect to ${\bm{w}}$
(cf. example \ref{ex:iterative_derivation_with_respect_to_variables}) and we extend it to
\[
	A \llbracket {\bm{w}} \rrbracket \{\{{\bm{Y}}\}\}
	\coloneqq A\llbracket w_1, \dots, w_n\rrbracket \llbracket Y_i^{({\bm{k}})} \mid {i \in \{ 1,\dots,n \}, {\bm{k}} \in {\mathbb N}^n} \rrbracket.
\]
with elements $(Y_i^{({\bm{k}})})_{i \in \{1, \dots, n\}, {\bm{k}} \in {\mathbb N}^n}$, algebraically independent over $A\llbracket {\bm{w}} \rrbracket$, by
\[
	\theta^{({\bm{l}})} \left( Y_i^{({\bm{k}})} \right) \coloneqq \binom{{\bm{k}} + {\bm{l}}}{{\bm{k}}} Y_i^{({\bm{k}}+{\bm{l}})}
\]
for all ${\bm{l}} \in {\mathbb N}^n$.
We denote by $A\llbracket {\bm{w}} \rrbracket \{ A\llbracket {\bm{Y}} \rrbracket \}_{\theta_{}}$ the iterative differential subring of $A \llbracket {\bm{w}} \rrbracket \{\{{\bm{Y}}\}\}$ generated by $A\llbracket {\bm{w}}, {\bm{Y}} \rrbracket$, where ${\bm{Y}}$ denotes the $n$-tuple $(Y_{1}^{({\bm{0}})}, \dots, Y_{n}^{({\bm{0}})})$.
For $F \in A\llbracket {\bm{w}} \rrbracket \{ A\llbracket {\bm{Y}} \rrbracket \}_{\theta_{}}$ and $\Phi = (\varphi_1, \dots, \varphi_n) \in \bm{\Gamma}_{\mkern-3mu nL}(A)$ we denote by $F_{|{\bm{Y}} = \Phi}$ the image of $F$ under the homomorphism of $A\llbracket {\bm{w}} \rrbracket$-algebras $A\llbracket {\bm{w}} \rrbracket \{ A\llbracket {\bm{Y}} \rrbracket \}_{\theta_{}} \to A\llbracket {\bm{w}} \rrbracket$ that sends $Y_i^{({\bm{k}})}$ to $\theta^{({\bm{k}})}(\varphi_i)$.

\begin{defn}
A \emph{Lie-Ritt functor} over $L$ is a group functor $G$ on the category of commutative $L$-algebras such that there exists an $n \in {\mathbb N}$ and an ideal $I \unlhd L\llbracket {\bm{w}} \rrbracket \{ L\llbracket {\bm{Y}} \rrbracket \}_{\theta_{}}$ such that $G(A) \cong Z(I)(A)$ for every commutative $L$-algebra $A$, where
\[ Z(I)(A) \coloneqq \{ \Phi \in \bm{\Gamma}_{\mkern-3mu nL}(A) \mid F_{|{\bm{Y}} = \Phi} = 0 \; \text{ for all } \; F \in I \}. \]
\end{defn}
\begin{rmk}
In \cite[Definition 1.8]{Umemura:1996b} Lie-Ritt functors over $L$ are defined using ideals in $L \llbracket {\bm{w}} \rrbracket \{\{ {\bm{Y}} \}\}$ instead of $L\llbracket {\bm{w}} \rrbracket \{ L\llbracket {\bm{Y}} \rrbracket \}_{\theta_{}}$.
Since the term $F_{|{\bm{Y}} = \Phi}$ is not well defined for elements $F \in L \llbracket {\bm{w}} \rrbracket \{\{ {\bm{Y}} \}\}$ in general, we use the above definition instead.
\end{rmk}

\begin{ex}
	\label{ex:additive_formal_group_as_a_Lie-Ritt_functor}
	We define a subgroup functor $G_+$ of $\bm{\Gamma}_{\mkern-3mu 1{\mathbb Z}}$ as
	\[
		G_+(A) \coloneqq \{ a_0 + w \mid a_0 \in N(A) \}
	\]
	for all commutative rings $A$.
	Let $I$ be the ideal in ${\mathbb Z} \llbracket w \rrbracket\{ {\mathbb Z} \llbracket Y \rrbracket \}_{\theta_{}}$ generated by $Y^{(1)} - 1$ and $Y^{(k)}$ for all $k \geq 2$.
	Then $G_+ = Z(I)$, i.e. $G_+$ is a Lie-Ritt functor over ${\mathbb Z}$.
	Furthermore, $G_+$ is isomorphic to the additive formal group scheme $\hat{{\mathbb G}}_a$.
\end{ex}
\begin{proof}
	Let $A$ be a commutative ring.
	An element $\varphi(w) = \sum_{i \geq 0} (a_i + \delta_{i,1}) w^i \in \bm{\Gamma}_{\mkern-3mu 1{\mathbb Z}}(A)$ lies in $Z(I)(A)$ if and only if $1 = \theta_{}^{(1)}(\varphi) = \sum_{i \geq 1} (a_i + \delta_{i,1}) i w^{i-1}$ and for all $k \geq 2$ the equation $0 = \theta_{}^{(k)}(\varphi) = \sum_{i \geq k} \binom{i}{k} (a_i + \delta_{i,1}) w^{i-k}$ holds.
	This is the case if and only if $a_k = 0$ for all $k \geq 1$, i.e. if $\varphi(w) = a_0 + w$ for some $a_0 \in N(A)$.
\end{proof}

\begin{ex}
	\label{ex:multiplicative_formal_group_as_a_Lie-Ritt_functor}
	We define a subgroup functor $G_*$ of $\bm{\Gamma}_{\mkern-3mu 1{\mathbb Z}}$ as
	\[
		G_*(A) \coloneqq \{ (1 + a_1) w \mid a_1 \in N(A) \}
	\]
	for all commutative rings $A$.
	Then $G_*$ is a Lie-Ritt functor over ${\mathbb Z}$ and
	$G_* = Z(I)$, where $I = ( \{ w Y^{(1)} - Y, Y^{(k)} \mid k \geq 2 \})$.
	Furthermore, $G_*$ is isomorphic to the multiplicative formal group scheme ${\hat{\mathbb G}}_m$.
\end{ex}
\begin{proof}
	An element $\varphi(w) = \sum_{i \geq 0} (a_i + \delta_{i,1}) w^i \in \bm{\Gamma}_{\mkern-3mu 1{\mathbb Z}}(A)$ lies in $Z(I)(A)$ if and only if $w \sum_{i \geq 1} i (a_i + \delta_{i,1}) w^{i-1} = \sum_{i \geq 0} (a_i + \delta_{i,1}) w^i$ and $\theta_{}^{(k)}(\sum_{i \geq 0} (a_i + \delta_{i,1}) w^i) = 0$ hold for all $k \geq 2$.
	This is the case if and only if $a_0 = 0$ and $a_k = 0$ for all $k \geq 2$, i.e. if $\varphi(w) \in G_*(A)$.
\end{proof}

\begin{ex}
	\label{ex:multiplicative_formal_group_as_a_Lie-Ritt_functor_for_exponential}
	Let $L$ be a field and $0 \neq y \in L$.
	We define a subgroup functor $\tilde{G}_*$ of $\bm{\Gamma}_{\mkern-3mu 1L}$ as
	\[
		\tilde{G}_*(A) \coloneqq \{ y a_1 + (1 + a_1) w \mid a_0 \in N(A) \}
	\]
	for all commutative $L$-algebras $A$.
	Then $\tilde{G}_*$ is a Lie-Ritt functor over $L$ and
	$\tilde{G}_* = Z(I)$, where
	$I = ( \{ (y+w) Y^{(1)} - Y - y
 , Y^{(k)} \mid k \geq 2 \})$.

	Furthermore, $\tilde{G}_*$ is isomorphic to the Lie-Ritt functor $G_*$ in example \ref{ex:multiplicative_formal_group_as_a_Lie-Ritt_functor} and thus also to the multiplicative formal group scheme ${\hat{\mathbb G}}_m$.
\end{ex}
\begin{proof}
	An element $\varphi(w) = \sum_{i \geq 0} (a_i + \delta_{i,1}) w^i \in \bm{\Gamma}_{\mkern-3mu 1L}(A)$ lies in $Z(I)(A)$ if and only if $(y+w) \sum_{i \geq 1} i (a_i + \delta_{i,1}) w^{i-1} - \sum_{i \in {\mathbb N}} (a_i + \delta_{i,1})w^i - y = 0$ and $\theta_{}^{(k)}(\sum_{i \geq 0} (a_i + \delta_{i,1}) w^i) = 0$ hold for all $k \geq 2$.
	This is the case if and only if
	$a_0 = y a_1$ and $a_k = 0$ for all $k \geq 2$, i.e. if $\varphi(w) \in \tilde{G}_*(A)$.

	An isomorphism $G_* \to \tilde{G}_*$ is given by sending $(1+a_1)w \in G_*(A)$ to  $y a_1 + (1+a_1)w \in \tilde{G}_*(A)$, which is multiplicative since if $(1+b_1)w$ is another element of $G_*(A)$, then the product of $(1+a_1)w$ and $(1+b_1)w$ in $G_*(A)$ is $(1+a_1)(1+b_1)w = (1 + a_1 + b_1 + a_1 b_1)w$, which has as image in $\tilde{G}_*(A)$ the element $y (a_1 + b_1 + a_1 b_1) + (1+a_1 + b_1 + a_1 b_1)w$.
	At the other hand, the product of $y a_1 + (1+a_1)w$ and $y b_1 + (1+b_1)w$ in $\tilde{G}_*(A)$ is
	$y a_1 + (1+a_1) (y b_1 + (1+b_1) w)
	= y(a_1 + b_1 + a_1 b_1) + (1+a_1 +b_1 + a_1 b_1) w$.
\end{proof}

The analogues of example \ref{ex:additive_formal_group_as_a_Lie-Ritt_functor} and \ref{ex:multiplicative_formal_group_as_a_Lie-Ritt_functor} in the setting of Umemura appeared in \cite[Example 1.9(i) and (ii)]{Umemura:1996b}.
Since he assumes ${\mathbb Q} \subseteq L$, it is sufficient for him to use the equation $Y^{(1)} - 1$ and $w Y^{(1)} - Y$ as generators of $I$ in the first and second example, respectively.
In the general case we have to add the equations $Y^{(k)}$ for $k \geq 2$.

\begin{prop}
\label{prop:Lie-Ritt_functors_are_formal_groups}
	Every Lie-Ritt functor over $L$ is isomorphic to a formal group scheme over $L$.
\end{prop}
\begin{proof}
	Let $A$ be a commutative $L$-algebra.
	Then for every $n \in {\mathbb N}$ we have an isomorphism
	\[
		\bm{\Gamma}_{\mkern-3mu nL} \mkern-1mu (\mkern-1.5mu A \mkern-0.5mu) \mkern-1mu
			\to {\widehat{\mathbb A}}^{\{1, \dots, n\} \times {\mathbb N}^n}_{L}\mkern-4mu(\mkern-1.5mu A \mkern-0.5mu),
		\; \left( \sum_{{\bm{k}} \in {\mathbb N}^n} \mkern-5mu (a_{i, {\bm{k}}} \mkern-3mu + \mkern-3mu \delta_{{\bm{k}}, {\bm{\delta}}_{i}}) {\bm{w}}^{\bm{k}} \mkern-3mu \right)_{\mkern-5mu i=1, \dots, n} \mkern-14mu
		\mapsto \mkern-4mu \left( \mkern-1mu a_{i, {\bm{k}}} \mkern-1mu \right)_{(i, {\bm{k}}) \in \{\mkern-1mu 1, \dots, n \mkern-0.5mu\} \mkern-0.5mu \times \mkern-0.5mu {\mathbb N}^n}, \]
	which is natural in $A$.
	It gives rise to an isomorphism from $\bm{\Gamma}_{\mkern-3mu nL}$ to the formal scheme ${\widehat{\mathbb A}}^{\{1, \dots, n\} \times {\mathbb N}^n}_{L}$, which has as $A$-points the set $N(A)^{\{1, \dots,n \} \times {\mathbb N}^{n}}$.

	There exist formal power series $(f_{i,{\bm{l}}})_{i \in \{1,\dots,n\}, {\bm{l}} \in {\mathbb N}^n}$ in variables $u_{j,{\bm{k}}}, v_{j,{\bm{k}}}$ with $j \in \{ 1, \dots, n \}$ and ${\bm{k}} \in {\mathbb N}^n$, coefficients in ${\mathbb Z}$ and constant terms equal to zero defined by
	\[ f_{i,{\bm{l}}}((u_{j,{\bm{k}}}, v_{j,{\bm{k}}})_{(j, {\bm{k}}) \in \{1, \dots, n\} \times {\mathbb N}^n}) \mkern-3mu = \mkern-3mu \sum_{{\bm{k}} \in {\mathbb N}^n} \mkern-3mu v_{i,{\bm{k}}} \mkern-15mu
		\sum_{\genfrac{}{}{0pt}{}{{\bm{l}}_{1,1}, \dots, {\bm{l}}_{1,k_1}, \dots, {\bm{l}}_{n,1}, \dots, {\bm{l}}_{n,k_n} \in {\mathbb N}^n}{\sum_{\mu=1}^n \sum_{\nu = 1}^{k_\mu} {\bm{l}}_{\mu,\nu} = {\bm{l}}}}
		\prod_{\mu=1}^n \prod_{\nu=1}^{k_\mu} u_{\mu, {\bm{l}}_{\mu,\nu}} \]
	such that for all elements
	$\Psi=(\psi_1, \dots, \psi_n)$ and $\Phi = (\varphi_1, \dots, \varphi_n)$ of $\bm{\Gamma}_{\mkern-3mu nL}(A)$ with $\varphi_i = \sum_{{\bm{k}} \in {\mathbb N}^n} a_{i, {\bm{k}}} {\bm{w}}^{\bm{k}}$
	and
	$\psi_i = \sum_{{\bm{k}} \in {\mathbb N}^n} b_{i, {\bm{k}}} {\bm{w}}^{\bm{k}}$ for all $i \in \{1, \dots, n\}$
	we have
	\[\psi_i(\Phi) = \sum_{{\bm{l}} \in {\mathbb N}^n} f_{i,{\bm{l}}}( (a_{j,{\bm{k}}}, b_{j,{\bm{k}}})_{(j, {\bm{k}}) \in \{1, \dots, n\} \times {\mathbb N}^n}) {\bm{w}}^{\bm{l}}.\]
	The formal power series $f_{i,{\bm{j}}}$ have the \emph{monomials have finite support condition} (cf. \cite[Definition 7.2]{Hazewinkel:1979})
	and since the multiplication in $\bm{\Gamma}_{\mkern-3mu nL}(A)$ is associative and unital, with unit the tuple $(w_1, \dots, w_n)$, we see that $f_{i, {\bm{l}}}$ give rise to an (infinite dimensional) formal group law in the sense of Hazewinkel (cf. \cite[Definition 7.5]{Hazewinkel:1979}).
	They also give rise to a morphism
	\[ {\widehat{\mathbb A}}^{\{1, \dots, n\} \times {\mathbb N}^n}_{L} \times {\widehat{\mathbb A}}^{\{1, \dots, n\} \times {\mathbb N}^n}_{L} \to {\widehat{\mathbb A}}^{\{1, \dots, n\} \times {\mathbb N}^n}_{L} \]
	of formal schemes over $L$, which defines a group law on ${\widehat{\mathbb A}}^{\{1, \dots, n\} \times {\mathbb N}^n}_{L}$ such that ${\widehat{\mathbb A}}^{\{1, \dots, n\} \times {\mathbb N}^n}_{L}$ becomes a formal group scheme over $L$, which is isomorphic to the group functor $\bm{\Gamma}_{\mkern-3mu nL}$.

	Let $G$ be an arbitrary Lie-Ritt functor over $L$ and let $I \unlhd L\llbracket{\bm{w}}\rrbracket\{ L\llbracket {\bm{Y}} \rrbracket\}_{\theta_{}}$ be such that $G \cong Z(I)$.
	Let $\Phi = (\varphi_1, \dots, \varphi_n) \in \bm{\Gamma}_{\mkern-3mu nL}(A)$ and $\varphi_i = \sum_{{\bm{k}} \in {\mathbb N}^n} a_{i, {\bm{k}}} {\bm{w}}^{\bm{k}}$ for all $i \in \{1, \dots, n\}$.
	For $h \in I$ the condition $h(\Phi) = 0$ is equivalent to a system of polynomial equations $(h_\lambda)_{\lambda \in \Lambda_h}$ among the coefficients $a_{i,{\bm{k}}}$.
	Thus, $G$ is isomorphic to the closed formal subgroup scheme of ${\widehat{\mathbb A}}^{\{1, \dots, n\} \times {\mathbb N}^n}_{L}$ defined by the polynomials $h_{\lambda}$ for all $h \in I$ and $\lambda \in \Lambda_h$.
\end{proof}

\subsection{The Umemura functor{} as a Lie-Ritt functor}
\begin{notation}
In this subsection we assume that $L | K$ is as in subsection~\ref{ssec:notation}.
\end{notation}

\begin{lem}
\label{lem_D-module_algebra_isomorphism_mu}
	For any commutative $L$-algebra $A$ there exists an injective homomorphism of $D \otimes_C D_{I \mkern-1.9mu D^{n}}$-module algebras\index{$\mu_{A,{\bm{u}}}$}
	\begin{align}
		\label{eqn_lem_D-module_algebra_isomorphism_mu}
		\mu_{A,{\bm{u}}} \colon {}_{C}\mkern-1mu\mathcal{M}(D,L) \hat{\otimes}_L A\llbracket {\bm{w}} \rrbracket & \to {}_{C}\mkern-1mu\mathcal{M}(D,A\llbracket {\bm{w}} \rrbracket) \\
		\sum_{{\bm{i}} \in {\mathbb N}^n} f_{\bm{i}} \otimes a_{\bm{i}} {\bm{w}}^{\bm{i}} & \mapsto \sum_{{\bm{i}} \in {\mathbb N}^n} \theta_{{\bm{u}}}(f_{\bm{i}}) \cdot  \rho_{0} (a_{\bm{i}} {\bm{w}}^{{\bm{i}}}), \nonumber
	\end{align}
	where we consider ${}_{C}\mkern-1mu\mathcal{M}(D, A \llbracket {\bm{w}} \rrbracket)$ as $D$-module algebra via ${\Psi_{\mkern-0.7mu int}}$ and as $D_{I \mkern-1.9mu D^{n}}$-module algebra with the $D_{I \mkern-1.9mu D^{n}}$-module algebra structure induced by $\theta_{{\bm{w}}}$ on $A \llbracket {\bm{w}} \rrbracket$ to ${}_{C}\mkern-1mu\mathcal{M}(D, A \llbracket {\bm{w}} \rrbracket)$ via lemma~\ref{lem_internal_D-module_algebra_structure_on_ModCDA} \ref{itm_lem_internal_D-module_algebra_structure_on_ModCDA_another_MAS}.
\end{lem}
\begin{proof}
	The composition of homomorphisms of $D \otimes_C D_{I \mkern-1.9mu D^{n}}$-module algebras
	\[ \xymatrixcolsep{0.65em} \xymatrix{ {}_{C}\mkern-1mu\mathcal{M}(\mkern-1mu D,L) \mkern-2mu \otimes_L \mkern-3mu A\llbracket {\bm{w}} \rrbracket \ar[rrr]^-{\theta_{{\bm{u}}} \otimes  \rho_{0} }
		& & & {}_{C}\mkern-1mu\mathcal{M}(\mkern-1mu D,L\llbracket {\bm{w}} \rrbracket) \mkern-2mu \otimes_L \mkern-3mu {}_{C}\mkern-1mu\mathcal{M}( \mkern-1mu D, \mkern-1mu A\llbracket {\bm{w}} \rrbracket) \ar[r]^-{m}
		& {}_{C}\mkern-1mu\mathcal{M}( \mkern-1mu D, \mkern-1mu A\llbracket {\bm{w}} \rrbracket), } \]
	is injective by lemma~\ref{lem_simple_linear_disjointness_for_constants}.
	We extend this homomorphism to the completion ${}_{C}\mkern-1mu\mathcal{M}(D,L) \hat{\otimes}_L A\llbracket {\bm{w}} \rrbracket$.
	This extension is again injective, since the inverse image of ${}_{C}\mkern-1mu\mathcal{M}(D, {\bm{w}}^{\bm{k}})$ is $(1 \otimes {\bm{w}}^{\bm{k}})$ and their intersection over all ${\bm{k}} \in {\mathbb N}^n$ is $(0)$.
\end{proof}

The following theorem generalizes \cite[Lemma 5.9]{Umemura:1996b} and \cite[Theorem 2.22]{Morikawa:2009}.
\begin{thm}
\label{thm:InfGal_is_a_Lie-Ritt_functor}
	The functor $\operatorname{Ume}(L | K)$ is a Lie-Ritt functor over $L$.
\end{thm}
\begin{proof}
	For a commutative $L$-algebra $A$ we consider the map
	\begin{align}
	\label{eqn:representation_of_InfGal_as_Lie_Ritt-functor}
	\begin{split}
		\operatorname{Ume}(L | K)(A) \longrightarrow \; & \bm{\Gamma}_{\mkern-3mu nL}(A) \\
		\varphi \longmapsto \; & (\operatorname{ev}_{1_D} \circ \mu_{A,{\bm{u}}} \circ \varphi(  \rho_{} (u_i) \otimes 1 ) - u_i)_{i=1,\dots,n}.
	\end{split}
	\end{align}
	Let $\varphi \in \operatorname{Ume}(L | K)(A)$ and define
	\[ \Phi \coloneqq (\varphi_1, \dots, \varphi_n) \coloneqq (\operatorname{ev}_{1_D} \circ \mu_{A,{\bm{u}}} \circ \varphi( \rho_{} (u_i) \otimes 1) - u_i)_{i=1,\dots,n}. \]
	Then we have
	\[ \operatorname{ev}_{1_D} \circ \mu_{A,{\bm{u}}} \circ \varphi(  \rho_{} (u_i) \otimes 1 ) - u_i
		\equiv w_i \mod N(A)\llbracket {\bm{w}} \rrbracket,
	\]
	so that \eqref{eqn:representation_of_InfGal_as_Lie_Ritt-functor} is well defined.
	We define
	\[
	F \colon L \to A\llbracket {\bm{w}} \rrbracket,
	\quad
	a \mapsto \operatorname{ev}_{1_D} \circ {}_\Phi \theta_{{\bm{u}}} \circ  \rho_{} (a),
	\]
	where ${}_{\Phi}\theta_{{\bm{u}}} \colon L \to A\llbracket {\bm{w}} \rrbracket$ is the composition of $\theta_{{\bm{u}}}$ with the endomorphism of $A\llbracket {\bm{w}} \rrbracket$ sending $w_{i}$ to $\varphi_{i}$ for all $i \in \{1, \dots, n\}$, and
	\[
	G \colon L \to A\llbracket {\bm{w}} \rrbracket,
	\quad
	a \mapsto \operatorname{ev}_{1_D} \circ \mu_{A,{\bm{u}}} \circ \varphi( \rho_{} (a) \otimes 1).
	\]
	Trivially $F$ and $G$ coincide on $K$ and for all $i=1,\dots,n$ we have
	\[
	F(u_i)
		\mkern-1mu = \mkern-1mu \operatorname{ev}_{1_D} \mkern-3mu \circ {}_\Phi \theta_{{\bm{u}}} \circ  \rho_{} (u_i)
		\mkern-1mu = \mkern-1mu u_i + \varphi_i
		\mkern-1mu = \mkern-1mu u_i + \operatorname{ev}_{1_D} \mkern-3mu \circ \mu_{A,{\bm{u}}} \circ \varphi( \rho_{} (u_i) \otimes 1) - u_i
		\mkern-1mu = \mkern-1mu G(u_i).
	\]
	Since $L$ is formally étale over $K(u_{1}, \dots, u_{n})$,
	$F$ and $G$ coincide also on $L$.
	Since $ \rho_{} , \varphi, \mu_{A,{\bm{u}}}$ and ${}_{\Phi} \theta_{{\bm{u}}}$ are homomorphisms of $D$-module algebras, it follows that for all $a \in L$
	\begin{align*}
		(\mu_{A,{\bm{u}}} \circ \varphi (  \rho_{} (a) \otimes 1 ))(d) & = (d.(\mu_{A, {\bm{u}}} \circ \varphi (  \rho_{} (a) \otimes 1)))(1_D) \\
		& = (\mu_{A,{\bm{u}}} \circ \varphi(  \rho_{} (d.a) \otimes 1))(1_D) \\
		& = ( {}_{\Phi}\theta_{{\bm{u}}}( \rho_{} (d.a)) )(1_D) \\
		& = (d.({}_{\Phi}\theta_{{\bm{u}}}( \rho_{} (a))))(1_D) \\
		& = ( {}_{\Phi}\theta_{{\bm{u}}}( \rho_{} (a)))(d),
	\end{align*}
	i.e. $\mu_{A,{\bm{u}}} \circ \varphi (  \rho_{} (a) \otimes 1 ) = {}_{\Phi}\theta_{{\bm{u}}}( \rho_{} (a))$.
	Using that for all $a \in L$
	\[
		\mu_{A,{\bm{u}}} \left( \sum_{{\bm{k}} \in {\mathbb N}^n} \theta_{{\bm{u}}}^{({\bm{k}})}( \rho_{} (a)) \otimes (\Phi-{\bm{w}})^{\bm{k}} \right) = {}_{\Phi}\theta_{{\bm{u}}}( \rho_{} (a))
	\]
	and the injectivity of $\mu_{A,{\bm{u}}}$, this implies
	\begin{align}
		\label{eq:explicit_description_of_infinitesimal_automorphisms_using_infinitesimal_transformations}
		\varphi( \rho_{} (a) \otimes 1) = \sum_{{\bm{k}} \in {\mathbb N}^n} \theta_{{\bm{u}}}^{({\bm{k}})}( \rho_{} (a)) \otimes (\Phi-{\bm{w}})^{\bm{k}}.
	\end{align}

	Next, we show that the natural transformation induced by \eqref{eqn:representation_of_InfGal_as_Lie_Ritt-functor} respects the group structures.
	To this end, let $\psi \in \operatorname{Ume}(L | K)(A)$ and let $\Psi$ be the image of it under \eqref{eqn:representation_of_InfGal_as_Lie_Ritt-functor}.
	By equation~\eqref{eq:explicit_description_of_infinitesimal_automorphisms_using_infinitesimal_transformations} we obtain, using the abbreviations $\Phi' = \Phi - {\bm{w}}$ and $\Psi' = \Psi - {\bm{w}}$, that for all $a \in L$
\begin{align*}
	\varphi \mkern-2mu \circ \mkern-2mu \psi (  \rho_{} (a) \mkern-3mu \otimes \mkern-3mu  1 )
		& \mkern-4mu = \varphi \left( \sum_{{\bm{k}} \in {\mathbb N}^n} \theta_{{\bm{u}}}^{({\bm{k}})}( \rho_{} (a)) \otimes (\Psi-{\bm{w}})^{\bm{k}} \right) \\
                & \mkern-4mu = \mkern-3mu \sum_{{\bm{k}} \in {\mathbb N}^n} \mkern-1mu \theta^{({\bm{k}})} \left( \sum_{{\bm{l}} \in {\mathbb N}^n} \theta_{{\bm{u}}}^{({\bm{l}})}( \rho_{} (a)) \otimes {\Phi'}^{\bm{l}} \right) (1 \otimes {\Psi'}^{\bm{k}}) \\
		& \mkern-4mu = \mkern-5mu \sum_{{\bm{k_1}}\mkern-3mu,{\bm{k_2}}\mkern-2mu, {\bm{l}} \in {\mathbb N}^n} \mkern-9mu \left( \binom{ {\bm{k_1 + l}}}{{\bm{l}}} \mkern-1mu \theta_{{\bm{u}}}^{(\mkern-1mu {\bm{k_1 + l)}}} ( \rho_{} ( a ) ) \otimes \theta_{{\bm{w}}}^{({\bm{k_2}} )} ( {\Phi'}^{\bm{l}} ) \right) (1 \otimes {\Psi'}^{{\bm{k_1}} + {\bm{k_2}}} ) \\
                & \mkern-4mu = \mkern-3mu \sum_{{\bm{m}},{\bm{k_2}}, {\bm{l}} \in {\mathbb N}^n} \binom{{\bm{m}}}{{\bm{l}}} \theta_{{\bm{u}}}^{({\bm{m}})}( \rho_{} (a)) \otimes \theta_{{\bm{w}}}^{({\bm{k_2}})}({\Phi'}^{\bm{l}}) {\Psi'}^{{\bm{m - l + k_2}}}) \\
                & = \sum_{{\bm{m}}, {\bm{l}} \in {\mathbb N}^n} \theta_{{\bm{u}}}^{({\bm{m}})}( \rho_{} (a)) \otimes \binom{{\bm{m}}}{{\bm{l}}} \Psi'^{{\bm{m}} - {\bm{l}}} (\Phi'(\Psi))^{{\bm{l}}} \\
		& = \sum_{{\bm{m}} \in {\mathbb N}^m} \theta_{{\bm{u}}}^{({\bm{m}})}( \rho_{} (a)) \otimes (\Phi(\Psi) - {\bm{w}})^{\bm{m}}.
\end{align*}
It follows that the image of $\varphi \circ \psi$ in $\bm{\Gamma}_{\mkern-3mu nL}(A)$ is $\Phi(\Psi)$.

It remains to show that the image of the natural transformation induced by \eqref{eqn:representation_of_InfGal_as_Lie_Ritt-functor} is of the form $Z(I)$ for some ideal $I$ of $L\llbracket {\bm{w}} \rrbracket \{ L\llbracket {\bm{Y}} \rrbracket \}_{\theta_{}}$.
For all $m \in {\mathbb N}$, elements $a_1, \dots a_m \in L$, iterative differential polynomials $F(X_{a_1}, \dots, X_{a_m}) \in \mathcal{K} \{ X_{a_1}, \dots, X_{a_m} \}_{D_{I \mkern-1.9mu D^{n}}}$ with coefficients in $\mathcal{K}$ and in iterative differential variables $X_{a_1}, \dots X_{a_m}$ fulfilling $F( \rho_{} (a_1), \dots,  \rho_{} (a_m)) = 0$ (with respect to the iterative derivation $\theta_{{\bm{u}}}$) and for all $d \in D$ we define
\[
F_d \coloneqq
	F^{\theta_{{\bm{u}}}}(
		{}_{{\bm{Y}}}\theta_{{\bm{u}}} (  \rho_{} ( a_1 ) ),
		\dots,
		{}_{{\bm{Y}}}\theta_{{\bm{u}}} (  \rho_{} ( a_m ) )
	)
\]
as an element of $L\llbracket {\bm{w}} \rrbracket \{ L \llbracket {\bm{Y}} \rrbracket \}_{\theta_{}}$,
where ${\bm{Y}}$ is considered as an $n$-tuple of iterative differential variables and $F^{\theta_{{\bm{u}}}}$ denotes the differential polynomial in $X_{a_1}, \dots, X_{a_m}$ obtained from $F$ by applying $\theta_{{\bm{u}}}$ to its coefficients.
Let $I$ be the ideal generated by all such $F_d$.
Then the image of \eqref{eqn:representation_of_InfGal_as_Lie_Ritt-functor} is equal to $Z(I)$.
\end{proof}

\begin{cor}
	The functor $\operatorname{Ume}(L | K)$ is a formal group scheme over $L$.
\end{cor}
\begin{proof}
	This follows from theorem~\ref{thm:InfGal_is_a_Lie-Ritt_functor} and proposition~\ref{prop:Lie-Ritt_functors_are_formal_groups}.
\end{proof}

\begin{ex}
	In the situation of example~\ref{ex:function_with_constant_derivation} the Umemura functor $\operatorname{Ume}(L | K)$ is isomorphic to the functor $G_+$ (cf. example~\ref{ex:additive_formal_group_as_a_Lie-Ritt_functor}) and thus to the additive formal group scheme ${\hat{\mathbb G}}_a$.
\end{ex}
\begin{proof}
Let $A$ be a commutative $L$-algebra.
Then every $\varphi \in \operatorname{Ume}(L | K)(A)$ is determined by $\varphi( \rho_{} (y) \otimes 1)$.
The relation
$\theta_{y}^{(1)}( \rho_{} (y))
= 1$
implies
\begin{equation}
\label{eqn:first_derivative}
	\theta_{}^{(1)}( \varphi( \rho_{} (y) \otimes 1)) = 0
\end{equation}
where $\theta_{}$ denotes the iterative derivation on $\mathcal{L} \hat{\otimes}_L A\llbracket w \rrbracket$ induced by $\theta_{u} (= \theta_{y})$ on $\mathcal{L}$ and $\theta_{w}$ on $A\llbracket w \rrbracket$.
Since for all $k \geq 2$ we have $\theta_{y}^{(k)}( \rho_{} (y)) = 0$, it follows
\begin{equation}
\label{eqn:higher_derivatives}
\theta_{}^{(k)}( \varphi( \rho_{} (y) \otimes 1) )
= 0.
\end{equation}
By theorem~\ref{thm:InfGal_is_a_Lie-Ritt_functor}, there exists a $\Phi \in \bm{\Gamma}_{\mkern-3mu 1L}(A)$ such that
$\varphi( \rho_{} (y) \otimes 1) = \sum_{k \in {\mathbb N}} \theta_{y}^{(k)}( \rho_{} (y)) \otimes (\Phi - w)^k$,
cf. equation~\eqref{eq:explicit_description_of_infinitesimal_automorphisms_using_infinitesimal_transformations}.
In our situation this becomes
\[
\varphi( \rho_{} (y) \otimes 1)
= \sum_{k \in {\mathbb N}} \theta_{y}^{(k)}( \rho_{} (y)) \otimes (\Phi - w)^k
=  \rho_{} (y) \otimes 1 + 1 \otimes (\Phi - w).
\]
Therefore, if we write $\Phi = \sum_{i \in {\mathbb N}} a_i w^i$, then
\[
\varphi( \rho_{} (y) \otimes 1)
= (y+t) \otimes 1 + 1 \otimes (\sum_{i \in {\mathbb N}} a_i w^i - w).
\]
From \eqref{eqn:first_derivative} and \eqref{eqn:higher_derivatives} we obtain $\sum_{i \geq 1} a_i i w^{i-1} = 1$ and $\sum_{i \in {\mathbb N}} a_i \theta_{w}^{(k)}(w^i) = 0$,
which implies $a_1$=1 and $a_k = 0$ for all $k \geq 2$.
Hence $\varphi( \rho_{} (y) \otimes 1) =  \rho_{} (y) \otimes 1 + 1 \otimes a_0$ and $\Phi = a_0 + w$ with $a_0 \in N(A)$.
The elements $\Phi \in \bm{\Gamma}_{\mkern-3mu 1L}(A)$ of this form
are the elements of $G_+(A)$.
Conversely, the automorphisms $\varphi$ of the form $\varphi( \rho_{} (y) \otimes 1) =  \rho_{} (y) \otimes 1 + 1 \otimes a_0$ belong to $\operatorname{Ume}(L | K)(A)$.
\end{proof}

\begin{ex}
	In the situation of example~\ref{ex:exponential_function} the Umemura functor $\operatorname{Ume}(L | K)$ is isomorphic to the functor $\tilde{G}_*$ (cf. example~\ref{ex:multiplicative_formal_group_as_a_Lie-Ritt_functor_for_exponential}) and thus to the multiplicative formal group scheme ${\hat{\mathbb G}}_m$.
\end{ex}
\begin{proof}
Let $A$ be a commutative $L$-algebra.
Then every $\varphi \in \operatorname{Ume}(L | K)(A)$ is determined by $\varphi( \rho_{} (y) \otimes 1)$.
The relations
$ \rho_{} (y) = y \theta_{y}^{(1)}( \rho_{} (y))$
and
$\theta_{y}^{(k)}( \rho_{} (y)) = 0$
imply
\[ \varphi( \rho_{} (y) \otimes 1) = (y \otimes 1) \theta_{}^{(1)}(\varphi( \rho_{} (y) \otimes 1))
\quad
\text{and}
\quad
\theta_{}^{(k)}(\varphi( \rho_{} (y) \otimes 1)) = 0
\]
for all $k \geq 2$.
By theorem~\ref{thm:InfGal_is_a_Lie-Ritt_functor} there exists a $\Phi \in \bm{\Gamma}_{\mkern-3mu 1}(A)$ such that
\[
\varphi( \rho_{} (y) \otimes 1) = \sum_{k \in {\mathbb N}} \theta_{y}^{(k)}( \rho_{} (y)) \otimes (\Phi - w)^k ,
\]
cf. equation~\eqref{eq:explicit_description_of_infinitesimal_automorphisms_using_infinitesimal_transformations}, which becomes
\[
\varphi( \rho_{} (y) \otimes 1) =  \rho_{} (y) \otimes 1 +  \rho_{} (y)/y \otimes (\Phi - w)
\]
here.
It follows
\begin{align*}
 \rho_{} (y) \otimes 1 +  \rho_{} (y)/y \otimes (\Phi - w)
&= (y \otimes 1) ( \theta_{}^{(1)}( \rho_{} (y) \otimes 1 +  \rho_{} (y)/y \otimes (\Phi - w)) ) \\
&=  \rho_{} (y) \otimes 1 +  \rho_{} (y) \otimes (\theta_{}^{(1)}(\Phi) - 1))
\end{align*}
and therefore
\begin{align*}
\Phi = (y+w)\theta_{}^{(1)}(\Phi) - y.
\end{align*}
\end{proof}

\section{Picard-Vessiot extensions of Artinian simple module algebras}
\label{sec:PV_extensions_of_Artinian_simple_module_algebras}
Amano and Masuoka unified the Picard-Vessiot theories of differential equations and difference equations using Artinian simple $D$-module algebras (cf. \cite{AmanoMasuoka:2005}), where $D$ is a bialgebra over a field $C$.
They restrict themselves to the case where the bialgebra $D$ is a pointed cocommutative Hopf algebra such that its irreducible component of $1$ is of Birkhoff-Witt type.
This excludes bialgebras such as $D_{end}$ (cf. example~\ref{ex:module_algebra_structures}~\ref{itm:ex:module_algebra_structures:endomorphisms}).
Here we sketch how their definitions and some of their results can be generalized to include this case as well.

\begin{notation}
	Let $C$ be a field, $G$ be a monoid and $D^1$ be a pointed irreducible cocommutative Hopf algebra of Birkhoff-Witt type over $C$ such that $D^1$ is a $C G$-module algebra (cf. example \ref{ex:module_algebra_structures}~\ref{itm:ex:module_algebra_structures:group_like_bialgebras_over_a_monoid}).
	We define $D \coloneqq D^1 \# C G$.
\end{notation}

\begin{defn}
	An extension of Artinian simple commutative $D$-module algebras $(L,  \rho_{L} ) | (K,  \rho_{K} )$ is \emph{Picard-Vessiot} if the following hold:
	\begin{enumerate}
		\item The elements $g \in G$ operate as injective endomorphisms on $L$.
		\item The constants $L^{ \rho_{L} }$ of $L$ coincide with the constants $K^{ \rho_{K} }$ of $K$.
		\item There exists an intermediate $D$-module algebra $(R,  \rho_{R} )$ of $K \subseteq L$ such that the total ring of fractions $\TotalQuot(R)$ of $R$ is equal to $L$ and such that the $K^{ \rho_{K} }$-subalgebra \[H \coloneqq (R \otimes_K R)^{ \rho_{R}  \otimes  \rho_{R} }\] of $R \otimes_K R$ generates $R \otimes_K R$ as left $R$-algebra, i.e.
		\[
			R \cdot H = R \otimes_K R.
		\]
		\label{itm_defn_Picard-Vessiot_extension_torsor_property}
	\end{enumerate}
	The Picard-Vessiot extension $L | K$ is called \emph{finitely generated} if it is finitely generated as extension of Artinian simple $D$-module algebras.
	\label{defn_Picard-Vessiot_extension}
\end{defn}

\begin{prop}
	\label{prop:fundamental_isomorphism_of_Picard-Vessiot_theory}
	Let $L | K$ be a Picard-Vessiot extension of Artinian simple commutative $D$-module algebras with constants $k \coloneqq L^{ \rho_{} }$ and $(R,  \rho_{R} )$ and $H$ be as in definition \ref{defn_Picard-Vessiot_extension}.
	Then the following hold:
	\begin{enumerate}
		\item The homomorphism
			\begin{equation}
				\mu \colon (R \otimes_k H,  \rho_{R}  \otimes  \rho_{0} ) \to (R \otimes_K R,  \rho_{R}  \otimes  \rho_{R} ), \quad a \otimes h \mapsto (a \otimes 1) \cdot h
				\label{eqn:prop:fundamental_isomorphism_of_Picard-Vessiot_theory}
			\end{equation}
			is an isomorphism of $D$-module algebras.
		\item The $k$-algebra $H$ carries a Hopf algebra structure induced by the $R$-coalgebra structure on $R \otimes_K R$, given by the counit
		\[  \varepsilon \colon R \otimes_K R \to R, \quad a \otimes b \mapsto ab \]
		and the comultiplication
		\[ \Delta \colon R \otimes_K R \to (R \otimes_K R) \otimes_R (R \otimes_K R), \quad a \otimes b \mapsto (a \otimes 1) \otimes (1 \otimes b). \]
		The antipode $S$ of $H$ is induced by the homomorphism
		\[ \tau \colon R \otimes_K R \to R \otimes_K R, \quad a \otimes b \mapsto b \otimes a. \]
		\item \label{itm:prop:fundamental_isomorphism_of_Picard-Vessiot_theory:uniqueness_of_principal_module_algebra} The intermediate $D$-module algebra $(R,  \rho_{R} )$ satisfying condition \ref{itm_defn_Picard-Vessiot_extension_torsor_property} in definition \ref{defn_Picard-Vessiot_extension} is unique.
	\end{enumerate}
\end{prop}
\begin{proof}
	This can be proven essentially as \cite[Proposition 3.4]{AmanoMasuoka:2005}, only the proof of \ref{itm:prop:fundamental_isomorphism_of_Picard-Vessiot_theory:uniqueness_of_principal_module_algebra} needing a slight adaption:

	If $R_1$ and $R_2$ satisfy condition \ref{itm_defn_Picard-Vessiot_extension_torsor_property} in definition \ref{defn_Picard-Vessiot_extension}, then $R_1 R_2$ satisfies it too.
	Therefore we can assume that $R_1$ is included in $R_2$.
	We define
	$H_1 \coloneqq (R_1 \otimes_K R_1)^{ \rho_{R_1}  \otimes  \rho_{R_1} }$
	and
	$H_2 \coloneqq (R_2 \otimes_K R_2)^{ \rho_{R_2}  \otimes  \rho_{R_2} }$.
	Then $H_1$ is a Hopf subalgebra of $H_2$ and thus $H_2$ is faithfully flat over $H_1$ by \cite[Chapter 14]{Waterhouse:1979}.
	Therefore $L \otimes_k H_2$ is faithfully flat over $L \otimes_k H_1$.
	By proposition~\ref{prop:equivalent_conditions_to_Artinian_for_a_Noetherian_simple_D-module_algebra_with_injective_group_like_elements}, the $K$-module $L$ is projective and so in particular flat.
	By \cite[Chapitre 1, \S 3.5, Proposition 9]{Bourbaki:1985:AlgebreCommutative:Ch1-4}, $L$ is also faithfully flat over $K$ and therefore $L \otimes_K R_1$ is faithfully flat over $R_1$ (cf. \cite[Chapitre 1, \S 3.3, Proposition 5]{Bourbaki:1985:AlgebreCommutative:Ch1-4}).
	Since by the above $L \otimes_K R_2 \cong L \otimes_k H_2$ is faithfully flat over $L \otimes_K R_1 \cong L \otimes_k H_1$, $L \otimes_K R_2$ is also faithfully flat over $R_1$ (cf. \cite[Chapitre I, \S 3.4, Proposition 7]{Bourbaki:1985:AlgebreCommutative:Ch1-4}) and hence $R_2$ is faithfully flat over $R_1$.
	Thus for each $r_1 \in R_1$ we have $r_1 R_1 = r_1 R_2 \cap R_1$.
	For any $r_2 \in R_2$ there exists a non-zero divisor $r_1 \in R_1$ such that $r_1 r_2 \in R_1$, since $R_2 \subseteq L = \TotalQuot(R_1)$, and so $r_1 r_2 \in r_1 R_2 \cap R_1 = r_1 R_1$.
	Therefore $r_2 \in R_1$, i.e. $R_2 = R_1$.
\end{proof}

\begin{defn}
	If $L| K$ is a Picard-Vessiot extension of Artinian simple commutative $D$-module algebras, then $R$ and $H$ in definition~\ref{defn_Picard-Vessiot_extension} are called the \emph{principal $D$-module algebra}\index{Picard-Vessiot extension!principal module algebra}\footnote{In Picard-Vessiot theory of differential and difference equations, this algebra is usually called \emph{Picard-Vessiot ring}.} and the \emph{Hopf algebra} of $L| K$, respectively.
	If we want to indicate $R$ and $H$, we denote the Picard-Vessiot extension $L | K$ also by $(L | K, R, H)$.
\end{defn}

\begin{defn}
	If $(L | K, R, H)$ is a Picard-Vessiot extension of Artinian simple commutative $D$-module algebras, then we define the \emph{Galois group scheme} \index{Galois group scheme}\index{Picard-Vessiot extension!Galois group scheme} $\operatorname{Gal}(L|K)$\index{$\operatorname{Gal}(L|K)$} of $L | K$ to be the affine group scheme $\operatorname{Spec}H$ over $L^{ \rho_{} }$.
	\label{defn_Galois_group_scheme_of_Picard-Vessiot_extension}
\end{defn}

\begin{prop}
	Let $(L | K, R, H)$ be a Picard-Vessiot extension of Artinian simple commutative $D$-module algebras with constants $k \coloneqq L^{ \rho_{} }$.
	Then for any commutative $k$-algebra $A$ the $A$-points of $\operatorname{Gal}(L|K) = \operatorname{Spec}H$ are isomorphic to the group of automorphisms of the $D$-module algebra $(R \otimes_k A,  \rho_{}  \otimes  \rho_{0} )$ that leave $K \otimes_k A$ fixed.
	\label{rmk_Galois_group_scheme_of_Picard-Vessiot_extensions}
\end{prop}
\begin{proof}
	This can be proven as \cite[Theorem 3.6.1]{Amano:2005}.
\end{proof}

\begin{prop}
\label{prop:properties_of_principal_module_algebra_of_a_Picard-Vessiot_extension_of_AS_module_algebras}
  Let $(L | K, R, H)$ be a Picard-Vessiot extension of Artinian simple commutative $D$-module algebras.
  \begin{enumerate}
    \item Then $R$ is a simple $D$-module algebra.
    \item The ring $R$ is isomorphic to $\prod_{P \in \Omega(L)} R/(P \cap R)$.
  \end{enumerate}
\end{prop}
\begin{proof}
\begin{enumerate}
  \item This can be proven as \cite[Corollary 3.12]{AmanoMasuoka:2005} (cf. also \cite[Proposition 3.5.9]{Amano:2005}).
  In its proof \cite[Proposition 3.10 (i)]{AmanoMasuoka:2005}
  is used, which also holds in our generalized setting.
  \item The ring $L$ is a localization of $R$ and thus $\Omega(L) \to \Omega(R), P \mapsto P \cap R$ is an injection.
    We obtain an injective homomorphism
    \begin{equation}
    \label{eqn:prop:properties_of_principal_module_algebra_of_a_Picard-Vessiot_extension_of_AS_module_algebras:homomorphism_to_product_of_quotients}
    R \to \prod_{P \in \Omega(L)} R/(P \cap R).
    \end{equation}
    If for minimal prime ideals $P, Q \in \Omega(L)$, the ideal $J \coloneqq (P \cap R) + (Q \cap R)$ would be strictly smaller than $R$, then as in the proof of \ref{prop:statements_on_Noetherian_simple_D-module_algebras} \ref{minimal_prime_ideals_D1_stable}, using that $R$ is simple as a $D$-module algebra, one can show that $J = P \cap R = Q \cap R$, i.e. that $P=Q$.
    Therefore \eqref{eqn:prop:properties_of_principal_module_algebra_of_a_Picard-Vessiot_extension_of_AS_module_algebras:homomorphism_to_product_of_quotients} is an isomorphism.
\end{enumerate}
\end{proof}

\begin{prop}
\label{prop:quotients_of_PV_extensions}
Let $(L | K, R, H)$ be a Picard-Vessiot extension of Artinian simple commutative $D$-module algebras.
Let further $\Omega(K) = \{ P_1, \dots, P_{m} \}$ and for all $i \in \{1, \dots, m\}$ let $Q_{i,1}, \dots, Q_{i,r_i}$ be the minimal prime ideals of $L$ lying over $P_i$.
We define
$L_i \coloneqq \prod_{j=1}^{r_i} L/Q_{i,j}$,
$R_i \coloneqq \prod_{j=1}^{r_i} R/(Q_{i,j} \cap R)$,
$K_i \coloneqq K/P_i$,
$L_{i,j} \coloneqq L/Q_{i,j}$
and $R_{i,j} \coloneqq R/(Q_{i,j} \cap R)$ for all $j \in \{1, \dots, r_i\}$.
We denote by $k$ the field of constants of $L$.
Then $(L_i | K_i, R_i, H)$ is a Picard-Vessiot extension of Artinian simple $D(G_{P_i})$-module algebras for all $i \in \{1, \dots, m\}$.
\end{prop}
\begin{proof}
  For each $i \in \{1, \dots, m\}$ we denote by $\pi_i \colon R \to R_i$ the canonical homomorphism.
  Since $L$ has Krull dimension $0$, we see that $Q_{i,j_1} + Q_{i,j_2} = L$ and thus $(Q_{i,j_1} \cap R) + (Q_{i,j_2} \cap R) = R$ for all $j_1 \neq j_2 \in \{1, \dots, r_i\}$.
  Hence $\pi_i$ is surjective.
  If we denote by $\mu_i$ the homomorphism $R_i \otimes_k H \to R_i \otimes_{K_i} R_i, r_i \otimes h \mapsto (r_i \otimes 1) h$, then the diagram
  \begin{align*}
    \xymatrix{R \otimes_k H \ar[d]^{\pi_i \otimes \id} \ar[r]^\mu & R \otimes_K R \ar[d]^{\pi_i \otimes \pi_i} \\
    R_i \otimes_{k} H \ar[r]^{\mu_i} & R_i \otimes_{K_i} R_i
    }
  \end{align*}
  commutes and since $\mu$ is an isomorphism, $\mu_i$ is an isomorphism too.
  We note that also $L_i^{D(G_{P_i})} = L^D = K^D = K_i^{D(G_{P_i})}$.
  Finally, also $\TotalQuot(R_i) = L_i$, since if $S$ is the set of non-zero divisors of $R$, then $(S^{-1}R)/S^{-1}(Q_{i,j} \cap R) \cong \bar{S}^{-1}(R/(Q_{i,j} \cap R))$, where $\bar{S}$ is the image of $S$ under $R \to R/(Q_{i,j} \cap R)$.
  We note that $S^{-1}(Q_{i,j} \cap R) = Q_{i,j}$ and that $\bar{S}$ is the set of non-zero divisors of $R_{i,j}$.
  Therefore $\TotalQuot(R_{i,j}) \cong L_{i,j}$ and thus $\TotalQuot(R_i) \cong L_i$.
\end{proof}

\begin{prop}
\label{prop:smoothness_of_Picard-Vessiot_extensions}
	Let $(L | K, R, H)$ be a Picard-Vessiot extension of Artinian simple commutative $D$-module algebras such that $R$ is a finitely generated $K$-algebra and $H$ is a finitely generated $k$-algebra.
	Then with the notation of proposition~\ref{prop:quotients_of_PV_extensions} we have:
	\begin{enumerate}
	\item
	\label{prop:smoothness_for_PV:conditions_for_product}
	The following conditions are equivalent:
	\begin{enumerate}
		\item \label{prop:smoothness_for_PV:R_smooth_over_K} $R$ is smooth over $K$,
		\item \label{prop:smoothness_for_PV:R1_smooth_over_K1} $R_i$ is smooth over $K_i$ for all $i \in \{1, \dots, m\}$,
		\item \label{prop:smoothness_for_PV:H_smooth_over_k} $H$ is smooth over $k$ and
		\item \label{prop:smoothness_for_PV:GalLK_smooth_over_k} $\operatorname{Gal}(L|K)$ is smooth over $k$.
	\end{enumerate}
	\item
	\label{prop:smoothness_for_PV:conditions_for_factors}
	The following conditions are also equivalent:
	\begin{enumerate}[resume]
		\item \label{prop:smoothness_for_PV:L1_separable_over_K1} $L_i$ is separable over $K_i$ for all $i \in \{1, \dots, m\}$,
		\item \label{prop:smoothness_for_PV:R1_separable_over_K1} $R_i$ is separable over $K_i$ for all $i \in \{1, \dots, m\}$ and
		\item \label{prop:smoothness_for_PV:H_otimes_k_bar_reduced} $\bar{k} \otimes_k H$ is reduced.
	\end{enumerate}
	\item
	The equivalent conditions
	in \ref{prop:smoothness_for_PV:conditions_for_product} imply those of \ref{prop:smoothness_for_PV:conditions_for_factors}.
	If $k$ is perfect, then the converse holds and these conditions are further equivalent to the following:
	\begin{enumerate}[resume]
		\item \label{prop:smoothness_for_PV:H_reduced} $H$ is reduced and
		\item \label{prop:smoothness_for_PV:R_otimesK_R_reduced} $R \otimes_K R$ is reduced.
	\end{enumerate}
	\end{enumerate}
\end{prop}
\begin{proof}

We first prove \ref{prop:smoothness_for_PV:GalLK_smooth_over_k}
$\Leftrightarrow$ \ref{prop:smoothness_for_PV:H_smooth_over_k}
$\Leftrightarrow$ \ref{prop:smoothness_for_PV:R_smooth_over_K}
$\Leftrightarrow$ \ref{prop:smoothness_for_PV:R1_smooth_over_K1}
$\Rightarrow$ \ref{prop:smoothness_for_PV:L1_separable_over_K1}
$\Leftrightarrow$ \ref{prop:smoothness_for_PV:R1_separable_over_K1}
$\Leftrightarrow$ \ref{prop:smoothness_for_PV:H_otimes_k_bar_reduced}:

The equivalence of \ref{prop:smoothness_for_PV:H_smooth_over_k} and \ref{prop:smoothness_for_PV:GalLK_smooth_over_k} is clear.

By \cite[Chapitre IV, 17.7.3 (ii)]{EGA4-4_IHES:1967}, $R \otimes_K R$ is smooth over $R$ if and only if $R$ is smooth over $K$.
In the same way $H$ is smooth over $k$ if and only if $R \otimes_k H$ is smooth over $R$.
Since $R \otimes_K R \cong R \otimes_k H$, we conclude that $R$ is smooth over $K$ if and only if $H$ is smooth over $k$.

Since smoothness is a local property (cf. \cite[Chapitre IV, 17.3.2 (iii)]{EGA4-4_IHES:1967}) and since $R \cong \prod_{i=1}^m R_i$ and $K \cong \prod_{i=1}^m K_i$, $R$ is smooth over $K$ if and only if $R_{i}$ is smooth over $K_{i}$ for all $i \in \{1, \dots, m \}$.

If $R_i$ is smooth over $K_i$, then $L_i$ is formally smooth over $K_i$, since as a localization $L_i$ is formally étale over $R_i$.
By \cite[Chapitre IV, 17.1.6]{EGA4-4_IHES:1967} we see that $L_i$ is formally smooth over $K_i$ if and only if $L_{i,j}$ is formally smooth over $K_i$ for all $j \in \{1, \dots, r_i\}$.
By the theorem of Cohen (cf. \cite[Chapitre 0, 19.6.1]{EGA4-1_IHES:1964}) $L_{i,j}$ is formally smooth over $K_i$ if and only if it is separable over $K_i$.
The product $L_i = \prod_{j=1}^{r_i} L_{i,j}$ is separable over $K_i$ if and only if each $L_{i,j}$ is separable over $K_i$.
Thus $L_i$ is formally smooth over $K_i$ if and only if it is separable over $K_i$.

If $L_i$ is separable over $K_i$, then obviously $R_i$ is also separable over $K_i$ and the converse holds too (cf. \cite[Chapitre V, \S 15.2, Proposition 4]{Bourbaki:1981:Algebre:Ch4-7}).

By proposition~\ref{prop:quotients_of_PV_extensions}, $(L_i | K_i, R_i, H)$ is a Picard-Vessiot extension of $D(G_{P_i})$-module algebras for all $i \in \{ 1, \dots, m \}$.
Therefore we have an isomorphism
\[
	R_{i} \otimes_{K_{i}} R_{i} \cong R_{i} \otimes_{k} H.
\]
Let $\mathfrak{m}_i$ be a maximal ideal of $R_{i}$.
Then $R_{i} / \mathfrak{m}_i$ can be embedded into an algebraic closure $\bar{K_{i}}$ of $K_i$ by Hilbert's Nullstellensatz.
We extend the induced isomorphism
\[
	R_{i}/\mathfrak{m}_i \otimes_{K_{i}} R_{i} \cong R_{i}/\mathfrak{m}_i \otimes_{k} H.
\]
to an isomorphism
\[
	\overline{K_{i}} \otimes_{K_{i}} R_{i} \cong \overline{K_{i}} \otimes_{k} H.
\]
Since $R_{i}$ is separable over $K_{i}$ if and only if $\overline{K_{i}} \otimes_{K_{i}} R_{i}$ is reduced, and since $\overline{K_{i}} \otimes_{k} H$ is reduced if and only if $\overline{k} \otimes_{k} H$ is reduced, the equivalence of \ref{prop:smoothness_for_PV:R1_separable_over_K1} and \ref{prop:smoothness_for_PV:H_otimes_k_bar_reduced}
 follows.

We now prove that \ref{prop:smoothness_for_PV:H_smooth_over_k}
$\Leftrightarrow$ \ref{prop:smoothness_for_PV:H_otimes_k_bar_reduced}
$\Leftrightarrow$ \ref{prop:smoothness_for_PV:H_reduced}
$\Leftrightarrow$ \ref{prop:smoothness_for_PV:R_otimesK_R_reduced} if $k$ is perfect:

Since $k$ is perfect, the affine group scheme $\operatorname{Gal}(L|K) = \operatorname{Spec}H$ is smooth over $k$ if and only if
$\bar{k} \otimes_{k} H$ is reduced (cf. \cite[11.6]{Waterhouse:1979})
if and only if $H$ is reduced (cf. \cite[Chapitre V, \S 15.2, Proposition 5]{Bourbaki:1981:Algebre:Ch4-7}).
Since $R$ is reduced, $H$ is reduced if and only if $R \otimes_k H \cong R \otimes_K R$ is reduced.
\end{proof}

\section{Comparison with Picard-Vessiot theory}

In this section we examine the Galois hull $\mathcal{L} | \mathcal{K}$ defined above in the case where $(L | K,R,H)$ is a finitely generated Picard-Vessiot extension of Artinian simple commutative $D$-module algebras such that $R$ is smooth over $K$ and compare the Umemura functor{} $\operatorname{Ume}(L | K)$ with the Galois group scheme $\operatorname{Gal}(L|K)$ of $L | K$ as defined by Amano and Masuoka.

\begin{notation}
	Let $C$ be a field, $G$ be a monoid and $D^1$ be a pointed irreducible cocommutative Hopf algebra of Birkhoff-Witt type over $C$ such that $D^1$ is a $C G$-module algebra.
	We define $D \coloneqq D^1 \# C G$.
	Let $(L | K, R, H)$ be a Picard-Vessiot extension of Artinian simple commutative $D$-module algebras such that $R | K$ is smooth.
	We further assume that there exists a matrix $X \in \GL_n(R)$ such that $R=K[X,X^{-1}]$ and $d(X)X^{-1} \in M_n(K)$ for all $d \in D$ (if $D$ is a Hopf algebra, then this is the case by \cite[Theorem 4.6]{AmanoMasuoka:2005}).
	Let $\Omega(K) = \{ P_1, \dots, P_m \}$ be the set of minimal prime ideals of $K$ and for each $i \in \{1, \dots, m\}$ let
	$Q_{i,1}, \dots, Q_{i,r_i}$ be the minimal prime ideals of $L$ lying over $P_i$.
	We define $K_i \coloneqq K/P_i$, $L_{i,j} \coloneqq L/{Q_{i,j}}$, $L_i \coloneqq \prod_{j=1}^{r_i} L/Q_{i,j}$ and $R_i \coloneqq \prod_{j=1}^{r_i} R/(Q_{i,j} \cap R)$.
	By proposition~\ref{prop:smoothness_of_Picard-Vessiot_extensions}, $L_i$ is separable over $K_i$ for all $i \in \{1, \dots, m \}$ and thus all $L_{i,j}$ $(j \in \{1, \dots, r_i\})$ are separable over $K_i$.
	We assume that the transcendence degree of $L_{i,j}$ over $K_i$ is the same for all $i \in \{1, \dots, m\}$ and all $j \in \{ 1, \dots, r_i \}$.
	Let ${\bm{u}}^{(i,j)} = (u^{(i,j)}_1, \dots, u^{(i,j)}_n)$ be a separating transcendence basis for the extension $L_{i,j}$ over $K_i$ and let $\theta_{{\bm{u}}}$ be the $n$-variate iterative derivation \eqref{eq:iterative_derivation_wrt_u} on $L$ over $K$.
	We denote by $\Psi$ the $D$-module algebra structure on $L$, by $ \rho_{} $ the associated homomorphism of $C$-algebras and by $k \coloneqq L^{ \rho_{} }$ the field of constants of $L$.
\end{notation}

\begin{lem}
	The subring of ${}_{C}\mkern-1mu\mathcal{M}(D,L)$ generated by $ \rho_{0} (L)$ and $ \rho_{} (L)$ is closed under the $n$-variate iterative derivation $\theta_{{\bm{u}}}$ and $ \rho_{0} (L)$ and $ \rho_{} (L)$ are linearly disjoint over the field of constants $k$;
	i.e. with the notation of subsection~\ref{ssec:extension_Lcal_Kcal} there is an isomorphism
	\begin{equation}
		\mathcal{L} =  \rho_{0} (L)[ \rho_{} (L)] \cong  \rho_{0} (L) \otimes_{k}  \rho_{} (L)
		 \label{eqn:lem:Lcal_is_generated_by_Kcal_and_gth_of_L_if_extension_is_Picard-Vessiot_isomorphism_for_Lcal}
	\end{equation}
	of $D$-module algebras.
	Similarly, $ \rho_{0} (L)[ \rho_{} (R)]$ is closed under $\theta_{{\bm{u}}}$ and $ \rho_{0} (L)$ and $ \rho_{} (R)$ are linearly disjoint over $k$, i.e. there is an isomorphism of $D$-module algebras
	\begin{equation}
		  \rho_{0} (L)[ \rho_{} (R)] \cong  \rho_{0} (L) \otimes_k  \rho_{} (R).
		 \label{eqn:lem:Lcal_is_generated_by_Kcal_and_gth_of_L_if_extension_is_Picard-Vessiot_isomorphism_for_PV_ring}
	\end{equation}
	\label{lem:Lcal_is_generated_by_Kcal_and_gth_of_L_if_extension_is_Picard-Vessiot}
\end{lem}
\begin{proof}
	The element $Z \coloneqq  \rho_{} (X)  \rho_{0} (X)^{-1}$ lies in $\GL_n({}_{C}\mkern-1mu\mathcal{M}(D,K))$ and so
	\[  \rho_{0} (L)[ \rho_{} (R)]
		= \mathcal{K}[ \rho_{} (R)]
		= \mathcal{K}[ \rho_{} (X),  \rho_{} (X)^{-1}]
		= \mathcal{K}[ Z, Z^{-1} ]
	\]
	is closed under $\theta_{{\bm{u}}}$.
	Since $ \rho_{} (\mkern-0.5mu L \mkern-0.5mu) \mkern-2mu = \mkern-2mu \TotalQuot(  \rho_{} \mkern-1mu(\mkern-0.5mu R\mkern-0.5mu)\mkern-0.5mu )$, for every non-zero divisor $a \in R$ the element $\theta_{{\bm{u}}}( \rho_{} (a))$ is invertible in $ \rho_{0} (L)[ \rho_{} (L)]\llbracket {\bm{w}} \rrbracket$ by \cite[Lemma 9.2.3]{Sweedler:1969}.
	Thus, $ \rho_{0} (L)[ \rho_{} (L)]$ is also closed under $\theta_{{\bm{u}}}$.

	It follows from lemma~\ref{lem:constants_of_simple_module_algebras_are_fields_and_linear_disjointness} that $ \rho_{} (L)$ and $ \rho_{0} (L)$ are linearly disjoint over $k$ and thus that $\mathcal{L}$ and $ \rho_{0} (L) \otimes_k  \rho_{} (L)$ are isomorphic as $D$-module algebras.
	Since $ \rho_{} (R)$ is a subalgebra of $ \rho_{} (L)$, the algebras $ \rho_{} (R)$ and $ \rho_{0} (L)$ are also linearly disjoint over $k$ and we obtain the isomorphism~\eqref{eqn:lem:Lcal_is_generated_by_Kcal_and_gth_of_L_if_extension_is_Picard-Vessiot_isomorphism_for_PV_ring}.
\end{proof}

\begin{lem}
	We assume that the field of constants $k$ is perfect.
	Then there exists a finite étale extension $K'$ of $K$, a matrix $B \in \GL_{n}(K')$ and a right $R \otimes_K K'$-linear automorphism $\gamma$
	of the $D$-module algebra $(R \otimes_k R \otimes_K K',  \rho_{}  \otimes  \rho_{0}  \otimes  \rho_{0} )$, defined by
	\begin{equation}
		\gamma(X \otimes 1 \otimes 1) \coloneqq (X \otimes 1 \otimes 1) (1 \otimes X^{-1} \otimes 1) (1 \otimes 1 \otimes B).
	\label{eqn_lem_helping_isomorphism_for_decomposition_in_PV_comparision_def_gamma}
	\end{equation}
	\label{lem_helping_isomorphism_for_decomposition_in_PV_comparision}
\end{lem}
\begin{proof}
	Let $\eta \colon (R,  \rho_{R} ) \to (R \otimes_k H,  \rho_{R}  \otimes  \rho_{0} )$ be the homomorphism of $D$-module algebras defined by $\eta(a) \coloneqq \mu^{-1}(1 \otimes a)$ for all $a \in R$, where $\mu \colon R \otimes_{k} H \to R \otimes_{K} R$ is the isomorphism of $D$-module algebras \eqref{eqn:prop:fundamental_isomorphism_of_Picard-Vessiot_theory}.
	Then $\eta$ fulfills
	\begin{equation}
		\eta(X) = (X \otimes (1 \otimes 1)) (1 \otimes (X^{-1} \otimes 1)(1 \otimes X)).
		\label{eqn_lem_helping_isomorphism_for_decomposition_in_PV_comparision_theta}
	\end{equation}
	By proposition~\ref{prop:smoothness_of_Picard-Vessiot_extensions}, $\operatorname{Gal}(L_i|K_i)$ is smooth over $k$ and by proposition~\ref{prop:quotients_of_PV_extensions}, $L_i | K_i$ is a Picard-Vessiot extension for all $i \in \{ 1, \dots, m \}$.
	Since the $K_{i}$-algebra $R_{i}$ is finitely generated, $\operatorname{Spec}R_{i}$ has a point in the algebraic closure of $K_{i}$ by Hilbert's Nullstellensatz.
	Since $\operatorname{Spec}R_{i}$ is a principal homogeneous space for the smooth affine group scheme $\operatorname{Gal}(L_{i}|K_{i})$, it has in fact a point in a finite separable field extension of $K_{i}$ (cf. \cite[18.5]{Waterhouse:1979}).
	Together they provide a point $\nu \colon R \to K'$ of $\operatorname{Spec}R = \operatorname{Spec}\prod_{i=1}^m R_i$ in a finite étale extension $K'$ of $K$.
	We extend the composition
	\begin{align*}
		\xymatrixcolsep{2.5em}
		\xymatrix{
			R \ar[r]^-{\eta}
				& R \mkern-2mu \otimes_k \mkern-2mu H \ar@{^{(}->}[r]
				& R \mkern-2mu \otimes_k \mkern-2mu R \mkern-1mu \otimes_K \mkern-2mu R \ar[rr]^{\id_{\mkern-2mu R} \mkern-2mu \otimes_k \mkern-2mu \id_{\mkern-2mu R} \mkern-2mu \otimes_K \mkern-2mu \nu }
				& & R \mkern-2mu \otimes_k \mkern-2mu R \mkern-2mu \otimes_K \mkern-2mu K'
		}
	\end{align*}
	right $R \otimes_K K'$-linearly to an endomorphism $\gamma$ of $R \otimes_k R \otimes_K K'$
	and we define $B \coloneqq \nu(X)$.
	Then the defining identity \eqref{eqn_lem_helping_isomorphism_for_decomposition_in_PV_comparision_def_gamma} for $\gamma$ follows from equation \eqref{eqn_lem_helping_isomorphism_for_decomposition_in_PV_comparision_theta} and clearly $\gamma$ is a homomorphism of $D$-module algebras.
The inverse of $\gamma$ is given by the right $R \otimes_K K'$-linear extension of
	\begin{align*}
		\xymatrixcolsep{2.55em}
		\xymatrix{
			R \ar[r]^-{\eta}
				& R \mkern-2mu \otimes_k \mkern-3mu H \ar[r]^-{\id_R \otimes_k S}
				& R \mkern-2mu \otimes_k \mkern-3mu H \ar@{^{(}->}[r]
				& R \mkern-2mu \otimes_k \mkern-3mu R \mkern-1mu \otimes_K \mkern-3mu R \ar[rr]^{\id_R \mkern-2mu \otimes_k \mkern-2mu \id_R \mkern-2mu \otimes_K \mkern-2mu \nu }
				& & R \mkern-2mu \otimes_k \mkern-3mu R \mkern-2mu \otimes_K \mkern-3mu K'
		}
	\end{align*}
	to an endomorphism of $R \otimes_k R \otimes_K K'$, where $S$ is the antipode of $H$.

\end{proof}

\begin{thm}
	\label{thm_comparison_between_InfGal_and_Gal}
	If the field of constants $k$ is perfect,
	then there exists a finite étale extension $L'$ of $L$ such that
	$\operatorname{Ume}(L | K) \times_L L'$ is isomorphic to the formal group scheme $\widehat{\operatorname{Gal}(L|K)}_{L'}$ associated to the base extension $\operatorname{Gal}(L|K)_{L'} = \operatorname{Gal}(L|K) \times_k L'$ of the Galois group scheme $\operatorname{Gal}(L|K)$.
\end{thm}
\begin{proof}
	Let $A$ be a commutative $L$-algebra.
	By remark~\ref{rmk_Galois_group_scheme_of_Picard-Vessiot_extensions}, $\operatorname{Gal}(L|K)(A)$ is isomorphic to the group $\Aut_D(R \otimes_{k} A | K \otimes_{k} A)$ of automorphisms of the $D$-module algebra $R \otimes_k A$ that leave $K \otimes_k A$ fixed.
	Thus $\widehat{\operatorname{Gal}(L|K)}(A)$ is isomorphic to the kernel
	\begin{equation}
	\label{eqn_thm_comparison_between_InfGal_and_Gal_kernel_RK}
		\Kernel \big( \Aut_D(R \otimes_{k} A | K \otimes_{k} A) \to \Aut_D(R \otimes_{k} A/N(A) | K \otimes_{k} A/N(A)) \big).
	\end{equation}
	The isomorphisms \eqref{eqn:lem:Lcal_is_generated_by_Kcal_and_gth_of_L_if_extension_is_Picard-Vessiot_isomorphism_for_Lcal} and \eqref{eqn:lem:Lcal_is_generated_by_Kcal_and_gth_of_L_if_extension_is_Picard-Vessiot_isomorphism_for_PV_ring} induce isomorphisms of algebras
	\begin{equation}
		\mathcal{L} \hat{\otimes}_L A\llbracket {\bm{w}}\rrbracket \cong (L \otimes_k A)\llbracket {\bm{w}}\rrbracket
		\label{eqn_lem_comparison_InfGalLK_and_InfGalRK_isom_Lcal}
	\end{equation}
	and
	\begin{equation}
		 \rho_{0} (L)[ \rho_{} (R)] \hat{\otimes}_L A\llbracket {\bm{w}}\rrbracket \cong (R \otimes_k A)\llbracket {\bm{w}}\rrbracket.
		\label{eqn_lem_comparison_InfGalLK_and_InfGalRK_isom_Rcal}
	\end{equation}
	Using these isomorphisms it is easy to see that $\operatorname{Ume}(L | K)(A)$ is isomorphic to the group of automorphisms of the $D \otimes_{C} D_{I \mkern-1.9mu D^{n}}$-module algebra $ \rho_{0} (L)[ \rho_{} (R)] \hat{\otimes}_L A\llbracket {\bm{w}}\rrbracket$ that leave $\mathcal{K} \mkern-1mu \hat{\otimes}_{\mkern-1mu L} \mkern-1mu A\llbracket \mkern-1mu {\bm{w}} \mkern-1mu \rrbracket$ fixed and are congruent to the identity modulo $ \rho_{0} ( \mkern-1mu L \mkern-1mu)[ \rho_{} ( \mkern-1muR)] \hat{\otimes}_{\mkern-1mu L} N(\mkern-1mu A) \mkern-1mu \llbracket \mkern-1mu {\bm{w}} \mkern-1mu \rrbracket$.
	We denote the latter by $\operatorname{Ume}(R | K)(A)$.

	By lemma~\ref{lem_helping_isomorphism_for_decomposition_in_PV_comparision}, there exists a finite étale extension $K'$ of $K$, a matrix $B \in \GL_{n}(K')$ and a right $R \otimes_K K'$-linear automorphism $\gamma$ of the $D$-module algebra $(R \otimes_k R \otimes_K K',  \rho_{}  \otimes  \rho_{0}  \otimes  \rho_{0} )$ defined by
	\[ \gamma(X \otimes 1 \otimes 1) \coloneqq (X \otimes 1 \otimes 1)(1 \otimes X^{-1} \otimes 1)(1 \otimes 1 \otimes B). \]
	There exists a finite étale extension $L'$ of $L$ containing $K'$ and $\gamma$ induces a left $K$-linear and right $L'$-linear automorphism $\bar{\gamma}$ of the $D$-module algebra $(R \otimes_k L',  \rho_{}  \otimes  \rho_{0} )$ defined by $\bar{\gamma}(X \otimes 1) \coloneqq (X \otimes 1)(1 \otimes X^{-1} B)$.
	The $n$-variate iterative derivation $\theta_{{\bm{u}}}$ extends uniquely from $L$ to $L'$ (cf. \cite[Theorem 27.2]{Matsumura:1989} or \cite[Proposition 1.2.2]{Heiderich:2010}) and we denote it again by $\theta_{{\bm{u}}}$.
	The ring $R \otimes_k L'$ is generated by $\bar{\gamma}(R \otimes_k 1)$ and $1 \otimes_k L'$, which are linearly disjoint over $k$ by lemma~\ref{lem:constants_of_simple_module_algebras_are_fields_and_linear_disjointness} (note that by proposition~\ref{prop:properties_of_principal_module_algebra_of_a_Picard-Vessiot_extension_of_AS_module_algebras}, the $D$-module algebra $R$ is simple); we have an isomorphism of $D$-module algebras
	\begin{align}
	\label{eq:RkL_linear_disjointness}
	R \otimes_k L'
		\mkern-2mu = \mkern-2mu \bar{\gamma}(R \otimes_k 1)[1 \otimes_k L']
		\mkern-2mu \cong \mkern-2mu \bar{\gamma}(R \otimes_k 1) \otimes_k L',
	\end{align}
	where the $D$-module algebra structure on $\bar{\gamma}(R \otimes_k 1) \otimes_k L'$ is $ \rho_{}  \otimes_k  \rho_{0}  \otimes_k  \rho_{0} $.
	The isomorphism \eqref{eqn:lem:Lcal_is_generated_by_Kcal_and_gth_of_L_if_extension_is_Picard-Vessiot_isomorphism_for_PV_ring} of $D$-module algebras extends $L'$-linearly to an isomorphism
	\begin{align}
		\xymatrix{
			R \otimes_k L' \ar[r]^-{ \rho_{}  \otimes_k  \rho_{0} }
			&  \rho_{} (R) \otimes_k  \rho_{0} (L') \ar[r]^-{m}
			&  \rho_{0} (L')[ \rho_{} (R)],
		}
		\label{eqn_thm_comparison_between_InfGal_and_Gal_decomposition1b}
	\end{align}
	where $R \otimes_k L'$ carries the $D$-module algebra structure $ \rho_{}  \otimes_k  \rho_{0} $ and $m$ is the restriction of the multiplication homomorphism of $ \rho_{0} (L')[ \rho_{} (R)]$.
	The image of $\bar{\gamma}(R \otimes_k 1)$ under this isomorphism in $ \rho_{0} (L')[ \rho_{} (R)]$ is $ \rho_{} (K)[Z, Z^{-1}]$ with $Z \coloneqq  \rho_{} (X) \rho_{0} (X)^{-1}  \rho_{0} (B)$ and the image of $1 \otimes_k L'$ under it is $ \rho_{0} (L')$.
	Thus, we obtain an isomorphism of $D \otimes_{C} D_{I \mkern-1.9mu D^{n}}$-module algebras
	\begin{align}
		\xymatrixcolsep{0.7em}
		\xymatrixrowsep{0.7em}
		\xymatrix{
		R \mkern-3mu \otimes_k \mkern-4mu L\mkern-2mu' \mkern-2mu \ar[rr]
		&
		& \mkern-2mu \bar{\gamma} \mkern-1mu (\mkern-2mu R \mkern-3mu \otimes_k \mkern-4mu 1 \mkern-2mu) \mkern-4mu \otimes_k \mkern-5mu L\mkern-2mu' \mkern-2mu \ar[rrrr]^-{m \mkern-1mu \circ \mkern-1mu (\mkern-1mu  \rho_{}  \mkern-1mu \otimes \mkern-1mu  \rho_{0}  \mkern-2mu) \otimes  \rho_{0}  }
		&
		&
		&
		& \mkern-2mu  \rho_{} \mkern-1mu (\mkern-2mu K \mkern-2mu) \mkern-2mu [\mkern-2mu Z \mkern-2mu, \mkern-2mu Z^{-1} \mkern-2mu] \mkern-4mu \otimes_k \mkern-3mu  \rho_{0}  \mkern-1mu (\mkern-2mu L\mkern-2mu' \mkern-2mu) \mkern-2mu \ar[r]^-{m}
		& \mkern-1mu  \rho_{0}  \mkern-1mu(\mkern-2mu L\mkern-2mu' \mkern-2mu) \mkern-2mu [ \rho_{}  \mkern-1mu( \mkern-2mu R \mkern-2mu) \mkern-2mu ],\\
		 \rho_{}  \mkern-1mu \otimes \mkern-1mu  \rho_{0}
			& &  \rho_{}  \mkern-1mu \otimes \mkern-1mu  \rho_{0}  \mkern-1mu \otimes \mkern-1mu  \rho_{0}
			& & & &  \rho_{int}  \mkern-1mu \otimes \mkern-1mu  \rho_{int}  &  \rho_{int}  \\
		\theta_{0} \mkern-1mu \otimes \mkern-1mu \theta_{{\bm{u}}}
			& & (\theta_{0} \mkern-3mu \otimes \mkern-3mu \theta_{0}) \mkern-3mu \otimes \mkern-3mu \theta_{{\bm{u}}}
			& & & & \theta_{{\bm{u}}}{} \mkern-3mu \otimes \mkern-3mu \theta_{{\bm{u}}}
			& \theta_{{\bm{u}}}
		}
		\label{eqn_thm_comparison_between_InfGal_and_Gal_decomposition2}
	\end{align}
	where the isomorphism at the left is \eqref{eq:RkL_linear_disjointness} and the $D$- and $D_{I \mkern-1.9mu D^{n}}$-module algebra structures are indicated in the two rows below the isomorphisms ($\theta_{0}$ denotes the trivial $n$-variate iterative derivation).
	For every commutative $L'$-algebra $A$, the isomorphism \eqref{eqn_thm_comparison_between_InfGal_and_Gal_decomposition2} gives rise to an isomorphism of $D \otimes_C D_{I \mkern-1.9mu D^{n}}$-module algebras
	\begin{align}
		\xymatrixcolsep{0.7em}
		\xymatrix{
			 \rho_{0} (\mkern-1mu L \mkern-1mu)[ \rho_{} (\mkern-1mu R\mkern-1mu)] \hat{\otimes}_L \mkern-0.5mu A\mkern-0.5mu\llbracket \mkern-1mu {\bm{w}} \mkern-1mu \rrbracket \mkern-3mu \ar[r]
				& \mkern-2mu  \rho_{0} (\mkern-1mu L\mkern-0.5mu' \mkern-1mu)[ \rho_{} (\mkern-1mu R \mkern-1mu)] \mkern-2mu \hat{\otimes}_{L\mkern-0.5mu'} \mkern-2mu A\mkern-0.5mu\llbracket \mkern-1mu {\bm{w}} \mkern-1mu \rrbracket \mkern-3mu \ar[r]
				& \mkern-2mu (\mkern-1mu R \mkern-2mu \otimes_k \mkern-3mu L\mkern-0.5mu' \mkern-1mu) \hat{\otimes}_{L\mkern-0.5mu'} \mkern-1.5mu A \mkern-0.5mu \llbracket \mkern-1mu {\bm{w}} \mkern-1mu \rrbracket \mkern-3mu \ar[r]
				& \mkern-2mu (\mkern-1mu R \mkern-2mu \otimes_k \mkern-3mu A \mkern-1mu) \mkern-1mu \llbracket \mkern-1mu {\bm{w}} \mkern-1mu \rrbracket,
		}
		\label{eqn_thm_comparison_between_InfGal_and_Gal_decomposition2b}
	\end{align}
	where on $(R \otimes_k A \mkern-1mu) \llbracket {\bm{w}} \rrbracket$ the $D$-module algebra structure is given by $ \rho_{}  \otimes_k  \rho_{0} $ on the coefficients with respect to ${\bm{w}}$ (as in \eqref{eqn_module_algebra_structure_on_images}) and the $D_{I \mkern-1.9mu D^{n}}$-module algebra structure is given by the $n$-variate iterative derivation $\theta_{{\bm{w}}}$ (cf. \ref{ex:iterative_derivation_with_respect_to_variables}).

	Given a $\varphi \mkern-2mu \in \mkern-2mu \operatorname{Ume}(R | K)(A)$,
	we obtain by composition with the vertical isomorphisms of $D \otimes_C D_{I \mkern-1.9mu D^{n}}$-module algebras, given by \eqref{eqn_thm_comparison_between_InfGal_and_Gal_decomposition2b}, in the diagram
	\begin{align*}
	\xymatrixrowsep{2.0pc}
	\xymatrixcolsep{2em}
	\xymatrix{
		 \rho_{0} (L)[ \rho_{} (R)] \hat{\otimes}_L A\llbracket {\bm{w}} \rrbracket \ar[r]^{\varphi} \ar[d]^{\sim}
			&  \rho_{0} (L)[ \rho_{} (R)] \hat{\otimes}_L A\llbracket {\bm{w}} \rrbracket \\
			(R \mkern-3mu \otimes_{k} \mkern-4mu A)\llbracket {\bm{w}} \rrbracket \ar[r]^{\sigma\llbracket {\bm{w}} \rrbracket}
			& (R \mkern-3mu \otimes_{k} \mkern-4mu A)\llbracket {\bm{w}} \rrbracket, \ar[u]_{\sim}
	}
	\end{align*}
	an automorphism of the $D \otimes_C D_{I \mkern-1.9mu D^{n}}$-module algebra $(R \otimes_k A)\llbracket {\bm{w}} \rrbracket$, which is of the form $\sigma\llbracket {\bm{w}} \rrbracket$, where $\sigma$ is an automorphism of the $D$-module algebra $R \otimes_k A$ of constants of the iterative derivation $\theta_{{\bm{w}}}$ on $(R \otimes_k A)\llbracket {\bm{w}} \rrbracket$.
	Then $\sigma$ is an element of the kernel~\eqref{eqn_thm_comparison_between_InfGal_and_Gal_kernel_RK} and this yields an isomorphism of groups between $\operatorname{Ume}(R | K)(A)$ and \eqref{eqn_thm_comparison_between_InfGal_and_Gal_kernel_RK}.
\end{proof}

\begin{cor}
	Under the assumptions of theorem~\ref{thm_comparison_between_InfGal_and_Gal} there exists a finite étale extension $L'$ of $L$ and an isomorphism
	\[ \operatorname{Ume}(L | K)(L'[\varepsilon]/(\varepsilon^2)) \cong \operatorname{Lie}(\operatorname{Gal}(L|K)) \otimes_{k} L'. \]
	\label{cor_comparision_of_Lie_algebra_for_PV_extensions}
\end{cor}
\begin{proof}
	This follows immediately from theorem~\ref{thm_comparison_between_InfGal_and_Gal} by taking $A = L'[\varepsilon]/(\varepsilon^2)$.
\end{proof}

In the case where $D = D_{end}$, the statement of corollary~\ref{cor_comparision_of_Lie_algebra_for_PV_extensions} is similar to the one of \cite[Theorem 3.3]{Morikawa:2009} and to \cite{Umemura:2010}.
Taking $D = D_{der}$, it provides a similar result as \cite[Theorem 5.15]{Umemura:1996b} in the case of finitely generated Picard-Vessiot extensions of differential fields in characteristic zero.

\begin{ex}
Let $L | K$ be the extension of example~\ref{ex:function_with_constant_derivation} (resp. example~\ref{ex:exponential_function}).
It is a Picard-Vessiot extension, in theorem~\ref{thm_comparison_between_InfGal_and_Gal} the extension $L'$ can be chosen to be $L$ and for every commutative $L$-algebra $A$,  the element $\varphi \in \operatorname{Ume}(L | K)(A)$ given by $\varphi( \rho_{} (y) \otimes 1) =  \rho_{} (y) \otimes 1 + 1 \otimes a_0$ for some $a_0 \in N(A)$ (resp. $\varphi( \rho_{} (y) \otimes 1) =  \rho_{} (y) \otimes (1+a_1)$ for some $a_1 \in N(A)$) corresponds under the isomorphism in theorem~\ref{thm_comparison_between_InfGal_and_Gal} to the automorphism $\sigma \in \widehat{\operatorname{Gal}(L|K)}(A)$ that fulfills $\sigma(y \otimes 1) = y \otimes 1 + 1 \otimes a_0$ (resp. $\sigma(y \otimes 1) = y \otimes (1+a_1)$).
\end{ex}

\begin{acknowledgements}
	I thank Hiroshi Umemura for helpful discussions on the topic of the article.
\end{acknowledgements}

\ifx \undefined \cprime \def \cprime {$\mathsurround=0pt '$}\fi
  \def\polhk#1{\setbox0=\hbox{#1}{\ooalign{\hidewidth
  \lower1.5ex\hbox{`}\hidewidth\crcr\unhbox0}}}


\begin{thebibliography}{MvdP03}

\bibitem[AM05]{AmanoMasuoka:2005}
Katsutoshi Amano and Akira Masuoka.
\newblock Picard-{V}essiot extensions of {A}rtinian simple module algebras.
\newblock {\em J. Algebra}, 285(2):743--767, 2005.

\bibitem[Ama05]{Amano:2005}
Katsutoshi Amano.
\newblock Relative invariants, difference equations, and the {P}icard-{V}essiot
  theory.
\newblock 2005, math/0503291.

\bibitem[Bou81]{Bourbaki:1981:Algebre:Ch4-7}
Nicolas Bourbaki.
\newblock {\em \'{E}l\'ements de math\'ematique}, volume 864 of {\em Lecture
  Notes in Mathematics}.
\newblock Masson, Paris, 1981.
\newblock Alg{\`e}bre. Chapitres 4 {\`a} 7. [Algebra. Chapters 4--7].

\bibitem[Bou85]{Bourbaki:1985:AlgebreCommutative:Ch1-4}
Nicolas Bourbaki.
\newblock {\em \'{E}l\'ements de math\'ematique}.
\newblock Masson, Paris, 1985.
\newblock Alg{\`e}bre commutative. Chapitres 1 {\`a} 4. [Commutative algebra.
  Chapters 1--4], Reprint.

\bibitem[BW03]{BrzezinskiWisbauer:2003}
Tomasz Brzezinski and Robert Wisbauer.
\newblock {\em Corings and comodules}, volume 309 of {\em London Mathematical
  Society Lecture Note Series}.
\newblock Cambridge University Press, Cambridge, 2003.

\bibitem[Cas06]{Casale:2006}
Guy Casale.
\newblock Enveloppe galoisienne d'une application rationnelle de
  {$\mathbb{P}\sp 1$}.
\newblock {\em Publ. Mat.}, 50(1):191--202, 2006.

\bibitem[Cas07]{Casale:2007}
Guy Casale.
\newblock The {G}alois groupoid of {P}icard-{P}ainlev\'e {VI} equation.
\newblock In {\em Algebraic, analytic and geometric aspects of complex
  differential equations and their deformations. {P}ainlev\'e hierarchies},
  RIMS K\^oky\^uroku Bessatsu, B2, pages 15--20. Res. Inst. Math. Sci. (RIMS),
  Kyoto, 2007.

\bibitem[Cas08]{Casale:2008}
Guy Casale.
\newblock Le groupo\"\i de de {G}alois de {$P\sb 1$} et son
  irr\'eductibilit\'e.
\newblock {\em Comment. Math. Helv.}, 83(3):471--519, 2008.

\bibitem[Gra09]{Granier:2009}
Anne Granier.
\newblock {\em Un D-groupo\"\i de de Galois pour les \'equations aux
  q-diff\'erences}.
\newblock PhD thesis, Institut de Math\'ematiques de Toulouse, Universit\'e
  Paul Sabatier, Toulouse, 2009.

\bibitem[Gro64]{EGA4-1_IHES:1964}
Alexander Grothendieck.
\newblock \'{E}l\'ements de g\'eom\'etrie alg\'ebrique. {IV}. \'{E}tude locale
  des sch\'emas et des morphismes de sch\'emas. {I}.
\newblock {\em Inst. Hautes \'Etudes Sci. Publ. Math.}, (20):259, 1964.

\bibitem[Gro67]{EGA4-4_IHES:1967}
Alexander Grothendieck.
\newblock \'{E}l\'ements de g\'eom\'etrie alg\'ebrique. {IV}. \'{E}tude locale
  des sch\'emas et des morphismes de sch\'emas {IV}.
\newblock {\em Inst. Hautes \'Etudes Sci. Publ. Math.}, (32):361, 1967.

\bibitem[Har10]{Hardouin:2010}
Charlotte Hardouin.
\newblock Iterative q-difference galois theory.
\newblock {\em J. Reine Angew. Math.}, 644:101--144, 2010.

\bibitem[Haz78]{Hazewinkel:1978}
Michiel Hazewinkel.
\newblock {\em Formal groups and applications}, volume~78 of {\em Pure and
  Applied Mathematics}.
\newblock Academic Press Inc. [Harcourt Brace Jovanovich Publishers], New York,
  1978.

\bibitem[Haz79]{Hazewinkel:1979}
Michiel Hazewinkel.
\newblock Infinite-dimensional universal formal group laws and formal
  {$A$}-modules.
\newblock In {\em Algebraic geometry ({P}roc. {S}ummer {M}eeting, {U}niv.
  {C}openhagen, {C}openhagen, 1978)}, volume 732 of {\em Lecture Notes in
  Math.}, pages 124--143. Springer, Berlin, 1979.

\bibitem[Hei07]{Heiderich:2007}
Florian Heiderich.
\newblock {P}icard-{V}essiot {T}heorie f\"{u}r lineare partielle
  {D}ifferentialgleichungen.
\newblock Diplomarbeit, Universit\"{a}t Heidelberg, Fakult\"{a}t f\"{u}r
  Mathematik und Informatik, 2007.

\bibitem[Hei10]{Heiderich:2010}
Florian Heiderich.
\newblock {\em Galois Theory of Module Fields}.
\newblock PhD thesis, Universitat de Barcelona, Departament d'\`{A}lgebra i
  Geometria, 2010.
\newblock Available from: \url{http://hdl.handle.net/10803/674}.

\bibitem[HS37]{HasseSchmidt:1937}
Helmut Hasse and Friedrich~Karl Schmidt.
\newblock {Noch eine Begr\"{u}ndung der Theorie des h\"{o}heren
  Differentialquotienten in einem algebraischen Funktionenk\"{o}rper in einer
  Unbestimmten}.
\newblock {\em J. Reine Angew. Math.}, 177:215--237, 1937.

\bibitem[Kei97]{Keigher:1997}
William~F. Keigher.
\newblock On the ring of {H}urwitz series.
\newblock {\em Comm. Algebra}, 25(6):1845--1859, 1997.

\bibitem[Mac71]{MacLane:1971}
Saunders MacLane.
\newblock {\em Categories for the working mathematician}.
\newblock Springer-Verlag, New York, 1971.
\newblock Graduate Texts in Mathematics, Vol. 5.

\bibitem[Mal01]{Malgrange:2001}
Bernard Malgrange.
\newblock Le groupo\"\i de de {G}alois d'un feuilletage.
\newblock In {\em Essays on geometry and related topics, {V}ol. 1, 2},
  volume~38 of {\em Monogr. Enseign. Math.}, pages 465--501, Geneva, 2001.
  Enseignement Math.

\bibitem[Mas10]{Masuoka:cross_bialgebra:preprint}
Akira Masuoka.
\newblock The $\times_r$-bialgebra associated with an iterative $q$-difference
  ring.
\newblock preprint, 2010.

\bibitem[Mat89]{Matsumura:1989}
Hideyuki Matsumura.
\newblock {\em Commutative Ring Theory}.
\newblock Cambridge University Press, Cambridge-New York, 1989.

\bibitem[Mau10]{Maurischat:2010b}
Andreas Maurischat.
\newblock Infinitesimal group schemes as iterative differential {G}alois
  groups.
\newblock {\em J. Pure Appl. Algebra}, 214(11):2092--2100, 2010.

\bibitem[Mon93]{Montgomery:1993}
Susan Montgomery.
\newblock {\em Hopf algebras and their actions on rings}, volume~82 of {\em
  CBMS Regional Conference Series in Mathematics}.
\newblock Published for the Conference Board of the Mathematical Sciences,
  Washington, DC, 1993.

\bibitem[Mor09]{Morikawa:2009}
Shuji Morikawa.
\newblock On a general difference {G}alois theory {I}.
\newblock {\em Ann. Inst. Fourier (Grenoble)}, 59(7):2709--2732, 2009.

\bibitem[MU09]{MorikawaUmemura:2009}
Shuji Morikawa and Hiroshi Umemura.
\newblock On a general difference {G}alois theory {II}.
\newblock {\em Ann. Inst. Fourier (Grenoble)}, 59(7):2733--2771, 2009.

\bibitem[MvdP03]{MatzatVanDerPut:2003}
Bernd~Heinrich Matzat and Marius van~der Put.
\newblock Iterative differential equations and the {A}bhyankar conjecture.
\newblock {\em J. Reine Angew. Math.}, 557:1--52, 2003.

\bibitem[Oku87]{Okugawa:1987}
K{\^o}taro Okugawa.
\newblock {\em Differential Algebra of Nonzero Characteristic}, volume~16 of
  {\em Lectures in Mathematics}.
\newblock Kinokuniya Company Ltd., Tokyo, 1987.

\bibitem[Swe69]{Sweedler:1969}
Moss~E. Sweedler.
\newblock {\em Hopf algebras}.
\newblock Mathematics Lecture Note Series. W. A. Benjamin, Inc., New York,
  1969.

\bibitem[Ume]{Umemura:2010}
Hiroshi Umemura.
\newblock {P}icard-{V}essiot theory in general {G}alois theory.
\newblock preprint.

\bibitem[Ume96]{Umemura:1996b}
Hiroshi Umemura.
\newblock Differential {G}alois theory of infinite dimension.
\newblock {\em Nagoya Math. J.}, 144:59--135, 1996.

\bibitem[Ume06]{Umemura:2006}
Hiroshi Umemura.
\newblock Galois theory and {P}ainlev\'e equations.
\newblock In {\em Th\'eories asymptotiques et \'equations de Painlev\'e},
  volume~14 of {\em S\'emin. Congr.}, pages 299--339, Paris, 2006. Soc. Math.
  France.

\bibitem[Wat79]{Waterhouse:1979}
William~C. Waterhouse.
\newblock {\em Introduction to Affine Group Schemes}, volume~66 of {\em
  Graduate Texts in Mathematics}.
\newblock Springer-Verlag, New York, 1979.

\end{thebibliography}
\end{document}